\documentclass[reqno]{amsart}

\usepackage{amsmath}
\usepackage{amsfonts}
\usepackage{pdfsync}
\usepackage{color}
\usepackage{graphicx}
\usepackage{amssymb}
\usepackage{epsf}
\usepackage{eucal}
\usepackage{amscd}
\usepackage{epsfig}
\usepackage{graphics}
\usepackage{amsthm}
\usepackage{fancyhdr}
\usepackage{delarray}
\usepackage{fullpage}
\usepackage{enumerate}
\usepackage{paralist}
\usepackage[all]{xy}
\usepackage{tikz}
\usepackage{verbatim}
\usepackage{cite}

\usepackage{hyperref}

\definecolor{darkgreen}{cmyk}{1,0,1,.2}
\definecolor{m}{rgb}{1,0.1,1}

\newdimen\theight
\def\TeXref#1{%
             \leavevmode\vadjust{\setbox0=\hbox{{\tt
                     \quad\quad  {\small \textrm #1}}}%
             \theight=\ht0
             \advance\theight by \lineskip
             \kern -\theight \vbox to
             \theight{\rightline{\rlap{\box0}}%
             \vss}%
             }}%

\DeclareMathOperator{\grad}{grad}

\newcommand{\sm}{\smallsetminus} 

\renewcommand{\epsilon}{\varepsilon}
\newcommand{\ol}{\overline}

\DeclareMathOperator{\Cl}{Cl}

\DeclareMathOperator{\diam}{diam}

\DeclareMathOperator{\dom}{dom}
\DeclareMathOperator{\im}{im}

\DeclareMathOperator{\id}{id}

\DeclareMathOperator{\Pen}{Pen}

\DeclareMathOperator{\inj}{inj}
\DeclareMathOperator{\Iso}{Iso}
\DeclareMathOperator{\Hess}{Hess}
\DeclareMathOperator{\Int}{Int}
\DeclareMathOperator{\ev}{ev}
\DeclareMathOperator{\GL}{GL}
\DeclareMathOperator{\Or}{O}

\DeclareMathOperator{\Met}{Met}

\DeclareMathOperator{\pr}{pr}

\newcommand{\FF}{\mathcal{F}}

\newcommand{\LL}{\mathcal{L}} 
\newcommand{\KK}{\mathcal{K}} 
\newcommand{\GG}{\mathcal{G}}
\newcommand{\CC}{\mathcal{C}}
\newcommand{\MM}{\mathcal{M}}
\newcommand{\RR}{\mathcal{R}}
\newcommand{\HH}{\mathcal{H}}
\newcommand{\VV}{\mathcal{V}}
\newcommand{\UU}{\mathcal{U}}

\newcommand{\PP}{\mathcal{P}}
\newcommand{\OO}{\mathcal{O}}
\newcommand{\TT}{\mathcal{T}}
\newcommand{\QQ}{\mathcal{Q}}

\newcommand{\DD}{\mathcal{D}}

\newcommand{\NN}{\mathcal{N}}
\newcommand{\ZZ}{\mathcal{Z}}

\newcommand{\SSS}{\mathcal{S}}

\newcommand{\Z}{\mathbb{Z}}

\newcommand{\R}{\mathbb{R}}
\newcommand{\N}{\mathbb{N}}

\newcommand{\bc}{\mathbf{c}}
\newcommand{\bd}{\mathbf{d}}
\newcommand{\bm}{\mathbf{m}}
\newcommand{\bv}{\mathbf{v}}

\theoremstyle{plain}

\newtheorem{thm}{Theorem}[section]
\newtheorem{lem}[thm]{Lemma}
\newtheorem{cor}[thm]{Corollary}
\newtheorem{prop}[thm]{Proposition}

\theoremstyle{definition}

\newtheorem{defn}[thm]{Definition}

\newtheorem{ex}[thm]{Example}

\theoremstyle{remark}

\newtheorem{rem}{Remark}

\newtheorem{claim}{Claim}

\title{A universal Riemannian foliated space}

\author[J.A. \'Alvarez L\'opez]{Jes\'us A. \'Alvarez L\'opez}
\address{Departamento de Xeometr\'{\i}a e Topolox\'{\i}a\\
         Facultade de Matem\'aticas\\
         Universidade de Santiago de Compostela\\
         Campus Vida\\
         15782 Santiago de Compostela\\
         Spain}
\email{jesus.alvarez@usc.es}
  
\author[R. Barral Lij\'o]{Ram\'on Barral Lij\'o}
\address{Departamento de Xeometr\'{\i}a e Topolox\'{\i}a\\
         Facultade de Matem\'aticas\\
         Universidade de Santiago de Compostela\\
         Campus Vida\\
         15782 Santiago de Compostela\\
         Spain}
\email{ramon.barral@usc.es}

\author[A. Candel]{Alberto Candel}
\address{Department of Mathematics\\
	CSUN\\
	Northridge, CA 91330\\
	USA}
\email{alberto.candel@csun.edu}

	
\date{}

\begin{document}

\maketitle

\begin{abstract}  
  It is proved that the isometry classes of pointed connected complete
  Riemannian $n$-manifolds form a Polish space, $\MM_*^\infty(n)$,
  with the topology described by the $C^\infty$ convergence of
  manifolds. This space has a canonical partition into sets defined by
  varying the distinguished point into each manifold. The locally
  non-periodic manifolds define an open dense subspace
  $\MM_{*,\text{\rm lnp}}^\infty(n)\subset\MM_*^\infty(n)$, which
  becomes a $C^\infty$ foliated space with the restriction of the
  canonical partition. Its leaves without holonomy form the subspace
  $\MM_{*,\text{\rm np}}^\infty(n)\subset\MM_{*,\text{\rm
      lnp}}^\infty(n)$ defined by the non-periodic manifolds. Moreover
  the leaves have a natural Riemannian structure so that
  $\MM_{*,\text{\rm lnp}}^\infty(n)$ becomes a Riemannian foliated
  space, which is universal among all sequential Riemannian foliated
  spaces satisfying certain property called
  covering-determination. $\MM_{*,\text{\rm lnp}}^\infty(n)$ is used to
  characterize the realization of complete connected Riemannian
  manifolds as dense leaves of covering-determined compact sequential
  Riemannian foliated spaces.
\end{abstract}  

\tableofcontents

\section{Introduction}\label{s: intro}

For any $n\in\N$ (we adopt the convention that $0\in \N$),
let $\MM_*(n)$ denote the set of isometry classes,
$[M,x]$, of pointed complete connected Riemannian $n$-manifolds,
$(M,x)$. The cardinality of each complete connected Riemannian $n$-manifold is
  less than or equal to the cardinality of the continuum, and
  therefore it may be assumed that its underlying set is contained in
  \(\R\). With this assumption, $\MM_*(n)$ is a well defined set. This set is only interesting for $n\ge2$ because
$\MM_*(0)=\{[\{0\},0]\}$ and $\MM_*(1)=\{[\R,0],[\mathbb{S}^1,1]\}$.
The set $\MM_*(n)$ can be considered as a subset of the Gromov space
$\MM_*$ of isometry classes of pointed proper metric spaces
\cite{Gromov1981}, \cite[Chapter~3]{Gromov1999}. However it is
interesting to consider a finer topology on $\MM_*(n)$, taking the
differentiable structure into account. For that purpose, the following
notion of $C^\infty$ convergence was defined on $\MM_*(n)$.
  
  \begin{defn}[{See e.g.\  \cite[Chapter~10, Section~3.2]{Petersen1998}}]\label{d: C^infty convergence in MM_*(n)}
    For each $m\in\N$, a sequence $[M_i,x_i]\in\MM_*(n)$ is said to be
    \emph{$C^m$ convergent} to $[M,x]\in\MM_*(n)$ if, for each compact
    domain $\Omega\subset M$ containing $x$, there are
    pointed $C^{m+1}$ embeddings $\phi_i:(\Omega,x)\to(M_i,x_i)$ for
    large enough $i$ such that $\phi_i^*g_i\to g|_\Omega$ as
    $i\to\infty$ with respect to the $C^m$ topology
    \cite[Chapter~2]{Hirsch1976}. If $[M_i,x_i]$ is $C^m$ convergent
    to $[M,x]$ for all $m$, then it is said that $[M_i,x_i]$ is
    \emph{$C^\infty$ convergent} to $[M,x]$.
\end{defn}

Here, a \emph{domain} in $M$ is a connected
      $C^\infty$ submanifold, possibly with boundary, of the same
      dimension as $M$. 

It is admitted that $C^\infty$ convergence defines a topology on
$\MM_*(n)$ \cite{Petersen1997}. However we are not aware of any proof in
the literature showing that it satisfies the conditions to describe a
topology \cite{Koutnik1985}, \cite{GutierresHofmann2007} (see also
\cite{Kelley1950} and \cite{Kelley1975} if $C^\infty$ convergence were
defined with nets or filters). This is only proved on subspaces
defined by manifolds of equi-bounded geometry, where the $C^\infty$
convergence coincides with convergence in $\MM_*$ \cite{Lessa:Reeb}
(see also \cite[Chapter~10]{Petersen1998}). The first main theorem of
the paper is the following.

\begin{thm}\label{t: C^infty convergence in MM_*(n)} 
  The $C^\infty$ convergence in $\MM_*(n)$ describes a
  Polish topology.
\end{thm}

Recall that a space is called \emph{Polish} if it is separable and completely metrizable.

The topology given by Theorem~\ref{t: C^infty convergence in MM_*(n)}
will be called the \emph{$C^\infty$ topology} on $\MM_*(n)$, and the
corresponding space is denoted by $\MM_*^\infty(n)$.

For each complete connected Riemannian $n$-manifold $M$, there is a
canonical continuous map $\iota:M\to\MM_*^\infty(n)$ given by
$\iota(x)=[M,x]$, which induces a continuous injective map
$\bar{\iota}:\Iso(M)\backslash M\to\MM_*^\infty(n)$, where $\Iso(M)$
denotes the isometry group of $M$. The more explicit notation
$\iota_M$ and $\bar\iota_M$ may be also used. The images of the maps
$\iota_M$ form a natural partition of $\MM_*^\infty(n)$, denoted by
$\FF_*(n)$.

A Riemannian manifold, $M$, is said to be \emph{non-periodic} if
$\Iso(M)=\{\id_M\}$, and is said to be \emph{locally non-periodic} if
each point $x\in M$ has a neighborhood $U_x$ such that
\[
\{\,h\in\Iso(M)\mid h(x)\in U_x\}=\{\id_M\}\;.
\]
Let $\MM_{*,\text{\rm np}}(n)$ and $\MM_{*,\text{\rm lnp}}(n)$ be the
$\FF_*(n)$-saturated subsets of $\MM_*(n)$ defined by non-periodic and
locally non-periodic manifolds, respectively. The notation
$\MM_{*,\text{\rm np}}^\infty(n)$ and $\MM_{*,\text{\rm
    lnp}}^\infty(n)$ is used when these sets are equipped with the
restriction of the $C^\infty$ topology. The restrictions of $\FF_*(n)$
to $\MM_{*,\text{\rm np}}(n)$ and $\MM_{*,\text{\rm lnp}}(n)$ are
respectively denoted by $\FF_{*,\text{\rm np}}(n)$ and
$\FF_{*,\text{\rm lnp}}(n)$. Note that $\MM_{*,\text{\rm np}}(0)=\{[\{0\},0]\}$ 
and $\MM_{*,\text{\rm lnp}}(1)=\emptyset$.

On the other hand, let $\MM_{*,\text{\rm c}}^\infty(n)$ (respectively, $\widehat\MM_{*,\text{\rm o}}^\infty(n)$) be the $\FF_*(n)$-saturated subspace of $\widehat\MM_*(n)$ consisting of classes $[M,x]$ such that $M$ is compact (respectively, open). Observe that, if $[N,y]$ is close enough to any $[M,x]\in\MM_{*,\text{\rm c}}^\infty(n)$, then $N$ is diffeomorphic to $M$. Thus $\MM_{*,\text{\rm c}}^\infty(n)$ is open in $\MM_*(n)$, and therefore $\MM_{*,\text{\rm o}}^\infty(n)$ is closed. Hence these are Polish subspaces of $\MM_*(n)$, as well as their intersections with any Polish subspace. The intersection of $\MM_{*,\text{\rm c/o}}^\infty(n)$ and $\MM_{*,\text{\rm (l)np}}^\infty(n)$ is denoted by $\MM_{*,\text{\rm (l)np,c/o}}^\infty(n)$. The restrictions of $\FF_*(n)$ to $\MM_{*,\text{\rm c/o}}(n)$ and $\MM_{*,\text{\rm (l)np,c/o}}(n)$ are denoted by $\FF_{*,\text{\rm c/o}}(n)$ and $\FF_{*,\text{\rm (l)np,c/o}}(n)$, respectively. The second main theorem of the paper is the following.

\begin{thm}\label{t: FF_*,lnp(n)}
  The following properties hold for $n\ge2$:
  \begin{enumerate}[{\rm (}i{\rm )}]

  \item\label{i: MM_*,lnp(n) is open and dense} $\MM_{*,\text{\rm
        lnp}}(n)$ is Polish and dense in $\MM_*^\infty(n)$.

  \item\label{i: FF_*,lnp(n) is foliated structure} $\MM_{*,\text{\rm
        lnp}}^\infty(n)\equiv(\MM_{*,\text{\rm
        lnp}}^\infty(n),\FF_{*,\text{\rm lnp}}(n))$ is a
    foliated space of dimension $n$.
    
  \item\label{i: FF_*,lnp,o(n) is transitive} $\FF_{*,\text{\rm lnp,o}}(n)$ is transitive.

  \item\label{i: FF_*,lnp(n) is C^infty and Riemannian} The foliated
    space $\MM_{*,\text{\rm lnp}}^\infty(n)$ has canonical $C^\infty$
    and Riemannian structures such that $\bar{\iota}:\Iso(M)\backslash
    M\to\iota(M)$ is an isometry for every locally non-periodic,
    complete, connected Riemannian manifold $M$.

  \item\label{i: MM_*,np(n)} For any locally non-periodic complete
    connected Riemannian manifold $M$, the quotient map
    $M\to\Iso(M)\backslash M$ corresponds to the holonomy covering of
    the leaf $\iota(M)$ by $\bar{\iota}:\Iso(M)\backslash
    M\to\iota(M)$. In particular, the set $\MM_{*,\text{\rm np}}(n)$
    is the union of leaves of $\MM_{*,\text{\rm lnp}}^\infty(n)$ with
    trivial holonomy groups.

  \end{enumerate}
\end{thm}

The following result states a universal property of $\MM_{*,\text{\rm lnp}}^\infty(n)$, which involves certain property called covering-determination (Definition~\ref{d: covering-determined}).

\begin{thm}\label{t: covering-determined} 
	Let $X$ be a sequential Riemannian foliated space of dimension $n\ge2$ whose leaves are complete. Then $X$ is isometric to a saturated subspace of $\MM_{*,\text{\rm lnp}}^\infty(n)$ if and only if it is covering-determined.
\end{thm}

Recall that a space $X$ is called \emph{sequential} if a subset $A\subset X$ is open whenever each convergent sequence $x_n\to x\in A$ in $X$ eventually belongs to $A$. For instance, first countable spaces are sequential. This condition could be removed by using convergence of nets or filters instead of sequences.

$\MM_{*,\text{\rm lnp}}^\infty(n)$ is used to prove the following result about realizations of Riemannian manifolds as leaves. It involves the obvious Riemannian versions of the conditions of being aperiodic or repetitive, which are standard for tilings or graphs (see e.g.\ \cite{GrunbaumShephard1987,Robinson2004,Frank2008}), and a weak version of aperiodicity  (Definitions~\ref{d: aperiodic} and~\ref{d: repetitive}).

\begin{thm}\label{t: M is isometric to a dense leaf}
  The following properties hold for a complete connected
  Riemannian manifold $M$ of bounded geometry and dimension $n\ge2$:
  \begin{enumerate}[{\rm(}i{\rm)}]
	
  \item\label{i: M is isometric to a dense leaf} $M$ is non-periodic
    and has a {\rm(}repetitive{\rm)} weakly aperiodic connected
    covering if and only if it is isometric to a dense leaf of a
    {\rm(}minimal{\rm)} covering-determined compact sequential
    Riemannian foliated space.
		
  \item\label{i: if M is aperiodic} If $M$ is aperiodic {\rm(}and
    repetitive{\rm)}, then it is isometric to a dense leaf of a
    {\rm(}minimal{\rm)} covering-determined compact sequential
    Riemannian foliated space whose leaves have trivial holonomy
    groups.
	
  \end{enumerate}
\end{thm}

\section{Preliminaries}

\subsection{Foliated spaces}\label{ss: fol sps}

Standard references for foliated spaces are \cite{MooreSchochet1988},
\cite[Chapter~11]{CandelConlon2000-I},
\cite[Part~1]{CandelConlon2003-II} and \cite{Ghys2000}.

Let $Z$ be a space and let $U$ be an open set in $\R^n\times Z$
($n\in\N$), with coordinates $(x,z)$. For $m\in\N$, a map $f:U\to
\R^p$ ($p\in\N$) is of \emph{class $C^m$} if its partial derivatives
up to order $m$ with respect to $x$ exist and are continuous on
$U$. If $f$ is of class $C^m$ for all $m$, then it is called of
\emph{class $C^\infty$}. Let $Z'$ be another space, and let
$h:U\to\R^p\times Z'$ ($p\in\N$) be a map of the form
$h(x,z)=(h_1(x,z),h_2(z))$, for maps $h_1:U\to\R^p$ and
$h_2:\pr_2(U)\to Z'$, where $\pr_2:\R^n\times Z\to Z$ is the second
factor projection. It will be said that $h$ is of \emph{class $C^m$}
if $h_1$ is of class $C^m$ and $h_2$ is continuous.

For $m\in\N\cup\{\infty\}$ and $n\in\N$, a \emph{foliated structure}
$\FF$ of \emph{class $C^m$} and \emph{dimension} $\dim\FF=n$ on a
space $X$ is defined by a collection $\UU=\{(U_i,\phi_i)\}$, where
$\{U_i\}$ is an open covering of $X$, and each $\phi_i$ is a
homeomorphism $U_i\to B_i\times Z_i$, for a locally compact Polish
space $Z_i$ and an open ball $B_i$ in $\R^n$, such that the coordinate
changes $\phi_j\phi_i^{-1}:\phi_i(U_i\cap U_j)\to\phi_j(U_i\cap U_j)$
are locally $C^m$ maps of the form
\[
\phi_j\phi_i^{-1}(x,z) = (g_{ij}(x,z),h_{ij}(z))\;.
\]
These maps $h_{ij}$ will be called the \emph{local transverse
  components} of the changes of coordinates. Each $(U_i,\phi_i)$ is
called a \emph{foliated chart}, the sets $\phi_i^{-1}(B_i\times
\{z\})$ ($z\in Z_i$) are called \emph{plaques}, and the collection
$\UU$ is called a \emph{foliated atlas} of \emph{class $C^m$}. Two
$C^m$ foliated atlases on $X$ define the same $C^m$ foliated structure
if their union is a $C^m$ foliated atlas. If we consider foliated
atlases so that the sets $Z_i$ are open in some fixed space, then
$\FF$ can be also described as a maximal foliated atlas of class
$C^m$. The term \emph{foliated space} (of \emph{class $C^m$}) is used
for $X\equiv(X,\FF)$. If no reference to the class $C^m$ is indicated,
then it is understood that $X$ is a $C^0$ (or \emph{topological})
foliated space. The concept of $C^m$ foliated space can be extended to
the case \emph{with boundary} in the obvious way, and the boundary of
a $C^m$ foliated space is a $C^m$ foliated space without boundary.

The foliated structure of a space $X$ induces a locally Euclidean
topology on $X$, the basic open sets being the plaques of all foliated
charts, which is finer than the original topology. The connected
components of $X$ in this topology are called \emph{leaves}. Each leaf
is a connected $C^m$ $n$-manifold with the differential structure
canonically induced by $\FF$. The leaf that contains each point $x\in
X$ is denoted by $L_x$. The leaves of $\FF$ form a partition of $X$
that determines the topological foliated structure. The corresponding
quotient space, called \emph{leaf space}, is denoted by $X/\FF$.

The restriction of $\FF$ to some open subset $U\subset X$ is the
foliated structure $\FF|_U$ on $U$ defined by the charts of $\FF$
whose domains are contained in $U$. More generally, a subspace
$Y\subset X$ is a \emph{$C^m$ foliated subspace} when it is a subspace
with a $C^m$ foliated structure $\GG$ so that, for each $y\in Y$,
there is a foliated chart of $\FF$ defined on a neighborhood $U$ of
$y$ in $X$, whose restriction to $U\cap Y$ can be considered as a
chart of $\GG$ in the obvious way; in particular,
$\dim\GG\le\dim\FF$. For instance, any saturated subspace is a $C^m$
foliated subspace.

A map between foliated spaces is said to be a \emph{foliated map} if
it maps leaves to leaves. A foliated map between $C^m$ foliated spaces
is said to be of \emph{class $C^m$} if its local representations in
terms of foliated charts are of class $C^m$. A \emph{$C^m$ foliated
  diffeomorphism} between $C^m$ foliated spaces is a $C^m$ foliated
homeomorphism between them whose inverse is also a $C^m$ foliated map.

Any topological space is a foliated space whose leaves are its
points. On the other hand, any connected $C^m$ $n$-manifold $M$ is a
$C^m$ foliated space of dimension $n$ with only one leaf. The $C^m$
foliated maps $M\to X$ can be considered as $C^m$ maps to the leaves
of $X$, and may be also called \emph{$C^m$ leafwise maps}. They form a
set denoted by $C^m(M,\FF)$, which can be equipped with the obvious
generalization of the (\emph{weak}) \emph{$C^m$ topology}. In
particular, for $m=0$, we get the subspace $C(M,\FF)\subset C(M,X)$
with the compact-open topology. For instance, $C(I,\FF)$ ($I=[0,1]$)
is the space of leafwise paths in $X$.

Many concepts of manifold theory readily extend to foliated spaces. In
particular, if $\FF$ is of class $C^m$ with $m\ge1$, there is a vector
bundle $T\FF$ over $X$ whose fiber at each point $x\in X$ is the
tangent space $T_xL_x$. Observe that $T\FF$ is a foliated space of
class $C^{m-1}$ with leaves $TL$ for leaves $L$ of $X$. Then we can
consider a $C^{m-1}$ Riemannian structure on $T\FF$, which is called a (\emph{leafwise})
\emph{Riemannian metric} on $X$. This is a
  section of the associated bundle over $X$ of positive definite
  symmetric bilinear forms on the fibers of $T\FF$, which is $C^{m-1}$
  as foliated map. In this paper, a \emph{Riemannian
  foliated space} is a $C^\infty$ foliated space equipped with a
$C^\infty$ Riemannian metric, and an \emph{isometry} between
Riemannian foliated spaces is a $C^\infty$ diffeomorphism between them
whose restrictions to the leaves are isometries; in this case, the
Riemannian foliated spaces are called \emph{isomertric}.

A foliated space has a ``transverse dynamics,'' which can be described
by using a pseudogroup (see \cite{Haefliger1985,Haefliger1988,Haefliger2002}). A \emph{pseudogroup} $\HH$ on a space is a
maximal collection of homeomorphisms between open subsets of $Z$ that
contains $\id_Z$, and is closed by the operations of composition,
inversion, restriction to open subsets of their domains, and
combination. This is a generalization of a dynamical system, and all
basic dynamical concepts can be directly generalized to
pseudogroups. For instance, we can consider its \emph{orbits}, and the
corresponding orbit space is denoted by $Z/\HH$. It is said that $\HH$
is \emph{generated} by a subset $E$ when all of its elements can be
obtained from the elements of $E$ by using the pseudogroup
operations. Certain \emph{equivalence} relation between pseudogroups
was introduced \cite{Haefliger1985}, \cite{Haefliger1988}, and
equivalent pseudogroups should be considered to represent the same
dynamics; in particular, they have homeomorphic orbit spaces.

The \emph{germ groupoid} of $\HH$ is the topological groupoid of germs
of maps in $\HH$ at all points of their domains, with the operation
induced by the composite of partial maps and the \'etale topology. Its
subspace of units can be canonically identified with $Z$. For each
$x\in Z$, the group of elements of this groupoid whose source and
range is $x$ is called the \emph{germ group} of $\HH$ at $x$. The germ
groups at points in the same orbit are conjugated in the germ
groupoid, and therefore the \emph{germ group} of each orbit is defined
up to isomorphisms. Under pseudogroup equivalences, corresponding
orbits have isomorphic germ groups.

Let $\UU=\{U_i,\phi_i\}$ be a foliated atlas of $\FF$, with
$\phi_i:U_i\to B_i\times Z_i$, and let $p_i=\pr_2\phi_i:U_i\to
Z_i$. The local transverse components of the corresponding changes of
coordinates can be considered as homeomorphisms between open subsets
of $Z=\bigsqcup_iZ_i$, which generate a pseudogroup $\HH$. The
equivalence class of $\HH$ depends only on $\FF$, and is called its
\emph{holonomy pseudogroup}. There is a canonical homeomorphism between
the leaf space and the orbit space, $X/\FF\to Z/\HH$, given by
$L\mapsto\HH(p_i(x))$ if $x\in L\cap U_i$.

The \emph{holonomy groups} of the leaves are the germ groups of the
corresponding $\HH$-orbits. The leaves with trivial holonomy groups
are called \emph{leaves without holonomy}. The union of leaves without
holonomy is denoted by $X_0$. If $X$ is second
countable, then $X_0$ is a dense $G_\delta$ saturated subset of $X$
\cite{Hector1977a,EpsteinMillettTischler1977}.

Given a loop $\alpha$ in a leaf $L$ with base point $x$, there is a
partition $0=t_0<t_1<\dots<t_k=1$ of $I$ and there are foliated charts
$(U_{i_1},\phi_{i_1}),\dots,(U_{i_k},\phi_{i_k})$ such that
$\alpha([t_{l-1},t_l])\subset U_{i_l}$ for $l\in\{1,\dots,k\}$. We can
assume $(U_{i_k},\phi_{i_k})=(U_{i_1},\phi_{i_1})$ because $\alpha$ is
a loop. Let $h_{i_{l-1},i_l}$ be the local transverse component of
each change of coordinates $\phi_{i_l}\phi_{i_{l-1}}^{-1}$ defined
around $p_{i_{l-1}}c(t_{l-1})$ and with
$h_{i_{l-1},i_l}p_{i_{l-1}}\alpha(t_{l-1})=p_{i_l}\alpha(t_l)$. The
germ the composition $h_{i_{k-1},i_k}\cdots h_{i_1,i_0}$ at
$p_{i_0}(x)=p_{i_k}(x)$ depends only on $\FF$ and the class of
$\alpha$ in $\pi_1(L,x)$, obtaining a surjective homomorphism of
$\pi_1(L,x)$ to the holonomy group of $L$. This homomorphism defines a
connected covering $\widetilde L^{\text{\rm hol}}$ of $L$, which is
called its \emph{holonomy covering}.

Now, let $R$ be an equivalence relation on a topological space $X$. A
subset of $X$ is called ($R$-) \emph{saturated} if it is a union of
($R$-) equivalence classes. The equivalence relation $R$ is said to be
(\emph{topologically})\emph{transitive} if there is an equivalence
class that is dense in \(X\). A subset $Y\subset X$ is called an ($R$-)
\emph{minimal set} if it is a minimal element of the family of
nonempty saturated closed subsets of $X$ ordered by inclusion; this is
equivalent to the condition that all equivalence classes in $Y$ are
dense in $Y$. In particular, $X$ (or $R$) is called \emph{minimal}
when all equivalence classes are dense in $X$. These concepts apply to
foliated spaces with the equivalence relation whose equivalence
classes are the leaves.

\subsection{Riemannian geometry}\label{ss: Riemannian geom}

Let $M$ be a Riemannian manifold possibly with boundary or
corners (in the sense of \cite{Cerf1961},
  \cite{Douady1964}). Connectedness of Riemannian manifolds is not assumed in Sections~\ref{ss: Riemannian geom},~\ref{s: quasi-isometries} and~\ref{s: center of mass} because it is not relevant for the concepts of these sections, but this property is assumed in the rest of the paper: it is needed in Section~\ref{s: partial quasi-isometries}, and it is implicit in Sections~\ref{s: C^infty topology}--\ref{s: bundles} and~\ref{s: foliated structure}--\ref{s: saturated subspaces} because the manifolds are given by elements of $\MM_*(n)$. The following standard notation will be
used. The metric tensor is denoted by $g$, the distance function on
each of the connected components of \(M\) by $d$, the tangent bundle by
$\pi:TM\to M$, the $\GL(n)$-principal bundle of tangent frames by
$\pi:PM\to M$, the $\Or(n)$-principal bundle of orthonormal tangent
frames by $\pi:QM\to M$, the Levi-Civita connection by $\nabla$, the
curvature by $\RR$, the exponential map by $\exp:TM\to M$ (if $M$ is
complete and $\partial M=\emptyset$), the open and closed balls of
center $x\in M$ and radius $r>0$ by $B(x,r)$ and $\ol B(x,r)$,
respectively, and the injectivity radius by $\inj$ (if $\partial
M=\emptyset$). The \emph{penumbra} around a subset $S\subset M$ of
radius $r>0$ is the set $\Pen(S,r)=\{\,x\in M\mid d(x,S)<r\,\}$. If
needed, ``$M$'' will be added to all of the above notation as a
subindex or superindex. When a family of Riemannian manifolds $M_i$ is
considered, we may add the subindex or superindex ``$i$'' instead of
``$M_i$'' to the above notation. A covering of $M$ is assumed to be
equipped with the lift of $g$.

For $m\in\Z^+$, let $T^{(m)}M=T\cdots TM$ ($m$ times). We also set
$T^{(0)}M=M$. If $l<m$, $T^{(l)}M$ is sometimes identified with a
regular submanifold of $T^{(m)}M$ via zero sections, and therefore,
for each $x\in M$, the notation $x$ may be also used for the zero
elements of $T_xM$, $T_xTM$, etc. When the vector space structure of
$T_xM$ is emphasized, its zero element is denoted by $0_x$, or simply
by $0$, and the image of the zero section of $\pi:TM\to M$ is denoted
by $Z\subset TM$. Let $\pi\colon T^{(m)}M \to T^{(l)}M$ be the vector
bundle projection given by composing the tangent bundle projections;
in particular, we have $\pi:T^{(m)}M \to M$. Given any $C^m$ map
between Riemannian manifolds, $\phi:M\to N$, the induced map
$T^{(m)}M\to T^{(m)}N$ will be denoted by $\phi_*^{(m)}$ (or simply
$\phi_*$ if $m=1$, $\phi_{**}$ if $m=2$, and so on).

Banach manifolds are also considered in some parts of the paper, using
analogous notation.

The Levi-Civita connection determines a decomposition
$T^{(2)}M=\HH\oplus\VV$, as direct sum of the horizontal and vertical
subbundles. The \emph{Sasaki metric} on $TM$ is the unique Riemannian
metric $g^{(1)}$ so that $\HH\perp\VV$ and the canonical identities
$\HH_\xi\equiv T_\xi M \equiv \VV_\xi$ are isometries for every
$\xi\in TM$.

Continuing by induction, for $m\ge2$, the \emph{Sasaki metric} on
$T^{(m)}M$ is defined by $g^{(m)}=(g^{(m-1)})^{(1)}$. The notation
$d^{(m)}$ is used for the corresponding distance function on the
connected components, and the corresponding open and closed balls of
center $\xi\in T^{(m)}M$ and radius $r>0$ are denoted by
$B^{(m)}(\xi,r)$ and $\ol B^{(m)}(\xi,r)$, respectively. We may add
the subindex ``$M$'' to this notation if necessary, or the subindex
``$i$'' instead of ``$M_i$'' when a family of Riemannian manifolds
$M_i$ is considered. From now on, $T^{(m)}M$ is assumed to be equipped
with $g^{(m)}$.

\begin{rem}\label{r: T^(l)M subset T^(m)M}
  The following properties hold for $l<m$ and $\pi:T^{(m)}M\to
  T^{(l)}M$:
  \begin{enumerate}[(i)]

  \item\label{i: g^(m)|_T^(l)M = g^(l)} $g^{(m)}|_{T^{(l)}M}=g^{(l)}$.
			
  \item\label{i: T^(l)M subset T^(m)M} The submanifold
    $T^{(l)}M\subset T^{(m)}M$ is totally geodesic and orthogonal to
    the fibers of $\pi$. This follows easily by induction on $m$,
    where the case $m=1$ is proved in \cite[Corollary of
    Theorem~13]{Sasaki1958}.
			
  \item\label{i: pi} The projection $\pi$ is a Riemannian submersion
    with totally geodesic fibers. Again, this follows by induction on
    $m$, and the case $m=1$ is proved in \cite[Theorems~14
    and~18]{Sasaki1958}.
		
  \item\label{i: pi(xi)} For every $\xi\in T^{(m)}M$, its projection
    $\pi(\xi)$ is the only point $\zeta\in T^{(l)}M$ that satisfies
    $d^{(m)}(\xi,\zeta)=d^{(m)}(\xi,T^{(l)}M)$. To see this, it is
    enough to prove that $\pi(\xi)$ is the only critical point of the
    distance function $d^{(m)}(\cdot,\xi)$ on $T^{(l)}M$. These
    critical points are just the points $\zeta\in T^{(l)}M$ where the
    shortest $g^{(m)}$-geodesics $\gamma$ from $\zeta$ to $\xi$ are
    orthogonal to $T^{(l)}M$ at $\zeta$. Hence $\gamma$ is a geodesic
    in $\pi^{-1}(\zeta)$ by~\eqref{i: pi}, obtaining $\zeta=\pi(\xi)$.
			
  \item\label{i: zeta'} For all $\zeta,\zeta'\in T^{(l)}M$, the point
    $\zeta'$ is the only $\xi\in\pi^{-1}(\zeta')$ satisfying
    $d^{(m)}(\xi,\zeta)=d^{(m)}(\xi,\pi^{-1}(\zeta))$. This follows
    like~\eqref{i: pi(xi)}, using~\eqref{i: T^(l)M subset T^(m)M}
    instead of~\eqref{i: pi}.
			
  \end{enumerate}
\end{rem}

Let $(U;x^1,\dots,x^n)$ be a chart of $M$. The corresponding metric
coefficients are denoted by $g_{ij}$, and the Christoffel symbols of
the first and second kind are denoted by $\Gamma_{ijk}$ and
$\Gamma_{ij}^k$, respectively. Using  the
  Einstein notation, recall that
\begin{equation}\label{Gamma_ijk}
  \Gamma_{ij}^\alpha g_{\alpha k}=\Gamma_{ijk}
  =\frac{1}{2}(\partial_ig_{jk}+\partial_jg_{ik}-\partial_kg_{ij})\;.
\end{equation}
Identify the functions $x^i$, $g_{ij}$, $\Gamma_{ijk}$ and
$\Gamma_{ij}^k$ with their lifts to $TU$. We get a chart
$(U^{(1)};x_{(1)}^1,\dots,x_{(1)}^{2n})$ of $TM$ with $U^{(1)}=TU$,
$x_{(1)}^i=x^i$ and $x_{(1)}^{n+i}=v^i$ for $1\le i\le n$, where the
functions $v^i$ give the coordinates of tangent vectors with respect
to the local frame $(\partial_1,\dots,\partial_n)$ of $TU$ induced by
$(U;x^1,\dots,x^n)$. The coefficients of the Sasaki metric $g^{(1)}$
with respect to $(TU;x_{(1)}^1,\dots,x_{(1)}^{2n})$ are \cite[Eq.~(3.5)]{Sasaki1958}:
\begin{equation}\label{Sasaki}
  \left.
    \begin{aligned}
      g^{(1)}_{ij} &= g_{ij} - g_{\alpha
        \gamma}\Gamma_{\mu\beta}^\alpha
      \Gamma_{\alpha\nu}^\beta v^\mu v^\nu \\
      g^{(1)}_{n+i\,j} &=\Gamma_{j\mu i}v^\mu \\
      g^{(1)}_{n+i\,n+j} &=g_{ij}
    \end{aligned}
    \ \right\}
\end{equation}
for $1\leq i,j \leq n$. Thus the metric coefficients
$g^{(1)}_{\alpha\beta}$ are given by universal fractional expressions
of the functions $g_{ij}$, $\partial_kg_{ij}$ and $v^i$ ($1\le
i,j,k\le n$).

Using induction again, for $m\ge2$, let
$(U^{(m)};x_{(m)}^1,\dots,x_{(m)}^{2^mn})$ be the chart of $T^{(m)}M$
induced by the chart
$(U^{(m-1)};x_{(m-1)}^1,\dots,x_{(m-1)}^{2^{m-1}n})$ of $T^{(m-1)}M$,
and let $g^{(m)}_{\alpha\beta}$ be the corresponding coefficients of
$g^{(m)}$.

\begin{lem}\label{l: g^(m)_alpha beta}
  \begin{enumerate}[{\rm(}i{\rm)}]
	
  \item\label{i: g^(m)_alpha beta} The coefficients
    $g^{(m)}_{\alpha\beta}$ are given by universal fractional
    expressions of the coordinates
    $x_{(m)}^{n+1},\dots,x_{(m)}^{2^mn}$ and the partial derivatives
    up to order $m$ of the coefficients $g_{ij}$.
		
  \item\label{i: partial up to order m of g_ij} For each $\rho>0$, the
    partial derivatives up to order $m$ of the coefficients $g_{ij}$
    are given by universal linear expressions of the functions
    $(\sigma^{(m)}_{\rho,\mu})^*g^{(m)}_{\alpha\beta}$ for
    $n+1\le\mu\le2^mn$, where $\sigma^{(m)}_{\rho,\mu}:U\to U^{(m)}$
    is the section of $\pi:U^{(m)}\to U$ determined
    by $(\sigma^{(m)}_{\rho,\mu})^*x_{(m)}^{\nu}=\rho\delta_{\mu\nu}$ for
    $n+1\le\nu\le2^mn$, using Kronecker's delta.
		
  \end{enumerate}
\end{lem}

\begin{proof}
  We proceed by induction on $m$. For $m=1$,~\eqref{i: g^(m)_alpha
    beta} holds by~\eqref{Gamma_ijk} and~\eqref{Sasaki}, and~\eqref{i:
    partial up to order m of g_ij} holds by the second and third
  equalities of~\eqref{Sasaki}, since
  $\partial_ig_{jk}=\Gamma_{ijk}+\Gamma_{ikj}$ by~\eqref{Gamma_ijk}.
  For arbitrary $m\ge2$, assuming that~\eqref{i: g^(m)_alpha beta}
  and~\eqref{i: partial up to order m of g_ij} hold for the case
  $m-1$, we get both properties for $m$ by applying the above case to
  $(g^{(m-1)})^{(1)}=g^{(m)}$.
\end{proof}

Let $\Omega\subset M$ be a compact domain and $m\in\N$.  Fix a finite
collection of charts of $M$ that covers $\Omega$,
$\UU=\{(U_a;x_a^1,\dots,x_a^n)\}$, and a family of compact subsets of
$M$ with the same index set as $\UU$, $\KK=\{K_a\}$, such that
$\Omega\subset\bigcup_aK_a$, and $K_a\subset U_a$ for all $a$. The
corresponding $C^m$ norm of a $C^m$ tensor $T$ on $\Omega$ is defined
by
\[
\|T\|_{C^m,\Omega,\UU,\KK} =\max_a\max_{x\in
  K_a\cap\Omega}\sum_{|I|\le m}\sum_{J,K}
\left|\frac{\partial^{|I|}T_{a,J}^K}{\partial x_a^I}(x)\right|\;,
\]
using the standard multi-index notation, where $T_{a,J}^K$ are the coefficients of $T$ on $U_a\cap\Omega$ with
respect to the frame induced by $(U_a;x_a^1,\dots,x_a^n)$. With this
norm, the $C^m$ tensors on $\Omega$ of a fixed type form a Banach
space. By taking the projective limit as $m\to\infty$, we get the
Fr\'echet space of $C^\infty$ tensors of that type equipped with the
\emph{$C^\infty$ topology} (see e.g.\ \cite{Hirsch1976}). Observe that
$\UU$ and $\KK$ are also qualified to define the norm $\|\
\|_{C^m,\Omega',\UU,\KK}$ for any compact subdomain
$\Omega'\subset\Omega$. It is well known that $\|\
\|_{C^m,\Omega,\UU,\KK}$ is equivalent to the norm $\|\
\|_{C^m,\Omega,g}$ defined by
\[
\|T\|_{C^m,\Omega,g}=\max_{0\le l\le
  m}\max_{x\in\Omega}|\nabla^lT(x)|\;;
\]
i.e., there is some $C\ge1$, depending only on $M$, $m$, $\Omega$, $\UU$,
$\KK$ and $g$, such that
\begin{equation}\label{norm equiv}
  \frac{1}{C}\,\|\ \|_{C^m,\Omega,\UU,\KK}\le\|\ \|_{C^m,\Omega,g}
  \le C\,\|\ \|_{C^m,\Omega,\UU,g}\;.
\end{equation}
	
When $\partial M=\emptyset$, it is said that $M$ is of \emph{bounded
  geometry} if $\inj_M>0$ and the function $|\nabla^m\RR|$ is bounded
for all $m\in\N$; in particular, $M$ is complete since
$\inj_M>0$. More precisely, given $r>0$ and a sequence $C_m>0$, if
$\inj_M\ge r$ and $|\nabla^m\RR|\le C_m$ for all $m\in\N$, then $(r,
C_m)$ is called a \emph{geometric bound} of $M$. A family $\CC$ of
Riemannian manifolds without boundary is called of \emph{equi-bounded
  geometry} if all of them are of bounded geometry with a common
geometric bound; i.e., their disjoint union is of bounded geometry.

\section{Quasi-isometries}\label{s: quasi-isometries}

Let $\phi\colon M \to N$ be a $C^1$ map between Riemannian
manifolds. Recall that $\phi$ is called a ($\lambda$-)
\emph{quasi-isometry}, or ($\lambda$-) \emph{quasi-isometric}, if
there is some $\lambda \geq 1$ such that $\frac{1}{\lambda}\,|\xi| \le
|\phi_*(\xi)| \leq \lambda \,|\xi|$ for every $\xi\in TM$. This
$\lambda$ is called a \emph{dilation bound} of $\phi$. The second of
the above inequalities, $|\phi_*(\xi)| \leq \lambda \,|\xi|$ for all
$\xi\in TM$, means that $|\phi_*|\le\lambda$; i.e.,
$|\phi_{*x}|\le\lambda$ for all $x\in M$.

\begin{rem}\label{r: quasi-isometries}
  \begin{enumerate}[(i)]
	
  \item\label{i: quasi-isometry is immersion} Every quasi-isometry is
    an immersion.
		
  \item\label{i: |phi_*| le lambda and distances} If
    $|\phi_*|\le\lambda$, then $\phi$ is $\lambda$-Lipschitz; i.e.,
    $d_N(\phi(x),\phi(y))\le\lambda\,d_M(x,y)$ for all $x,y\in M$.
		
  \item\label{i: quasi-isometry and distances} If $\phi:M\to N$ is a
    $\lambda$-quasi-isometry, then $\phi$ is $\lambda$-bi-Lipschitz;
    i.e., for all $x,y\in M$,
    \[
    \frac{1}{\lambda}\,d_M(x,y)\le
    d_N(\phi(x),\phi(y))\le\lambda\,d_M(x,y)\;.
    \]
			
  \item\label{i: |(psi phi)_*| le lambda mu} Let $\psi:N\to L$ be
    another $C^1$ map between Riemannian manifolds. If
    $|\phi_*|\le\lambda$ and $|\psi_*|\le\mu$, then
    $|(\psi\phi)_*|\le\lambda\mu$.
		
  \item\label{i: composition of quasi-isometries} The composition of a
    $\lambda$-quasi-isometry and a $\mu$-quasi-isometry is a
    $\lambda\mu$-quasi-isometry.
		
  \item\label{i: inverse of a quasi-isometric diffeo} The inverse of a
    $\lambda$-quasi-isometric diffeomorphism is a
    $\lambda$-quasi-isometric diffeomorphism.
		
  \end{enumerate}
\end{rem}

Consider the subbundle $T^{\le r}M=\{\,\xi\in TM\mid|\xi|\le
r\,\}\subset TM$ for each $r>0$. If $M$ has no boundary, then $T^{\le
  r}M$ is a manifold with boundary, being $\partial T^{\le
  r}M=T^rM:=\{\,\xi\in TM\mid|\xi|=r\,\}$; otherwise, $T^{\le r}M$ is
a manifold with corners. Also, define $T^{(m),\le r}M$ by induction on
$m\in\Z^+$, setting $T^{(1),\le r}M=T^{\le r}M$ and $T^{(m),\le
  r}M=T^{\le r}T^{(m-1),\le r}M$. Note that $T^{(m),\le r}T^{(m'),\le
  r}M=T^{(m+m'),\le r}M$.

\begin{defn}\label{d: quasi-isometry of order m}
	\begin{enumerate}[(i)]
	
        \item\label{i: quasi-isometry of order m} It is said that
          $\phi:M\to N$ is a ($\lambda$-) \emph{quasi-isometry of
            order $m\in\N$}, or a ($\lambda$-) \emph{quasi-isometric
            map of order $m$}, if it is $C^{m+1}$ and
          $\phi_*^{(m)}:T^{(m),\le1}M\to T^{(m)}N$ is a ($\lambda$-)
          quasi-isometry. This $\lambda$ is called a \emph{dilation
            bound of order $m$} of $\phi$. The infimum of all
          dilations bounds of order $m$ is called the \emph{dilation
            of order $m$}. If $\phi$ is a quasi-isometry of order $m$
          for all $m\in\N$, then it is called a \emph{quasi-isometry
            of order $\infty$}.
		
        \item\label{i: equi-quasi-isometries of order m} A collection
          $\Phi$ of maps between Riemannian manifolds is called a
          family of \emph{equi-quasi-isometries of order $m\in\N$} if
          it is a family of quasi-isometries of order $m$ with some
          common dilation bound of order $m$, which is called an
          \emph{equi-dilation bound of order $m$}. If $\Phi$ is a
          collection of equi-quasi-isometries of order $m$ for all
          $m\in\N$, then it is called a family of
          \emph{equi-quasi-isometries of order $\infty$}.
		
        \item\label{i: quasi-isometric with order m} A Riemannian
          manifold $M$ is said to be \emph{quasi-isometric with order
            $m$} to another Riemannian manifold $N$ when there is a
          quasi-isometric diffeomorphism of order $m$, $M\to N$. With
          more generality, a collection $\{M_i\}$ of Riemannian
          manifolds is called \emph{equi-quasi-isometric with order
            $m$} to another collection $\{N_i\}$ of Riemannian
          manifolds, with the same index set, when there is a
          collection of equi-quasi-isometric diffeomorphisms of order
          $m$, $\{M_i\to N_i\}$.
		
	\end{enumerate}
\end{defn}

\begin{rem}\label{r: quasi-isometry of order m}
  \begin{enumerate}[(i)]
	
  \item\label{i: quasi-isometries of order 0} The
    $\lambda$-quasi-isometries of order $0$ are the
    $\lambda$-quasi-isometries.
		
  \item\label{i: q-i of order m are q-i of order m-1} By
    Remark~\ref{r: T^(l)M subset T^(m)M}-\eqref{i: g^(m)|_T^(l)M =
      g^(l)}, if $\phi$ is a $\lambda$-quasi-isometry of order
    $m\ge1$, then it is a $\lambda$-quasi-isometry of order $m-1$.
		
  \item\label{i: phi_*^(m') is a lambda-q.-i. of order m-m'} For
    integers $0\le m'\le m$, if $\phi$ is a $\lambda$-quasi-isometry
    of order $m$, then $\phi_*^{(m')}$ is a $\lambda$-quasi-isometry
    of order $m-m'$.
		
  \end{enumerate}
\end{rem}

To begin with, let us clarify the concept of quasi-isometry of order
$1$. Consider the splittings $T^{(2)}M=\HH\oplus\VV$ and
$T^{(2)}N=\HH'\oplus\VV'$, where $\HH$ and $\HH'$ are the horizontal
subbundles, and $\VV$ and $\VV'$ are the vertical subbundles. Fix any
$x\in M$ and $\xi\in T_xM$, and let $x'=\phi(x)$ and
$\xi'=\phi_*(\xi)$. We have the canonical identities
\begin{equation}\label{HH oplus VV}
  T_\xi TM=\HH_\xi\oplus\VV_\xi\equiv T_xM\oplus T_xM\;,\quad
  T_{\xi'}TN=\HH'_{\xi'}\oplus\VV'_{\xi'}\equiv T_{x'}N\oplus T_{x'}N\;.
\end{equation}
The pull-back Riemannian vector bundle $\phi^*TN$ is endowed with the
pull-back $\nabla'$ of the Riemannian connection of $N$, and let
$\phi_*:TM\to\phi^*TN$ also denote the homomorphism over $\id_M$
induced by $\phi$. Let $X$ be a $C^\infty$ tangent vector field on
some neighborhood of $x$ in $M$ so that $X(x)=\xi$; thus $\phi_*X$ is
a $C^1$ local section of $\phi^*TN$ around $x$ satisfying
$(\phi_*X)(x)=\xi'\in(\phi^*TN)_x\equiv T_{\phi(x)}N$. Then, for any
$\zeta\in T_xM$ and each $C^\infty$ function $f$ defined on some
neighborhood of $x$, we have
\begin{multline*}
  \nabla'_{\zeta}(\phi_*(fX))-\phi_*(\nabla_\zeta (fX))
  =f(x)\,\nabla'_{\zeta}(\phi_*X)+df(\zeta)\,\phi_*\xi-f(x)\,\phi_*(\nabla_\zeta X)-df(\zeta)\,\phi_*\xi\\
  =f(x)\,(\nabla'_{\zeta}(\phi_*X)-\phi_*(\nabla_\zeta X))
\end{multline*}
in $(\phi^*TN)_x\equiv T_{x'}N$. Therefore
$A_\phi(\zeta\otimes\xi):=\nabla'_{\zeta}(\phi_*X)-\phi_*(\nabla_\zeta
X)$ depends only on $\zeta\otimes\xi$, and this expression defines a
continuous section $A_\phi$ of $TM^*\otimes
TM^*\otimes\phi^*TN$. Observe that $X$ can be chosen so that
$\nabla_\zeta X=0$, giving
$A_\phi(\zeta\otimes\xi)=\nabla'_{\zeta}(\phi_*X)$ in this case. Then,
from the definitions of tangent map and covariant derivative, it
easily follows that, according to~\eqref{HH oplus VV},
\begin{equation}\label{phi_**}
  \phi_{**\xi}(\zeta_1,\zeta_2)\equiv(\phi_*(\zeta_1),\phi_*(\zeta_2)+A_\phi(\zeta_1\otimes\xi))
\end{equation}
for all $\zeta_1,\zeta_2\in T_xM$.  

\begin{rem}
  If $TM$ were used instead of $T^{\le1}M$ in the definition of
  quasi-isometries of order $1$, we would get $A_\phi=0$, which is too
  restrictive. On the other hand, it would be weaker to use $T^1M$
  instead of $T^{\le1}M$.
\end{rem}

\begin{lem}\label{l: |phi_**xi| le mu}
  Suppose that $\phi:M\to N$ is $C^2$. Then the following properties
  hold for $r>0$ and $\mu,\nu,K\ge0$:
  \begin{enumerate}[{\rm(}i{\rm)}]

  \item\label{i: |phi_**xi| le mu} If $|\phi_{**\xi}|\le\mu$ for all
    $\xi\in T^{\le r}M$, then $|\phi_*|\le\mu$ and $|A_\phi|\le\mu/r$.

  \item\label{i: |phi_*| le nu} If $|\phi_*|\le\nu$ and $|A_\phi|\le
    K$, then $|\phi_{**\xi}|\le\sqrt{2}(\nu+Kr)$ for all $\xi\in
    T^{\le r}M$.

  \end{enumerate}
\end{lem}

\begin{proof}
  Assume that $|\phi_{**\xi}|\le\mu$ for all $\xi\in T^{\le r}M$. We
  get $|\phi_*|\le\mu$ by Remark~\ref{r: T^(l)M subset
    T^(m)M}-\eqref{i: g^(m)|_T^(l)M = g^(l)}. Furthermore, for all
  $x\in M$ and $\xi,\zeta\in T_xM$ with $|\xi|=r$, according
  to~\eqref{HH oplus VV} and~\eqref{phi_**},
  \[
  |A_\phi(\zeta\otimes\xi)|\le|(\phi_{*x}(\zeta),A_\phi(\zeta\otimes\xi))|=|\phi_{**\xi}(\zeta,0)|
  \le\mu\,|(\zeta,0)|=\mu\,|\zeta|=\frac{\mu}{r}\,|\zeta|\,|\xi|\;.
  \]
	
  Now, suppose that $|\phi_*|\le\nu$ and $|A_\phi|\le K$. Fix all
  $x\in M$ and $\xi,\zeta_1,\zeta_2\in T_xM$ with $|\xi|\le r$,
  according to~\eqref{HH oplus VV} and~\eqref{phi_**},
		\begin{multline*}
			|\phi_{**\xi}(\zeta_1,\zeta_2)|\le|\phi_*(\zeta_1)|+|\phi_*(\zeta_2)+A_\phi(\zeta_1\otimes\xi)|
			\le\nu\,|\zeta_1|+\nu\,|\zeta_2|+K\,|\zeta_1|\,|\xi|\\
			\le\nu\,|\zeta_1|+\nu\,|\zeta_2|+Kr\,|\zeta_1|
			\le(\nu+Kr)\,(|\zeta_1|+|\zeta_2|)
			\le\sqrt{2}(\nu+Kr)\,|(\zeta_1,\zeta_2)|\;.\qed
		\end{multline*}
\renewcommand{\qed}{}
\end{proof}

\begin{lem}\label{l: A_phi}
  Suppose that $\phi:M\to N$ is $C^2$. Then the following conditions
  are equivalent for $r>0$:
  \begin{enumerate}[{\rm(}i{\rm)}]

  \item\label{l: phi_*:T^le rM to TN} $\phi_*:T^{\le r}M\to TN$ is a
    quasi-isometry.

  \item\label{i: A_phi} $\phi$ is a quasi-isometry and $|A_\phi|$ is
    uniformly bounded.

  \end{enumerate}
  In this case, the constants involved in the above properties are
  related in the following way:
  \begin{enumerate}[{\rm(}a{\rm)}]
					
  \item\label{i: |A_phi| le mu/r} If $\mu$ is a dilation bound of
    $\phi_*:T^{\le r}M\to TN$, then $\mu$ is a dilation bound of
    $\phi$ and $|A_\phi|\le\mu/r$.
			
  \item\label{i: kappa} If $\nu$ is a dilation bound of $\phi$,
    $|A_\phi|\le K$, and $0<\kappa<1$ with $\nu K\kappa
    r<1$, then
    \[
    \mu=\max\left\{\sqrt{2}(\nu+Kr),\frac{\sqrt{2}\nu}{1-\nu K\kappa
        r},\frac{\sqrt{2}\nu}{\kappa}\right\}
    \]
    is a dilation bound of $\phi_*:T^{\le r}M\to TN$.
			
  \end{enumerate}
\end{lem}

\begin{proof}
  Assume that~\eqref{l: phi_*:T^le rM to TN} holds, and let $\mu$ be a
  dilation bound of order $1$ of $\phi$. Then $\phi$ is a
  $\mu$-quasi-isometry by Remark~\ref{r: T^(l)M subset
    T^(m)M}-\eqref{i: g^(m)|_T^(l)M = g^(l)}. This shows~\eqref{i:
    A_phi} and~\eqref{i: |A_phi| le mu/r} by Lemma~\ref{l: |phi_**xi|
    le mu}-\eqref{i: |phi_**xi| le mu}.
	
  Now, suppose that ~\eqref{i: A_phi} holds, and take $\nu$, $K$,
  $\kappa$ and $\mu$ like in~\eqref{i: kappa}. For all $x\in M$ and
  $\xi,\zeta_1,\zeta_2\in T_xM$ with $|\xi|\le r$, according
  to~\eqref{HH oplus VV} and~\eqref{phi_**},
  \begin{multline*}
    |\phi_{**\xi}(\zeta_1,\zeta_2)|
    \ge\frac{1}{\sqrt{2}}\,(|\phi_*(\zeta_1)|+|\phi_*(\zeta_2)+A_\phi(\zeta_1\otimes\xi)|)
    \ge\frac{1}{\sqrt{2}}\,(|\phi_*(\zeta_1)|+\kappa\,|\phi_*(\zeta_2)+A_\phi(\zeta_1\otimes\xi)|)\\
    \ge\frac{1}{\sqrt{2}}\,(|\phi_*(\zeta_1)|+\kappa(|\phi_*(\zeta_2)|-|A_\phi(\zeta_1\otimes\xi)|))
    \ge\frac{1}{\sqrt{2}}\,\left(\left(\frac{1}{\nu}-K\kappa\,|\xi|\right)|\zeta_1|
      +\frac{\kappa}{\nu}\,|\zeta_2|\right)\\
    \ge\frac{1}{\sqrt{2}}\,\left(\left(\frac{1}{\nu}-K\kappa
        r\right)|\zeta_1| +\frac{\kappa}{\nu}\,|\zeta_2|\right)
    \ge\frac{1}{\mu}\,(|\zeta_1|+|\zeta_2|)
    \ge\frac{1}{\mu}\,|(\zeta_1,\zeta_2)|\;.
  \end{multline*}
  This gives~\eqref{l: phi_*:T^le rM to TN} and~\eqref{i: kappa} by
  Lemma~\ref{l: |phi_**xi| le mu}-\eqref{i: |phi_*| le nu}.
\end{proof}

For $c>0$, let $h_c:TM\to TM$ be the $C^\infty$ diffeomorphism defined
by $h_c(\xi)=c\xi$. Observe that $h_c(T^{\le1}M)=T^{\le c}M$, and the
following diagram is commutative:
\[
\begin{CD}
  TM @>{\phi_*}>> TN \\
  @V{h_c}VV @VV{h_c}V \\
  TM @>{\phi_*}>> TN
\end{CD}
\]

For each $m\in\Z^+$, let $\HH^{(m+1)}$ and $\VV^{(m+1)}$ denote the
horizontal and vertical vector subbundles of $T^{(m+1)}M$ over
$T^{(m)}M$. Thus, for $\xi\in T^{(m-1)}M$ and $\zeta\in T_\xi
T^{(m-1)}M$,
\begin{equation}\label{HH^(m+1)_zeta oplus VV^(m+1)}
  T_\zeta T^{(m)}M=\HH^{(m+1)}_\zeta\oplus\VV^{(m+1)}_\zeta
  \equiv T_\xi T^{(m-1)}M\oplus T_\xi T^{(m-1)}M\;.
\end{equation}

\begin{lem}\label{l: h_c*^(m)}
  For all $m\in\Z^+$, there is an orthogonal vector bundle
  decomposition, $T^{(m+1)}M=\PP^{(m+1)}\oplus\QQ^{(m+1)}$, preserved
  by $h_{c*}^{(m)}$, such that, for $\xi\in T^{(m-1)}M$, $\zeta\in
  T_\xi T^{(m-1)}M$ and $\zeta'=h_{c*}^{(m)}(\zeta)$, the canonical
  identity $T_\zeta T^{(m)}M\equiv T_{\zeta'}T^{(m)}M$ given
  by~\eqref{HH^(m+1)_zeta oplus VV^(m+1)} induces identities,
  $\PP^{(m+1)}_\zeta\equiv\PP^{(m+1)}_{\zeta'}$ and
  $\QQ^{(m+1)}_\zeta\equiv\QQ^{(m+1)}_{\zeta'}$, so that
  $h_{c*}^{(m)}:\PP^{(m+1)}_\zeta\to\PP^{(m+1)}_{\zeta'}\equiv\PP^{(m+1)}_\zeta$
  is the identity, and
  $h_{c*}^{(m)}:\QQ^{(m+1)}_\zeta\to\QQ^{(m+1)}_{\zeta'}\equiv\QQ^{(m+1)}_\zeta$
  is multiplication by $c$.
\end{lem}

\begin{proof}
  The proof is by induction on $m$. By the definition of connection,
  $h_{c*}$ preserves the orthogonal decomposition
  $T^{(2)}M=\HH\oplus\VV$. Moreover, for $\zeta\in TM$ and
  $\zeta'=c\zeta$, $h_{c*}:\HH_\zeta\to\HH_{\zeta'}\equiv\HH_\zeta$ is
  the identity, and $h_{c*}:\VV_\zeta\to\VV_{\zeta'}\equiv\VV_\zeta$
  is multiplication by $c$. Thus the statement is true in this case
  with $\PP^{(2)}=\HH$ and $\QQ^{(2)}=\VV$.
	
  Now, suppose that $m\ge2$ and the result holds for $m-1$. For
  $\xi\in T^{(m-1)}M$ and $\zeta\in T_\xi T^{(m-1)}M$, we have
  canonical identities
  \begin{equation}\label{HH^(m+1)_zeta equiv VV^(m+1)}
    \HH^{(m+1)}_\zeta\equiv\VV^{(m+1)}_\zeta\equiv T_\xi T^{(m-1)}M
    =\PP^{(m)}_\xi\oplus\QQ^{(m)}_\xi\;,
  \end{equation}
  obtaining orthogonal decompositions,
  $\HH^{(m+1)}=\HH\PP^{(m)}\oplus\HH\QQ^{(m)}$ and
  $\VV^{(m+1)}=\VV\PP^{(m)}\oplus\VV\QQ^{(m)}$, where
  $(\HH\PP^{(m)})_\zeta\equiv\PP^{(m)}_\xi\equiv(\VV\PP^{(m)})_\zeta$
  and
  $(\HH\QQ^{(m)})_\zeta\equiv\QQ^{(m)}_\xi\equiv(\VV\QQ^{(m)})_\zeta$
  according to~\eqref{HH^(m+1)_zeta equiv VV^(m+1)}. Then the result
  follows with $\PP^{(m+1)}=\HH\PP^{(m)}\oplus\VV\PP^{(m)}$ and
  $\QQ^{(m+1)}=\HH\QQ^{(m)}\oplus\VV\QQ^{(m)}$.
\end{proof}


\begin{cor}\label{c: h_c*^(m)}
  For all $m\in\Z^+$ and $c,r>0$, we have $h_{c*}^{(m)}(T^{(m+1),\le
    r}M)\subset T^{(m+1),\le\bar cr}M$, where $\bar c=\max\{c,1\}$,
  and $h_{c*}^{(m)}:T^{(m+1)}M\to T^{(m+1)}M$ is a $\hat
  c$-quasi-isometry, where $\hat c=\max\{c,1/c\}$.
\end{cor}

\begin{lem}\label{l: |(phi_*^(m))_*xi| le lambda}
  For all $m\in\Z^+$, $r,s>0$ and $\lambda\ge0$, there is some
  $\mu\ge0$ such that, for any $C^{m+1}$ map between Riemannian
  manifolds, $\phi:M\to N$, if $|(\phi_*^{(m)})_{*\xi}|\le\lambda$ for
  all $\xi\in T^{(m),\le r}M$, then $|(\phi_*^{(m)})_{*\xi}|\le\mu$
  for all $\xi\in T^{(m),\le s}M$. Moreover $\mu$ can be chosen so
  that $\mu s\to0$ as $s\to0$ for fixed $m$, $r$ and $\lambda$.
\end{lem}

\begin{proof}
  We proceed by induction on $m$.
	
  For $m=1$, we have $|\phi_{**\xi}|\le\lambda$ for all $\xi\in T^{\le
    r}M$. Then $|\phi_*|\le\lambda$ and $|A_\phi|\le\lambda/r$ by
  Lemma~\ref{l: |phi_**xi| le mu}-\eqref{i: |phi_**xi| le mu}. Using
  Lemma~\ref{l: |phi_**xi| le mu}-\eqref{i: |phi_*| le nu}, it follows
  that $|\phi_{**\xi}|\le\sqrt{2}\lambda(1+s/r)=:\mu$ for all $\xi\in
  T^{\le s}M$. Note that $\mu s\to0$ as $s\to0$ for fixed $r$ and
  $\lambda$ in this case.
	
  Now, assume that $m\ge2$ and the result holds for $m-1$. For $c=r/s$
  and $t=\min\{cr,r\}$, the diagram
  \begin{equation}\label{CD with phi_*^(m), h_1/c* and h_c*}
    \begin{CD}
      T^{(m),\le r}M @>{\phi_*^{(m)}}>> T^{(m)}N \\
      @A{h_{1/c*}^{(m-1)}}AA @VV{h_{c*}^{(m-1)}}V \\
      T^{(m-1),\le t}T^{\le s}M @>{\phi_*^{(m)}}>> T^{(m)}N
    \end{CD}
  \end{equation}
  is defined and commutative. By Corollary~\ref{c: h_c*^(m)} and
  Remark~\ref{r: quasi-isometries}-\eqref{i: |(psi phi)_*| le lambda
    mu}, it follows that $|(\phi_*^{(m)})_{*\xi}|\le\hat c^2\lambda$
  for all $\xi\in T^{(m-1),\le t}T^{\le s}M$, where $\hat
  c=\max\{c,1/c\}$. Then, by the induction hypothesis applied to the
  map $\phi_*:T^{\le s}M\to TN$, there is some $\mu\ge0$, depending
  only on $m-1$, $t$, $s$ and $\hat c^2\lambda$, such that
  $|(\phi_*^{(m)})_{*\xi}|\le\mu$ for all $\xi\in T^{(m-1),\le
    s}T^{\le s}M=T^{(m),\le s}M$, and so that $\mu s\to0$ as $s\to0$ for fixed $m$,
  $t$ and $\hat c^2\lambda$.
\end{proof}

\begin{cor}\label{c: phi_*^(m)(T^(m),le rM) subset T^(m),le lambda rN}
  For all $m\in\Z^+$, $r>0$ and $\lambda\ge0$, there is some $s>0$
  such that, for any $C^{m+1}$ map between Riemannian manifolds,
  $\phi:M\to N$, if $|(\phi_*^{(m)})_{*\xi}|\le\lambda$ for all
  $\xi\in T^{(m),\le1}M$, then $\phi_*^{(m+1)}(T^{(m+1),\le
    s}M)\subset T^{(m+1),\le r}N$.
\end{cor}

\begin{proof}
  This is also proved by induction on $m$. The statement is true for
  $m=0$ because, if $|\phi_*|\le\lambda$, then $\phi_*(T^{\le
    s}M)\subset T^{\le\lambda s}N$ for all $s>0$, and therefore it is
  enough to take $s=r/\lambda$ in this case.
	
  Now, assume that $m\ge1$ and the result is true for $m-1$. By
  Remark~\ref{r: T^(l)M subset T^(m)M}-\eqref{i: g^(m)|_T^(l)M =
    g^(l)}, if $|(\phi_*^{(m)})_{*\xi}|\le\lambda$ for all $\xi\in
  T^{(m),\le1}M$, then $|(\phi_*^{(m-1)})_{*\xi}|\le\lambda$ for all
  $\xi\in T^{(m-1),\le1}M$. Hence, by the induction hypothesis, for
  all $r>0$, there is some $s>0$, as small as desired, such that
  $\phi_*^{(m)}(T^{(m),\le s}M)\subset T^{(m),\le r}N$. On the other
  hand, by Lemma~\ref{l: |(phi_*^(m))_*xi| le lambda}, there is some
  $\mu>0$, depending on $m$, $r$, $s$ and $\lambda$, such that
  $|(\phi_*^{(m)})_{*\xi}|\le\mu$ for all $\xi\in T^{(m),\le s}M$, and
  satisfying $\mu s\to0$ as $s\to0$ for fixed $m$, $r$ and
  $\lambda$. Thus we can choose $s$, and the corresponding $\mu$, so
  that $\mu s\le r$. Then
		\[
			\phi_*^{(m+1)}(T^{(m+1),\le s}M)\subset T^{\le\mu s}T^{(m),\le r}N\subset T^{(m+1),\le r}N\;.\qed
		\]
\renewcommand{\qed}{}
\end{proof}

\begin{lem}\label{l: T^(m),le rM}
  For $m\in\Z^+$, $r,s>0$ and $\lambda\ge1$, there is some $\mu\ge1$
  such that, for any $C^{m+1}$ map between Riemannian manifolds,
  $\phi:M\to N$, if $\phi_*^{(m)}:T^{(m),\le r}M\to T^{(m)}N$ is a
  $\lambda$-quasi-isometry, then $\phi_*^{(m)}:T^{\le s}M\to T^{(m)}N$
  is a $\mu$-quasi-isometry.
\end{lem}

\begin{proof}
  Again, we use induction on $m$. The case $m=1$ is a direct
  consequence of Lemma~\ref{l: A_phi}.
	
  Now, assume that $m\ge2$ and the result holds for $m-1$. Consider
  the notation of the proof of Lemma~\ref{l: |(phi_*^(m))_*xi| le
    lambda}. From the commutativity of~\eqref{CD with phi_*^(m),
    h_1/c* and h_c*}, and using Corollary~\ref{c: h_c*^(m)}
  and Remark~\ref{r: quasi-isometries}-\eqref{i: composition of
    quasi-isometries}, it follows that the lower horizontal arrow
  of~\eqref{CD with phi_*^(m), h_1/c* and h_c*} is a $\hat
  c^2\lambda$-quasi-isometry. Then, by the induction hypothesis
  applied to the map $\phi_*:T^{\le s}M\to TN$, there is some $\mu>0$,
  depending only on $m-1$, $t$, $s$ and $\hat c^2\lambda$, such that
  $\phi_*^{(m)}:T^{(m),\le s}M\to T^{(m)}N$ is a $\mu$-quasi-isometry.
\end{proof}

\begin{rem}
  According to Lemma~\ref{l: T^(m),le rM}, we could use any
  $T^{(m),\le r}M$ instead of $T^{(m),\le1}M$ to define
  quasi-isometries of order $m$, but the dilation bounds of order $m$
  would be different.
\end{rem}

\begin{prop}\label{p: quasi-isometries of order m}
  \begin{enumerate}[{\rm(}i{\rm)}]
	
  \item\label{i: composites of quasi-isometries of order m} For all
    $m\in\N$ and $\lambda,\mu\ge1$, there is some $\nu\ge1$ such that,
    if $\phi:M\to N$ and $\psi:N\to L$ are quasi-isometries of order
    $m$, and $\lambda$ and $\mu$ are respective dilation bounds of
    order $m$, then $\psi\phi$ is a $\nu$-quasi-isometry of order $m$.
		
  \item\label{i: inverses of quasi-isometric diffeomorphisms of order
      m} For all $m\in\N$ and $\lambda\ge1$, there is some $\mu\ge1$
    such that, if $\phi:M\to N$ is a $\lambda$-quasi-isometric
    diffeomorphism of order $m$, then $\phi^{-1}$ is a
    $\mu$-quasi-isometry of order $m$.
		
  \end{enumerate}
\end{prop}

\begin{proof}
  Let us prove~\eqref{i: composites of quasi-isometries of order
    m}. By Corollary~\ref{c: phi_*^(m)(T^(m),le rM) subset T^(m),le
    lambda rN}, there is some $r>0$, depending on $m$ and $\lambda$,
  such that
  \[
  \phi_*^{(m+1)}(T^{(m+1),\le r}M)\subset T^{(m+1),\le1}N\;,
  \]
  and therefore $\phi_*^{(m)}(T^{(m),\le r}M)\subset
  T^{(m),\le1}N$. On the other hand, by Lemma~\ref{l: T^(m),le rM},
  there is some $\lambda'\ge1$, depending on $m$, $r$ and $\lambda$,
  such that $\phi_*^{(m)}:T^{(m),\le r}M\to T^{(m),\le1}N$ is a
  $\lambda'$-quasi-isometry. So
  \[
  (\psi\phi)_*^{(m)}=\psi_*^{(m)}\phi_*^{(m)}:T^{(m),\le r}M\to
  T^{(m)}L
  \]
  is a $\lambda'\mu$-quasi-isometry by Remark~\ref{r:
    quasi-isometries}-\eqref{i: composition of
    quasi-isometries}. Thus, by Lemma~\ref{l: T^(m),le rM}, there is
  some $\nu\ge1$, depending on $m$, $r$ and $\lambda'\mu$, so that
  $(\psi\phi)_*^{(m)}:T^{(m),\le1}M\to T^{(m)}L$ is a
  $\nu$-quasi-isometry; i.e., $\psi\phi$ is a $\nu$-quasi-isometry of
  order $m$.
	
  Now, let us prove~\eqref{i: inverses of quasi-isometric
    diffeomorphisms of order m}. By Corollary~\ref{c:
    phi_*^(m)(T^(m),le rM) subset T^(m),le lambda rN}, there is some
  $r>0$, depending on $m$ and $\lambda$, such that
  \[
  (\phi^{-1})_*^{(m+1)}(T^{(m+1),\le r}N)\subset T^{(m+1),\le1}M\;,
  \]
  and therefore $(\phi^{-1})_*^{(m)}(T^{(m),\le r}N)\subset
  T^{(m),\le1}M$. So
  \[
  \phi_*^{(m)}:(\phi^{-1})_*^{(m)}(T^{(m),\le r}N)\to T^{(m),\le r}N
  \]
  is a $\lambda$-quasi-isometric diffeomorphism, obtaining that
  \[
  (\phi^{-1})_*^{(m)}=(\phi_*^{(m)})^{-1}:T^{(m),\le
    r}N\to(\phi^{-1})_*^{(m)}(T^{(m),\le r}N)
  \]
  is a $\lambda$-quasi-isometry by Remark~\ref{r:
    quasi-isometries}-\eqref{i: inverse of a quasi-isometric
    diffeo}. Thus, by Lemma~\ref{l: T^(m),le rM}, there is some
  $\mu\ge1$, depending on $m$, $r$ and $\lambda$, so that
  $(\phi^{-1})_*^{(m)}:T^{(m),\le1}N\to T^{(m)}M$ is a
  $\mu$-quasi-isometry; i.e., $\phi^{-1}$ is a $\mu$-quasi-isometry of
  order $m$.
\end{proof}


\begin{cor}\label{c: being quasi-isometric is equiv rel}
  ``Being quasi-isometric with order $m$'' is an equivalence relation.
\end{cor}

Let $M$ and $N$ be connected Riemannian manifolds. For every
$m\in\N\cup\{\infty\}$, consider the weak $C^m$ topology on $C^m(M,N)$
(see \cite{Hirsch1976}). For $x\in M$ and $\Phi\subset C^m(M,N)$, let
$\Phi(x)=\{\,\phi(x)\mid\phi\in\Phi\,\}\subset N$.

\begin{prop}\label{p: Arzela-Ascoli}
  Assume that $N$ is complete. Let $x_0\in M$, and let $\Phi\subset
  C^{m+1}(M,N)$ be a family of equi-quasi-isometries of order
  $m\in\N\cup\{\infty\}$. Then $\Phi$ is precompact in $C^m(M,N)$ if
  and only if $\Phi(x_0)$ is bounded in $N$.
\end{prop}

\begin{proof}
  The ``only if'' part follows because the evaluation map $C^m(M,N)\to
  N$, $\phi\mapsto\phi(x_0)$, is continuous.

  For $m\in\N$, the ``if'' part is proved by induction. For $m=0$, the
  assumption that $\Phi\subset C^1(M,N)$ is a family of
  equi-quasi-isometries implies that $\Phi$ is equi-continuous by
  Remark~\ref{r: quasi-isometries}-\eqref{i: quasi-isometry and
    distances}. On the other hand,
  $\Phi(x)\subset\Pen_N(\Phi(x_0),\lambda\,d(x,x_0))$ for any $x\in M$
  by Remark~\ref{r: quasi-isometries}-\eqref{i: quasi-isometry and
    distances}, where $\lambda\ge1$ is an equi-dilation bound of
  $\Phi$. So $\Phi(x)$ is precompact in $N$ because $\Phi(x_0)$ is
  bounded and $N$ is complete. Therefore $\Phi$ is precompact in
  $C(M,N)$ by the Arzel\`a-Ascoli theorem.

  Now, take an integer $m\ge1$ and assume that the result holds for
  $m-1$. The map $C^m(M,N)\to C^{m-1}(T^{\le1}M,TN)$,
  $\phi\mapsto\phi_*|_{T^{\le1}M}$, is an embedding. So it is enough
  to prove that the image $\Phi_*$ of $\Phi$ by this map is precompact
  in $C^{m-1}(T^{\le1}M,TN)$. This holds by the induction hypothesis
  because $\Phi_*\subset C^m(T^{\le1}M,TN)$ is a family of
  equi-quasi-isometries of order $m-1$ by Remark~\ref{r:
    quasi-isometry of order m}-\eqref{i: phi_*^(m') is a
    lambda-q.-i. of order m-m'}.

  The ``if'' part for $m=\infty$ can be proved as follows. In this
  case, we have proved that $\Phi$ is precompact in $C^l(M,N)$ for
  every $l\in\N$. By the continuity of the inclusion maps
  $C^{l+1}(M,N)\hookrightarrow C^l(M,N)$, it follows that $\Phi$ has
  the same closure $\ol\Phi$ in $C^l(M,N)$ and $C^{l+1}(M,N)$, and the
  weak $C^l$ and $C^{l+1}$ topologies coincide on $\ol\Phi$. Therefore
  $\ol\Phi$ is the closure of $\Phi$ in $C^\infty(M,N)$ too, and the
  weak $C^\infty$ and $C^l$ topologies coincide on $\ol\Phi$ for any
  $l\in\N$. Thus $\Phi$ is precompact in $C^\infty(M,N)$.
\end{proof}

\section{Partial quasi-isometries}\label{s: partial quasi-isometries}

Let $M$ and $N$ be connected complete Riemannian manifolds without
boundary.

\begin{defn}
  For $m\in\N$, a partial map $f:M\rightarrowtail N$ is called a
  \emph{$C^m$ local diffeomorphism} if $\dom f$ and $\im f$ are open
  in $M$ and $N$, respectively, and $f:\dom f\to\im f$ is a $C^m$
  diffeomorphism. If moreover $f(x)=y$ for distinguished points,
  $x\in\dom f$ and $y\in\im f$, then $f$ is said to be \emph{pointed},
  and the notation $f:(M,x)\rightarrowtail(N,y)$ is used. The term
  \emph{local homeomorphism} is used in the $C^0$ case.
\end{defn}

The term ``$C^m$ local diffeomorfism'' ($m\ge1$) may be also used in the standard sense, referring to any $C^m$ map $M\to N$ whose tangent map is an isomorphism at every point of $M$. The context will always clarify this ambiguity.

\begin{defn}\label{d: (m,R,lambda)-pointed local quasi-isometry}
  For $m\in\N$, $R>0$ and $\lambda\ge1$, a $C^{m+1}$ pointed local
  diffeomorphism $\phi \colon (M,x) \rightarrowtail (N,y)$ is called
  an \emph{$(m,R,\lambda)$-pointed local quasi-isometry}, or a
  \emph{local quasi-isometry} of \emph{type} $(m,R,\lambda)$, if the
  restriction $\phi_*^{(m)}:\Omega^{(m)}\to T^{(m)}N$ is a
  $\lambda$-quasi-isometry for some compact domain
  $\Omega^{(m)}\subset\dom\phi_*^{(m)}$ with
  $B_M^{(m)}(x,R)\subset\Omega^{(m)}$.
\end{defn}

\begin{rem}\label{r: (m,R,lambda)-...}
  \begin{enumerate}[(i)]
	
  \item\label{i: ... (m,R,lambda) is ... (m',R',lambda')} Any pointed
    local quasi-isometry $(M,x)\rightarrowtail(N,y)$ of type
    $(m,R,\lambda)$ is also of type $(m',R',\lambda')$ for $0\leq m'
    \leq m$, $0<R'<R$ and $\lambda'>\lambda$ (using Remark~\ref{r:
      T^(l)M subset T^(m)M}-\eqref{i: g^(m)|_T^(l)M = g^(l)}).
		
  \item\label{i: ... (m,R,lambda) if and only if ... (m-m',R,lambda)} For integers
    $0\leq m' \leq m$, any pointed $C^{m+1}$ local diffeomorphism
    $\phi:(M,x)\rightarrowtail(N,y)$ is a pointed local quasi-isometry
    of type $(m,R,\lambda)$ if and only if
    $\phi_*^{(m')}:(T^{(m')}M,x)\rightarrowtail(T^{(m')}N,y)$ is a
    pointed local quasi-isometry of type $(m-m',R,\lambda)$.
		
  \item\label{i: C^infty approximation of (m,R,lambda)-...} If there
    is an $(m,R,\lambda)$-pointed local quasi-isometry
    $(M,x)\rightarrowtail(N,y)$, then, for all $R'<R$ and
    $\lambda'>\lambda$, there is a $C^\infty$
    $(m,R',\lambda')$-pointed local quasi-isometry
    $(M,x)\rightarrowtail(N,y)$ by \cite[Theorem~2.7]{Hirsch1976}.
		
  \end{enumerate}
\end{rem}

\begin{lem}\label{l: composition and inversion of pointed local quasi-isometries}
  The following properties hold:
  \begin{enumerate}[{\rm(}i{\rm)}]
		
  \item\label{i: composition} If $\phi: (M,x) \rightarrowtail (N,y)$
    and $\psi: (N,y)\rightarrowtail(L,z)$ are pointed local
    quasi-isometries of types $(m,R,\lambda)$ and $(m,\lambda
    R,\lambda')$, respectively, then
    $\psi\circ\phi:(M,x)\rightarrowtail (L,z)$ is an
    $(m,R,\lambda\lambda')$-pointed local quasi-isometry.
			
  \item\label{i: inversion} If $\phi: (M,x) \rightarrowtail (N,y)$ is
    an $(m,\lambda R,\lambda)$-pointed local quasi-isometry, then
    $\phi^{-1}:(N,y)\rightarrowtail(M,x)$ is an
    $(m,R,\lambda)$-pointed local quasi-isometry.
			
  \end{enumerate}
\end{lem}

\begin{proof}
  To prove~\eqref{i: composition}, it is enough to show that $\ol
  B_M^{(m)}(x,R)\subset\dom(\psi\circ\phi)_*^{(m)}$ by Remark~\ref{r:
    quasi-isometries}-\eqref{i: composition of quasi-isometries}. For
  $\xi\in\ol B_M^{(m)}(x,R)$, we have $\xi\in\dom\phi$ and
  $d_N^{(m)}(y,\phi_*^{(m)}(\xi))\le\lambda\,
  d_M^{(m)}(x,\xi)\le\lambda R$ by Remark~\ref{r:
    quasi-isometries}-\eqref{i: quasi-isometry and distances},
  obtaining that $\xi\in\dom(\psi\circ\phi)_*^{(m)}$ since $(\psi\circ\phi)_*^{(m)}=\psi_*^{(m)}\circ\phi_*^{(m)}$.
	
  To prove~\eqref{i: inversion}, it is enough to show that $\ol
  B_N^{(m)}(y,R)\subset\phi_*^{(m)}(\ol B_M^{(m)}(x,\lambda R))$ by Remark~\ref{r:
    quasi-isometries}-\eqref{i: inverse of a quasi-isometric
    diffeo}. Let $A=\ol B_N^{(m)}(y,R)\cap\im\phi_*^{(m)}$, which is open in
  $\ol B_N^{(m)}(y,R)$ and contains $y$. For any $\zeta\in A$, there
  is some $\xi\in\dom\phi_*^{(m)}$ so that $\phi_*^{(m)}(\xi)=\zeta$. Then
  $d_M^{(m)}(x,\xi)\le\lambda d_N(y,\zeta)\le\lambda R$ by
  Remark~\ref{r: quasi-isometries}-\eqref{i: quasi-isometry and
    distances}, obtaining that $\xi\in\ol B_M^{(m)}(x,\lambda
  R)$. Thus $A=\phi_*^{(m)}(\ol B_M^{(m)}(x,\lambda R))\cap\ol
  B_N^{(m)}(y,R)$, which is closed in $\ol B_N^{(m)}(y,R)$. Therefore
  $\ol B_N^{(m)}(y,R)=A\subset\phi_*^{(m)}(\ol B_M^{(m)}(x,\lambda R))$ because $\ol B_N^{(m)}(y,R)$ is connected.
\end{proof}

\section{The $C^\infty$ topology on $\MM_*(n)$}\label{s: C^infty topology}

\begin{defn}\label{d: U^m_R,r}
  For $m\in\N$ and $R,r>0$, let $U^m_{R,r}$ be the set of pairs
  $([M,x],[N,y])\in\MM_*(n)\times\MM_*(n)$ such that there is some
  $(m,R,\lambda)$-pointed local quasi-isometry $(M,x)\rightarrowtail
  (N,y)$ for some $\lambda\in[1,e^r)$.
\end{defn}

The following standard notation is
    used for a set $X$ and relations $U,V\subset X\times X$:
\begin{align*}
	U^{-1}&=\{\,(y,x)\in X\times X\mid(x,y)\in U\,\}\;,\\
	V\circ U&=\{\,(x,z)\in X\times X\mid\exists y\in X\ \text{so that}\ (x,y)\in U\ \text{and}\ (y,z)\in V\,\}\;.
\end{align*}
Moreover the diagonal of $X\times X$ is denoted by $\Delta$.

\begin{prop}\label{p: C^infty uniformity in MM_*(n)}
  The following properties hold for all $m,m'\in\N$ and $R,S,r,s>0$:
\begin{enumerate}[{\rm(}i{\rm)}]
  	
\item\label{i: ^-1} $(U^m_{e^rR,r})^{-1} \subset U^m_{R,r}$.
		
\item\label{i: cap} $U^{m_0}_{R_0,r_0}\subset U^m_{R,r}\cap
  U^{m'}_{S,s}$, where $m_0=\max\{m,m'\}$, $R_0=\max\{R,S\}$ and
  $r_0=\min\{r,s\}$.
		
\item\label{i: Delta} $\Delta\subset U^m_{R,r}$.
		
\item\label{i: circ} $U^m_{e^sR,r}\circ U^m_{R,s} \subset
  U^m_{R,r+s}$.
		
\end{enumerate}
\end{prop}

\begin{proof}
  Properties~\eqref{i: cap} and~\eqref{i: Delta} are elementary,
  and~\eqref{i: ^-1} and~\eqref{i: circ} are consequences of
  Lemma~\ref{l: composition and inversion of pointed local
    quasi-isometries}.
\end{proof}

\begin{prop}\label{p: Hausdorff MM_*(n)}
$\bigcap_{R,r>0}U^m_{R,r}=\Delta$ for all $m\in\N$.
\end{prop}

\begin{proof}
  We only prove ``$\subset$'' because ``$\supset$'' is obvious. For
  $([M,x],[N,y])\in\bigcap_{R,r>0}U^m_{R,r}$, there is a sequence of
  pointed local quasi-isometries $\phi_i:(M,x)\rightarrowtail(N,y)$,
  with corresponding types $(m,R_i,\lambda_i)$, such that
  $R_i\uparrow\infty$ and $\lambda_i\downarrow1$ as $i\to\infty$. Let
  us prove that $[M,x]=[N,y]$.

  First, we inductively construct a pointed isometric immersion
  $\psi:(M,x)\to(N,y)$.

  The restrictions $\phi_i:(B_M(x,R_1),x)\to(N,y)$ are pointed
  equi-quasi-isometries of order $m$ ($\lambda_1$ is an equi-dilation
  bound of order $m$). By Proposition~\ref{p: Arzela-Ascoli}, there is
  some subsequence $\phi_{k(1,l)}$ whose restriction to $B_M(x,R_1)$
  converges to some pointed $C^m$ function
  $\psi_1:(B_M(x,R_1),x)\to(N,y)$ in the weak $C^m$ topology. Since
  $\lambda_i\downarrow1$, it follows that $\psi_1$ is an isometric
  immersion.

  Now assume that, for some $i\ge1$, there is some subsequence
  $\phi_{k(i,l)}$ whose restriction to $B_M(x,R_i)$ converges to some
  pointed isometric immersion $\psi_i:(B_M(x,R_i),x)\to(N,y)$.  As
  before, by Proposition~\ref{p: Arzela-Ascoli}, the sequence
  $\phi_{k(i,l)}$ has some subsequence $\phi_{k(i+1,l)}$ whose
  restriction to $B_M(x,R_{i+1})$ converges to some pointed isometric
  immersion $\psi_{i+1}:(B_M(x,R_{i+1}),x)\to(N,y)$ in the weak $C^m$
  topology. Moreover $\psi_{i+1}|_{B_M(x,R_i)}=\psi_i$. Thus the maps
  $\psi_i$ can be combined to define the desired pointed isometric
  immersion $\psi:(M,x)\to(N,y)$.

  Now, let us show that $\psi$ is indeed a pointed isometry, and
  therefore $[M,x]=[N,y]$, as desired. By Lemma~\ref{l: composition
    and inversion of pointed local quasi-isometries}-\eqref{i:
    inversion}, each inverse $\phi_i^{-1}:(N,y)\rightarrowtail(M,x)$ is
  an $(m,R'_i,\lambda_i)$-pointed local quasi-isometry, where
  $R'_i=R_i/\lambda_i\uparrow\infty$. By using Proposition~\ref{p:
    Arzela-Ascoli} as above, we get a subsequence
  $\phi^{-1}_{k'(i,l)}$ of each sequence $\phi^{-1}_{k(i,l)}$, whose
  restriction to $B_N(y,R'_i)$ converges to a pointed isometric
  immersion $\psi'_i:(B_N(y,R'_i),y)\to(M,x)$ in the weak $C^m$
  topology, and such that $\phi^{-1}_{k'(i+1,l)}$ is also a
  subsequence of $\phi^{-1}_{k'(i,l)}$. So
  $\psi'_{i+1}|_{B_N(y,R'_i)}=\psi'_i$ for all $i$, obtaining that the
  maps $\psi'_i$ can be combined to define a pointed isometric
  immersion $\psi':(N,y)\to(M,x)$. Since the operation of composition
  is continuous with respect to the weak $C^m$ topology \cite[p.~64,
  Exercise~10]{Hirsch1976}, we get $\psi_i\psi'_i=\id_{B_N(y,R'_i)}$
  for all $i$, giving $\psi\psi'=\id_N$. Therefore $\psi'$ is
  injective. Moreover $\psi'$ is also surjective because $M$ and $N$
  are complete. Hence $\psi'$ is an isometry whose inverse is $\psi$.
\end{proof}

By Propositions~\ref{p: C^infty uniformity in MM_*(n)} and~\ref{p:
  Hausdorff MM_*(n)}, the sets $U_{R,r}^m$ form a base of entourages
of a separating uniformity on $\MM_*(n)$, which is called the
\emph{$C^\infty$ uniformity}. It will be proved that the induced
topology satisfies the statement of Theorem~\ref{t: C^infty
  convergence in MM_*(n)}; thus it is called the \emph{$C^\infty$
  topology}, and the corresponding space is denoted by
$\MM_*^\infty(n)$. The notation $\Cl_\infty$ and $\Int_\infty$ will be
used for the closure and interior operators in $\MM_*^\infty(n)$.

The following lemma will be used.

\begin{lem}\label{l: bd_U}
  For any open $U\subset\MM_*^\infty(n)$, the map
  $\bd_U:\MM_*^\infty(n)\to[0,\infty]$, defined by
  \[
  \bd_U([M,x])=\inf\{\,d_M(x,x')\mid x'\in M,\ [M,x']\in U\,\}\;,
  \]
  is upper semicontinuous.
\end{lem}

Here, recall that $\inf\emptyset=\infty$ in $\R$.

\begin{proof}
  To prove that $\bd_U$ is upper semicontinuous at some
  $[M,x]\in\MM_*^\infty(n)$, we can assume that
  $D:=\bd_U([M,x])<\infty$. Given any $\epsilon>0$, there is some
  $x'\in B_M(x,D+\epsilon)$ such that $[M,x']\in U$. Since $U$ is
  open, we have $U^m_{R,r}(M,x')\subset U$ for some $m\in\N$ and
  $R,r>0$ with $R\ge D+\epsilon$ and $e^rd_M(x,x')<D+\epsilon$. Given
  any $[N,y]\in U^m_{2R,r}(M,x)$, there is some
  $(m,2R,\lambda)$-pointed local quasi-isometry
  $\phi:(M,x)\rightarrowtail(N,y)$ for some $\lambda\in[1,e^r)$. Take
  some $\delta>0$ such that $\lambda(d_M(x,x')+\delta)<D+\epsilon$,
  and let $\alpha$ be a smooth curve in $B_M(x,D+\epsilon)$ of length
  $<d_M(x,x')+\delta$ from $x$ to $x'$. Hence $\phi\alpha$ is a well
  defined $C^{m+1}$ curve in $N$ from $y$ to $y':=\phi(x')$ of length
  $<\lambda(d_M(x,x')+\delta)<D+\epsilon$, obtaining that
  $d_N(y,y')<D+\epsilon$. On the other hand, $\phi$ is also an
  $(m,R,\lambda)$-pointed local quasi-isometry
  $(M,x')\rightarrowtail(N,y')$, showing that $[N,y']\in
  U^m_{R,r}(M,x')\subset U$. So $\bd_U([N,y])<D+\epsilon$.
\end{proof}

\section{Convergence in the $C^\infty$ topology}

\begin{lem}\label{l: polarization}
  Let $g$ and $g'$ be positive definite scalar products on a real
  vector space $V$, and let $|\ |$ and $|\ |'$ denote the respective
  induced norms on the vector space of tensors over $V$. The following
  properties hold:
  \begin{enumerate}[{\rm(}i{\rm)}]
	
  \item\label{i: polarization} If $\lambda\ge1$ satisfies
    $\frac{1}{\lambda}|v|'\le|v|\le\lambda|v|'$ for all $v\in V$, then
    $|g-g'|\le\lambda^2-\lambda^{-2}$.
		
  \item\label{i: norm} If $|g-g'|\le\epsilon$ for some
    $\epsilon\in[0,1)$, then
    $\sqrt{1-\epsilon}\,|v|\le|v|'\le\sqrt{1+\epsilon}\,|v|$ for all
    $v\in V$.
		
  \item\label{i: |omega|} If $\lambda\ge1$ satisfies
    $\frac{1}{\lambda} |v|'\le|v|\le\lambda|v|'$ for all $v\in V$,
    then
    $\frac{1}{\lambda^2}|\omega|'\le|\omega|\le\lambda^2|\omega|'$ for
    all $\omega\in V^*\otimes V^*$.
		
  \end{enumerate}
\end{lem}

\begin{proof}
  To prove~\eqref{i: polarization}, take arbitrary vectors $v,w\in V$
  with $|v|=|w|=1$. By polarization,
  \begin{multline*}
    (g-g')(v,w)=\frac{1}{4}\left(|v+w|^2-|v-w|^2-|v+w|'^2+|v-w|'^2\right)\\
    \le\frac{1}{4}\left(\left(1-1/\lambda^2\right)|v+w|^2+(\lambda^2-1)|v-w|^2\right)
    \le1-\frac{1}{\lambda^2}+\lambda^2-1=\lambda^2-\frac{1}{\lambda^2}\;.
  \end{multline*}
  Interchanging $g$ and $g'$ in these inequalities, it also follows
  that $|(g-g')(v,w)|\le\lambda^2-\lambda^{-2}$.
	
  Property~\eqref{i: norm} follows because, for any $v\in V$,
  \[
  (1-\epsilon)|v|^2\le|v|^2-||v|^2-|v|'^2|\le
  |v|'^2\le|v|^2+||v|^2-|v|'^2|\le(1+\epsilon)|v|^2\;.
  \]
		
  Let us prove~\eqref{i: |omega|}. For all $v,w\in V\sm\{0\}$,
  \[
  \frac{|\omega(v,w)|}{|v|'\,|w|'}\le\lambda^2\,\frac{|\omega(v,w)|}{|v|\,|w|}\le\lambda^2\,|\omega|\;,
  \]
  obtaining $|\omega|'\le\lambda^2|\omega|$. Interchanging the roles
  of $|\ |$ and $|\ |'$, we also get $|\omega|\le\lambda^2|\omega|'$.
\end{proof}

The following coordinate free description of $C^m$ convergence is a
direct consequence of~\eqref{norm equiv}.

\begin{lem}[Lessa {\cite[Lemma~7.1]{Lessa:Reeb}}]\label{l: Lessa}
  For $m\in\N$, a sequence $[M_i,x_i]\in\MM_*(n)$ is $C^m$ convergent
  to $[M,x]\in\MM_*(n)$ if and only if, for every compact domain
  $\Omega\subset M$ containing $x$, there are pointed $C^{m+1}$
  embeddings $\phi_i\colon (\Omega,x) \to (M_i,x_i)$, for $i$ large
  enough, such that $\|g_M-\phi_i^*g_{M_i}\|_{C^m,\Omega,g_M}\to0$ as
  $i\to \infty$.
\end{lem}
 
\begin{defn}\label{d: D^m_R,r}
  For $R,r>0$ and $m\in\N$, let $D^m_{R,r}$ be the set of pairs
  $([M,x],[N,y])\in\MM_*(n)\times\MM_*(n)$ such that there is some
  $C^{m+1}$ pointed local diffeomorphism $\phi\colon (M,x)
  \rightarrowtail (N,y)$ so that
  $\|g_M-\phi^*g_N\|_{C^m,\Omega,g_M}<r$ for some compact domain
  $\Omega\subset\dom\phi$ with $B_M(x,R)\subset\Omega$.
\end{defn} 

Given a set $X$, for $U\subset X\times X$
     and $x\in X$, let $U(x)=\{\,y\in Y\mid(x,y)\in U\,\}$. In the
     case of $U\subset\MM_*(n)\times\MM_*(n)$ and $[M,x]\in\MM_*(n)$,
     we simply write $U(M,x)$.
 
 \begin{rem}\label{r: D^m_R,r}
   By Lemma~\ref{l: Lessa}, a sequence $[M_i,x_i]\in\MM_*(n)$ is
   $C^\infty$ convergent to $[M,x]\in\MM_*(n)$ if and only if it is
   eventually in $D^m_{R,r}(M,x)$ for arbitrary
   $m\in\N$ and $R,r>0$.
 \end{rem}
 
\begin{prop}\label{p: D^m_R,epsilon(M,x) subset U^m_R,r(M,x)}
  \begin{enumerate}[{\rm(}i{\rm)}]

  \item\label{i: D^0_R,epsilon subset U^0_R,r} For all $R,r>0$, if
    $0<\epsilon\le\min\{1-e^{-2r},e^{2r}-1\}$, then
    $D^0_{R,\epsilon}\subset U^0_{R,r}$.
		
  \item\label{i: D^m_R,epsilon(M,x) subset U^m_R,r(M,x)} For all
    $m\in\Z^+$, $R, r>0$ and $[M,x]\in\MM_*(n)$, there is some
    $\epsilon>0$ such that $D^m_{R,\epsilon}(M,x)\subset
    U^m_{R,r}(M,x)$.

  \end{enumerate}
\end{prop}

\begin{proof}
  Let us show~\eqref{i: D^0_R,epsilon subset U^0_R,r}. If
  $([M,x],[N,y])\in D^0_{R,\epsilon}$, then there is a $C^1$ pointed
  local diffeomorphism $\phi:(M,x)\rightarrowtail(N,y)$ such that
  $\epsilon_0:=\|g_M-\phi^*g_N\|_{C^0,\Omega,g_M}<\epsilon$ for some
  compact domain $\Omega\subset\dom\phi$ with
  $B_M(x,R)\subset\Omega$. Choose some $\lambda\in[1,e^r)$ such that
  $\epsilon_0\le\min\{1-\lambda^{-2},\lambda^2-1\}$. Set $g=g_M$ and
  $g'=\phi^*g_N$, and let $|\ |$ and $|\ |'$ denote the respective
  norms. For $\xi\in T\Omega$, we have
  \[
  \frac{1}{\lambda}\,|\xi|\le\sqrt{1-\epsilon_0}\,|\xi|\le|\xi|'\le\sqrt{1+\epsilon_0}\,|\xi|\le\lambda\,|\xi|
  \]
  by Lemma~\ref{l: polarization}-\eqref{i: norm}. Thus $\phi$ is a
  $(0,R,\lambda)$-pointed local quasi-isometry, obtaining that
  $([M,x],[N,y])\in U^0_{R,r}$.

  Let us prove~\eqref{i: D^m_R,epsilon(M,x) subset U^m_R,r(M,x)}. Take
  $m\in\Z^+$, $R, r>0$ and $[M,x]\in\MM_*(n)$. Let $\UU$ be a finite
  collection of charts of $M$ with domains $U_a$, and let
  $\KK=\{K_a\}$ be a family of compact subsets of $M$, with the same
  index set as $\UU$, such that $K_a\subset U_a$ for all $a$, and $\ol
  B_M(x,R)\subset\Int(K)$ for $K=\bigcup_aK_a$. Let $\epsilon>0$, to
  be fixed later. For any $[N,y]\in D^m_{R,\epsilon}(M,x)$, there is a
  $C^{m+1}$ pointed local diffeomorphism
  $\phi\colon(M,x)\rightarrowtail(N,y)$ so that
  $\|g_M-\phi^*g_N\|_{C^m,\Omega,g_M}<\epsilon$ for some compact
  domain $\Omega\subset\dom\phi\cap\Int(K)$ with
  $B_M(x,R)\subset\Omega$. By continuity, there is another compact
  domain $\Omega'\subset\dom\phi\cap\Int(K)$ such that
  $\Omega\subset\Int(\Omega')$ and
  $\|g_M-\phi^*g_N\|_{C^m,\Omega',g_M}<\epsilon$. As before, let
  $g=g_M$ and $g'=\phi^*g_N$.
	
  With the notation of Section~\ref{ss: Riemannian geom}, let
  $\UU^{(m)}$ be the family of induced charts of $T^{(m)}M$ with
  domains $U_a^{(m)}$, let $\KK^{(m)}$ be the family of compact
  subsets
  \[
  K^{(m)}_a=\{\,\xi\in T^{(m)}M\mid\pi(\xi)\in K_a,\
  d_M^{(m)}(\xi,\pi(\xi))\le R'\,\} \subset U_a^{(m)}\;,
  \]
  for some $R'>R$, where $\pi:T^{(m)}M\to M$, and let
  $K^{(m)}=\bigcup_aK^{(m)}_a$. Since $\ol
  B^{(m)}_M(x,R)\subset\Int(K^{(m)})$ and $\pi(\ol B^{(m)}_M(x,R))=\ol
  B_M(x,R)\subset\Omega\subset\Int(\Omega')$ by Remark~\ref{r: T^(l)M
    subset T^(m)M}-\eqref{i: pi(xi)},\eqref{i: zeta'}, there is some
  compact domain $\Omega^{(m)}\subset T^{(m)}M$ such that
  $B^{(m)}_M(x,R)\subset\Omega^{(m)}\subset K^{(m)}$ and
  $\pi(\Omega^{(m)})\subset\Omega'$.
	
  Choose the following constants:
  \begin{itemize}
		
  \item some $C\ge1$ satisfying~\eqref{norm equiv} with $\UU$, $\KK$,
    $\Omega'$ and $g$;
	
  \item some $C^{(m)}\ge1$ satisfying~\eqref{norm equiv} with
    $\UU^{(m)}$, $\KK^{(m)}$, $\Omega^{(m)}$ and $g^{(m)}$;
	
  \item some $\delta\in(0,\min\{1-e^{-2r},e^{2r}-1\}]$; and,
			
  \item by Lemma~\ref{l: g^(m)_alpha beta}-\eqref{i: g^(m)_alpha
      beta}, some $\epsilon'>0$ such that
    \[
    \|g-g'\|_{C^m,\Omega',\UU,\KK}<\epsilon'\,\Longrightarrow\,
    \|g^{(m)}-g'^{(m)}\|_{C^0,\Omega^{(m)},\UU^{(m)},\KK^{(m)}}<\delta/C^{(m)}\;.
    \]
			
  \end{itemize}
  Suppose that $\epsilon\le\epsilon'/C$. Then
  \begin{multline*}
    \|g-g'\|_{C^m,\Omega',g}<\epsilon
    \,\Longrightarrow\,\|g-g'\|_{C^m,\Omega',\UU,\KK}<C\epsilon\le\epsilon'\\
    \,\Longrightarrow\,\|g^{(m)}-g'^{(m)}\|_{C^0,\Omega^{(m)},\UU^{(m)},\KK^{(m)}}<\delta/C^{(m)}
    \,\Longrightarrow\,\delta_0:=\|g^{(m)}-g'^{(m)}\|_{C^0,\Omega^{(m)},g^{(m)}}<\delta\;.
  \end{multline*}
  For any $\lambda\in[1,e^r)$ such that
  $\delta_0\le\min\{1-\lambda^{-2},\lambda^2-1\}$, we have
  $\frac{1}{\lambda}\,|\xi|^{(m)}\le|\xi|'^{(m)}\le\lambda\,|\xi|^{(m)}$
  for all $\xi\in T\Omega^{(m)}$ by Lemma~\ref{l:
    polarization}-\eqref{i: norm}, where $|\ |^{(m)}$ and $|\
  |'^{(m)}$ denote the norms defined by $g^{(m)}$ and $g'^{(m)}$,
  respectively. So $\phi$ is an $(m,R,\lambda)$-pointed local
  quasi-isometry $(M,x)\rightarrowtail (N,y)$, and therefore $[N,y]\in
  U^{(m)}_{R,r}(M,x)$.
\end{proof}

\begin{prop}\label{p: U^m_R,epsilon(M,x) subset D^m_R,r(M,x)}
  \begin{enumerate}[{\rm(}i{\rm)}]

  \item\label{i: U^0_R,epsilon subset D^0_R,r} For all $R,r>0$, if
    $e^{2\epsilon}-e^{-2\epsilon}\le r$, then $U^0_{R,\epsilon}\subset
    D^0_{R,r}$.
		
  \item\label{i: U^m_R,epsilon(M,x) subset D^m_R,r(M,x)} For all
    $m\in\Z^+$, $R, r>0$ and $[M,x]\in\MM_*(n)$, there is some
    $\epsilon>0$ such that $U^m_{R,\epsilon}(M,x)\subset
    D^m_{R,r}(M,x)$.

  \end{enumerate}
\end{prop}

\begin{proof}
  Let us show~\eqref{i: U^0_R,epsilon subset D^0_R,r}. If
  $([M,x],[N,y])\in U^0_{R,\epsilon}$, then there is a
  $(0,R,\lambda)$-pointed local quasi-isometry
  $\phi:(M,x)\rightarrowtail(N,y)$ for some
  $\lambda\in[1,e^\epsilon)$. Set $g=g_M$ and $g'=\phi^*g_N$, and let
  $|\ |$ and $|\ |'$ denote the respective norms. Thus there is some
  compact domain $\Omega\subset\dom\phi$ such that
  $B_M(x,R)\subset\Omega$ and
  $\frac{1}{\lambda}\,|\xi|\le|\xi|'\le\lambda\,|\xi|$ for all $\xi\in
  T\Omega$. By Lemma~\ref{l: polarization}-\eqref{i: polarization}, it
  follows that
  \[
  \|g-g'\|_{C^0,\Omega,g}\le\lambda^2-\lambda^{-2}<e^{2\epsilon}-e^{-2\epsilon}\le
  r\;.
  \]
  So $([M,x],[N,y])\in D^0_{R,r}$.

  Let us prove~\eqref{i: U^m_R,epsilon(M,x) subset D^m_R,r(M,x)}. Let
  $m\in\Z^+$, $R, r>0$ and $[M,x]\in\MM_*(n)$. Take $\UU$, $\KK$, $K$,
  $\UU^{(m)}$, $\KK^{(m)}$ and $K^{(m)}$ like in the proof of
  Proposition~\ref{p: D^m_R,epsilon(M,x) subset
    U^m_R,r(M,x)}-\eqref{i: D^m_R,epsilon(M,x) subset
    U^m_R,r(M,x)}. Let $\epsilon>0$, to be fixed later. For any
  $[N,y]\in U^m_{R,\epsilon}(M,x)$, there is an
  $(m,R,\lambda)$-pointed local quasi-isometry
  $\phi:(M,x)\rightarrowtail (N,y)$ for some
  $\lambda\in[1,e^\epsilon)$. Again, let $g=g_M$ and
  $g'=\phi^*g_N$. Thus there is a compact domain
  $\Omega^{(m)}\subset\dom\phi_*^{(m)}\cap\Int(K^{(m)})$ so that
  $B_M^{(m)}(x,R)\subset\Omega^{(m)}$ and
  $\frac{1}{\lambda}\,|\xi|^{(m)}\le|\xi|'^{(m)}\le\lambda\,|\xi|^{(m)}$
  for all $\xi\in T\Omega^{(m)}$, where $|\ |^{(m)}$ and $|\ |'^{(m)}$
  denote the norms defined by $g^{(m)}$ and $g'^{(m)}$,
  respectively. By continuity, given any
  $\lambda'\in(\lambda,e^\epsilon)$, there is some compact domain
  $\Omega'^{(m)}\subset\dom\phi_*^{(m)}\cap K^{(m)}$ such that
  $\Omega^{(m)}\subset\Int(\Omega'^{(m)})$ and
  $\frac{1}{\lambda'}\,|\xi|^{(m)}\le|\xi|'^{(m)}\le\lambda'\,|\xi|^{(m)}$
  for all $\xi\in\Omega'^{(m)}$. By Lemma~\ref{l:
    polarization}-\eqref{i: polarization}, it follows that
  \[
  \|g^{(m)}-g'^{(m)}\|_{C^0,\Omega'^{(m)},g^{(m)}}\le\lambda'^2-\lambda'^{-2}
  <e^{2\epsilon}-e^{-2\epsilon}\;.
  \]
  There is some compact domain $\Omega\subset M$ such that
  $\Omega^{(m)}\cap M\subset\Omega\subset\Int(\Omega'^{(m)})$. Thus
  $\Omega\subset\Omega'^{(m)}\cap M\subset K^{(m)}\cap M=K$, and
  \[
  B_M(x,R)= B_M^{(m)}(x,R)\cap M\subset\Omega^{(m)}\cap M\subset\Omega
  \]
  by Remark~\ref{r: T^(l)M subset T^(m)M}-\eqref{i: T^(l)M subset
    T^(m)M}. Take some $C\ge1$ satisfying~\eqref{norm equiv} with
  $\UU$, $\KK$, $\Omega$ and $g$, and some $C^{(m)}\ge1$
  satisfying~\eqref{norm equiv} with $\UU^{(m)}$, $\KK^{(m)}$,
  $\Omega'^{(m)}$ and $g^{(m)}$. For $\rho>0$ and $n+1\le\mu\le2^mn$,
  let $\sigma^{(m)}_{a,\rho,\mu}:U_a\to U_a^{(m)}$ be the section of
  each projection $\pi:U_a^{(m)}\to U_a$ of the type used in
  Lemma~\ref{l: g^(m)_alpha beta}-\eqref{i: partial up to order m of
    g_ij}. Since $\Omega\subset\Int(\Omega'^{(m)})$, there is some
  $\rho>0$ so that $\sigma^{(m)}_{\rho,\mu}(K_a\cap\Omega)\subset
  K_a^{(m)}\cap\Omega'^{(m)}$ for all $a$ and $\mu$. Thus, by
  Lemma~\ref{l: g^(m)_alpha beta}-\eqref{i: partial up to order m of
    g_ij}, there is some $\epsilon'>0$, depending on $r$ and $\rho$,
  such that
  \[
  \|g^{(m)}-g'^{(m)}\|_{C^0,\Omega'^{(m)},\UU^{(m)},\KK^{(m)}}<\epsilon'\,\Longrightarrow\,
  \|g-g'\|_{C^m,\Omega,\UU,\KK}<r/C\;.
  \]
  Suppose that
  $e^{2\epsilon}-e^{-2\epsilon}\le\epsilon'/C^{(m)}$. Then
  \begin{multline*}
    \|g^{(m)}-g'^{(m)}\|_{C^0,\Omega'^{(m)},g^{(m)}}<e^{2\epsilon}-e^{-2\epsilon}
    \,\Longrightarrow\,\|g^{(m)}-g'^{(m)}\|_{C^0,\Omega'^{(m)},\UU^{(m)},\KK^{(m)}}
    <C^{(m)}(e^{2\epsilon}-e^{-2\epsilon})\le\epsilon'\\
    \,\Longrightarrow\,\|g-g'\|_{C^m,\Omega,\UU,\KK}<r/C
    \,\Longrightarrow\,\|g-g'\|_{C^m,\Omega,g}<r\;,
  \end{multline*}
  showing that $[N,y]\in D^{(m)}_{R,r}(M,x)$.
\end{proof}

\begin{cor}\label{c: C^infty convergence in MM_*(n)}
  The $C^\infty$ convergence in $\MM_*(n)$ describes the $C^\infty$
  topology.
\end{cor}

\begin{proof}
  This is a direct consequence of Remark~\ref{r: D^m_R,r} and
  Propositions~\ref{p: D^m_R,epsilon(M,x) subset U^m_R,r(M,x)}
  and~\ref{p: U^m_R,epsilon(M,x) subset D^m_R,r(M,x)}.
\end{proof}

\section{$\MM_*^\infty(n)$ is Polish}\label{s: MM_*^infty(n) is Polish}

\begin{prop}\label{p: MM_*^infty(n) is separable}
$\MM_*^\infty(n)$ is separable.
\end{prop}

\begin{proof}
  The isometry classes of pointed compact Riemannian manifolds form a
  subspace, $\MM_{*,\text{\rm c}}^\infty(n)\subset\MM_*^\infty(n)$,
  which is dense because, for all $[M,x]\in\MM_*^\infty(n)$ and $R>0$,
  the ball $B_M(x,R)$ can be isometrically embedded in a compact
  Riemannian manifold.

  As a consequence of the finiteness theorems of Cheeger on Riemannian
  manifolds \cite{Cheeger1970}, it follows that there are countably
  many diffeomorphism classes of compact $C^\infty$ manifolds (see
  \cite[Corollary~37, p.~320]{Petersen1998} or
  \cite[Theorem~IX.8.1]{Chavel2006}). Thus there is a countable
  family $\mathcal{C}$ of $C^\infty$ compact manifolds containing
  exactly one representative of every diffeomorphism class.

  For every $M\in\CC$, the set of metrics on $M$, $\Met(M)$, is an
  open subspace of the space of smooth sections, $C^\infty(M;
  T^*M\odot T^*M)$, with the $C^\infty$ topology, where ``$\odot$''
  denotes the symmetric product. Then, since $C^\infty(M; T^*M\odot
  T^*M)$ is separable, we can choose a countable dense subset
  $\GG_M\subset\Met(M)$. Choose also a countable dense subset
  $\DD_M\subset M$.

  Clearly, the countable set
  \[
  \{\,[(M,g),x]\mid M\in\mathcal{C},\ g\in\GG_M,\ x\in\DD_M\,\}
  \]
  is dense in $\MM_{*,\text{\rm c}}^\infty(n)$, and therefore it is
  also dense in $\MM_*^\infty(n)$.
\end{proof}

\begin{rem}\label{r: MM_*,c^infty(n) is dense}
	Observe that the proof of Proposition~\ref{p: MM_*^infty(n) is separable} shows that $\MM_{*,\text{\rm c}}^\infty(n)$ is dense in $\MM_*^\infty(n)$.
\end{rem}

\begin{prop}\label{p: MM_*^infty(n) is completely metrizable}
$\MM_{*}^\infty(n)$ is completely metrizable.
\end{prop}

\begin{proof}
  The $C^\infty$ uniformity on $\MM_*(n)$ is metrizable because it is
  separating and has a countable base of entourages
  \cite[Corollary~38.4]{Willard1970}. Thus it is enough to check that
  the $C^\infty$ uniformity on $\MM_*(n)$ is complete.
	
  Consider an arbitrary Cauchy sequence $[M_i,x_i]$ in $\MM_*(n)$ with
  respect to the $C^\infty$ uniformity. We have to prove that
  $[M_i,x_i]$ is convergent in $\MM_*^\infty(n)$. By taking a
  subsequence if necessary, we can suppose that
  $([M_i,x_i],[M_{i+1},x_{i+1}])\in U^{m_i}_{R_i,r_i}$ for sequences,
  $m_i\uparrow\infty$ in $\N$, and $R_i\uparrow\infty$ and
  $r_i\downarrow0$ in $\R^+$, such that $\sum_ir_i<\infty$, and
  $R_{i+1}\ge e^{r_i}R_i$ for all $i$. Let $\bar r_i=\sum_{j\ge
    i}r_j$. Consider other sequences $R'_i,R''_i\uparrow\infty$ in
  $\R^+$ such that $R'_i<R''_i\le e^{-\bar r_i}R_i$ and $R'_{i+1}\ge
  e^{r_i}R''_i$.

  For each $i$, there is some $\lambda_i\in(1,e^{r_i})$ and some
  $(m_i,R_i,\lambda_i)$-pointed local quasi-isometry
  $\phi_i\colon(M_i,x_i)\rightarrowtail(M_{i+1},x_{i+1})$, which can
  be assumed to be $C^\infty$ by Remark~\ref{r:
    (m,R,lambda)-...}-\eqref{i: C^infty approximation of
    (m,R,lambda)-...}. Then $\bar\lambda_i:=\prod_{j\ge
    i}\lambda_j<e^{\bar r_i}<\infty$. For $i<j$, the pointed local
  quasi-isometry
  $\psi_{ij}=\phi_{j-1}\cdots\phi_i:(M_i,x_i)\rightarrowtail(M_j,x_j)$ is
  of type $(m_i,R_i/\bar\lambda_i,\bar\lambda_i)$ by Lemma~\ref{l:
    composition and inversion of pointed local
    quasi-isometries}-\eqref{i: composition}.
	
  For $i,m\in\N$, let
  \begin{alignat*}{3}
    B_i&=B_i(x_i,R_i)\;,&\quad B_i'&=B_i(x_i,R'_i)\;,&\quad B_i''&=B_i(x_i,R''_i)\;,\\
    B_i^{(m)}&=B_i^{(m)}(x_i,R_i)\;,&\quad
    B_i'^{(m)}&=B_i^{(m)}(x_i,R'_i)\;,&\quad
    B_i''^{(m)}&=B_i^{(m_i)}(x_i,R''_i)\;.
  \end{alignat*}
  A bar will be added to this notation when the corresponding closed
  balls are considered. We have $\phi_i(\ol B_i)\subset B_{i+1}$
  because $R_{i+1}>\lambda_iR_i$, and $\phi_{i*}^{(m_i)}(\ol
  B_i''^{(m_i)})\subset B_{i+1}'^{(m_i)}\subset B_{i+1}'^{(m_{i+1})}$
  since $R'_{i+1}>\lambda_iR''_i$ and by Remark~\ref{r: T^(l)M subset
    T^(m)M}-\eqref{i: g^(m)|_T^(l)M = g^(l)}. Furthermore
  $B_i''\subset\dom\psi_{ij}$ and
  $B_i''^{(m_i)}\subset\dom\psi_{ij*}^{(m_i)}$ for $i<j$ because
  $R''\le R_i/\bar\lambda_i$. Therefore $\psi_{ij}(B_i)\subset B_j$
  and $\psi_{ij*}^{(m_i)}(B_i''^{(m_i)})\subset B_j'^{(m_j)}$.
	
  The restrictions $\psi_{ij}:B_i\to B_j$ form a direct system of
  spaces, whose direct limit is denoted by $\widehat M$. Let
  $\psi_i:B_i\to\widehat M$ be the induced maps, whose images,
  $\widehat B_i:=\psi_i(B_i)$, form an exhausting increasing sequence
  of subsets of $\widehat M$. All points $\psi_i(x_i)$ are equal in
  $\widehat M$, and will be denoted by $\hat x$. The space $\widehat
  M$ is connected because it is the union of the connected subspaces
  $\widehat B_i$ whose intersection contains $\hat x$. By the
  definition of the direct limit and since the maps $\psi_{ij}$ are
  open embeddings, it follows that all maps $\psi_i$ are open
  embeddings, and therefore $\widehat M$ is a Hausdorff
  $n$-manifold. Equip each $\widehat B_i$ with the $C^\infty$
  structure that corresponds to the $C^\infty$ structure of $B_i$ by
  $\psi_i$. These $C^\infty$ structures are compatible one another
  because the open embeddings $\psi_{ij}$ are $C^\infty$, and
  therefore they define a $C^\infty$ structure on $\widehat
  M$. Moreover let $\hat g_i$ be the Riemannian metric on each
  $\widehat B_i$ that corresponds to $g_i|_{B_i}$ via $\psi_i$.
	
  Take some compact domains, $\Omega_i$ in every $M_i$ and
  $\Omega_i^{(m_i)}$ in $T^{(m_i)}M_i$, such that
  $B_i'\subset\Omega_i\subset\Int(\Omega_i^{(m_i)})$ and
  $B_i'^{(m_i)}\subset\Omega_i^{(m_i)}\subset B_i''^{(m_i)}$; thus
  $\Omega_i\subset B_i''$ by Remark~\ref{r: T^(l)M subset
    T^(m)M}-\eqref{i: T^(l)M subset T^(m)M}. Let
  $\widehat\Omega_i=\psi_i(\Omega_i)$.
	
  \begin{claim}\label{cl: widehat M = bigcup_i widehat Omega_i}
    $\widehat M=\bigcup_i\widehat\Omega_i$.
  \end{claim}
	
  This equality holds because, for each $i$, there is some $j$ so that
  $R'_j>\bar\lambda_iR_i$, obtaining
  \[
  \psi_{ij}(B_i)\subset B_j(x_j,\bar\lambda_iR_i)\subset
  B_j'\subset\Omega_j\;,
  \]
  and therefore $\widehat
  B_i=\psi_j\psi_{ij}(B_i)\subset\psi_j(\Omega_j)=\widehat\Omega_j$.
	
  \begin{claim}\label{cl: hat g_j}
    For all $i$, the restrictions $\hat g_j|_{\widehat\Omega_i}$, with
    $j\ge i$, form a convergent sequence in the space of $C^{m_i}$
    sections, $C^{m_i}(\widehat\Omega_i;T\widehat\Omega_i^*\odot
    T\widehat\Omega_i^*)$, with the $C^{m_i}$ topology, and its limit,
    $\hat g_{i,\infty}$, is positive definite at every point.
  \end{claim}
	
  Clearly, Claim~\ref{cl: hat g_j} follows by showing that the
  restrictions of the metrics $g_{ij}:=\psi_{ij}^*g_j$ to $\Omega_i$,
  for $j\ge i$, form a convergent sequence in
  $C^{m_i}(\Omega_i;T\Omega_i^*\odot T\Omega_i^*)$, and its limit,
  $g_{i,\infty}$, is positive definite at every point. To begin with, let us
  show that $g_{ij}|_{\Omega_i}$ is a Cauchy sequence with respect to
  $\|\ \|_{C^{m_i},\Omega_i,g_i}$.
	
  We have
  \begin{equation}\label{|g_i^(m_i) - g_ij^(m_i)|}
    \frac{1}{\bar\lambda_i}\,|\xi|_i^{(m_i)}\le|\xi|_{ij}^{(m_i)}\le\bar\lambda_i\,|\xi|_i^{(m_i)}
  \end{equation}
  for all $\xi\in T\Omega_i^{(m_i)}$, where $|\ |_i^{(m_i)}$ and $|\
  |_{ij}^{(m_i)}$ are the norms defined by $g_i^{(m_i)}$ and
  $g_{ij}^{(m_i)}$, respectively. By Lemma~\ref{l:
    polarization}-\eqref{i: polarization}, it follows that
  \[
  \|g_i^{(m_i)}-g_{ij}^{(m_i)}\|_{C^0,\Omega_i^{(m_i)},g_i^{(m_i)}}
  \le{\bar\lambda}_i^2-{\bar\lambda_i}^{-2}\;.
  \]
  Then, for $k\ge j$,
  \begin{multline}\label{|g_ij^(m_i) -
      g_ik^(m_i)|_C^0,Omega_i^(m_i),g_ij^(m_i)}
    \|g_{ij}^{(m_i)}-g_{ik}^{(m_i)}\|_{C^0,\Omega_i^{(m_i)},g_{ij}^{(m_i)}}
    =\|g_j^{(m_i)}-g_{jk}^{(m_i)}\|_{C^0,\psi_{ij*}^{(m_i)}(\Omega_i^{(m_i)}),g_j^{(m_i)}}\\
    \le\|g_j^{(m_j)}-g_{jk}^{(m_j)}\|_{C^0,\Omega_j^{(m_j)},g_j^{(m_j)}}
    \le\bar\lambda_j^2-\bar\lambda_j^{-2}
  \end{multline}
  because
  \[
  \psi_{ij*}^{(m_i)}(\Omega_i^{(m_i)})\subset\psi_{ij*}^{(m_i)}(B_i''^{(m_i)})
  \subset B_j'^{(m_j)}\subset\Omega_j^{(m_j)}
  \]
  and $g_{jk}^{(m_j)}=g_{jk}^{(m_i)}$ on $\Omega_j^{(m_j)}\cap
  B_j^{(m_i)}\supset\psi_{ij*}^{(m_i)}(\Omega_i^{(m_i)})$
  (Remark~\ref{r: T^(l)M subset T^(m)M}-\eqref{i: g^(m)|_T^(l)M =
    g^(l)}). We get
  \begin{equation}\label{|g_ij^(m_i) -
      g_ik^(m_i)|_C^0,Omega_i^(m_i),g_i^(m_i)}
    \|g_{ij}^{(m_i)}-g_{ik}^{(m_i)}\|_{C^0,\Omega_i^{(m_i)},g_i^{(m_i)}}
    \le\bar\lambda_i^2(\bar\lambda_j^2-\bar\lambda_j^{-2})
  \end{equation}
  by~\eqref{|g_i^(m_i) - g_ij^(m_i)|},~\eqref{|g_ij^(m_i) -
    g_ik^(m_i)|_C^0,Omega_i^(m_i),g_ij^(m_i)} and Lemma~\ref{l:
    polarization}-\eqref{i: |omega|}.
	
  Let $\UU_i$ be a finite collection of charts of $M_i$ with domains
  $U_{i,a}$, and let $\KK_i=\{K_{i,a}\}$ be a family of compact
  subsets of $M_i$, with the same index set as $\UU_i$, such that
  $K_{i,a}\subset U_{i,a}$ for all $a$, and $\ol
  B_i''\subset\bigcup_aK_{i,a}=:K_i$. Thus $\Omega_i\subset K_i$. With
  the notation of Section~\ref{ss: Riemannian geom}, let
  $\UU_i^{(m_i)}$ be the family of induced charts of $T^{(m_i)}M_i$
  with domains $U_{i,a}^{(m_i)}$. Like in the proof of
  Proposition~\ref{p: D^m_R,epsilon(M,x) subset
    U^m_R,r(M,x)}-\eqref{i: D^m_R,epsilon(M,x) subset U^m_R,r(M,x)},
  let $\KK_i^{(m_i)}$ be the family of compact subsets
  \[
  K_{i,a}^{(m_i)}=\{\,\xi\in B_i^{(m_i)}\mid\pi(\xi)\in K_{i,a},\
  d_i^{(m_i)}(\xi,\pi_i(\xi))\le R'''_i\,\} \subset U_{i,a}^{(m_i)}\;,
  \]
  for some $R'''_i>R''_i$, where $\pi:B_i^{(m_i)}\to B_i$. We have
  $B_i''^{(m_i)}\subset\bigcup_aK^{(m_i)}_{i,a}=:K_i^{(m_i)}$. Hence
  $\Omega_i^{(m_i)}\subset K_i^{(m_i)}$.
	
  Choose some $C_i\ge1$ satisfying~\eqref{norm equiv} with $\UU_i$,
  $\KK_i$, $\Omega_i$ and $g_i$, and some $C_i^{(m_i)}\ge1$
  satisfying~\eqref{norm equiv} with $\UU_i^{(m_i)}$, $\KK_i^{(m_i)}$,
  $\Omega_i^{(m_i)}$ and $g^{(m_i)}$. For any $\rho>0$ and
  $n+1\le\mu\le2^{m_i}n$, let $\sigma^{(m_i)}_{i,a,\rho,\mu}:U_{i,a}\to
  U_{i,a}^{(m_i)}$ be the section of each projection
  $\pi:U_{i,a}^{(m_i)}\to U_{i,a}$ of the type used in Lemma~\ref{l:
    g^(m)_alpha beta}-\eqref{i: partial up to order m of g_ij}. Since
  $\Omega_i\subset\Int(\Omega_i^{(m_i)})$, there is some $\rho>0$ so
  that $\sigma^{(m_i)}_{i,a,\rho,\mu}(K_{i,a}\cap\Omega_i)\subset
  K_{i,a}^{(m_i)}\cap\Omega_i^{(m_i)}$ for all $a$ and $\mu$. Thus, by
  Lemma~\ref{l: g^(m)_alpha beta}-\eqref{i: partial up to order m of
    g_ij}, given any $\epsilon>0$, there is some $\delta>0$, depending
  on $\epsilon$ and $\rho$, such that
  \begin{equation}\label{... < delta Longrightarrow ... < epsilon/C_i}
    \|g_{ij}^{(m_i)}-g_{ik}^{(m_i)}\|_{C^0,\Omega_i^{(m_i)},\UU_i^{(m_i)},\KK_i^{(m_i)}}<\delta
    \,\Longrightarrow\,\|g_{ij}-g_{ik}\|_{C^{m_i},\Omega_i,\UU_i,\KK_i}<\epsilon/C_i\;.
  \end{equation}
  Since $\bar\lambda_j\downarrow1$, we have
  $\bar\lambda_i^2(\bar\lambda_j^2-\bar\lambda_j^{-2})<\delta/C_i^{(m_i)}$
  for $j$ large enough, giving
  \begin{multline*}
    \|g_{ij}^{(m_i)}-g_{ik}^{(m_i)}\|_{C^0,\Omega_i^{(m_i)},g_i^{(m_i)}}<\delta/C_i^{(m_i)}
    \,\Longrightarrow\,\|g_{ij}^{(m_i)}-g_{ik}^{(m_i)}\|_{C^0,\Omega_i^{(m_i)},\UU_i^{(m_i)},\KK_i^{(m_i)}}<\delta\\
    \,\Longrightarrow\,\|g_{ij}-g_{ik}\|_{C^{m_i},\Omega_i,\UU_i,\KK_i}<\epsilon/C_i
    \,\Longrightarrow\,\|g_{ij}-g_{ik}\|_{C^{m_i},\Omega_i,g_i}<\epsilon
  \end{multline*}
  by~\eqref{|g_ij^(m_i) -
    g_ik^(m_i)|_C^0,Omega_i^(m_i),g_i^(m_i)},~\eqref{... < delta
    Longrightarrow ... < epsilon/C_i} and~\eqref{norm equiv}. This
  shows that $g_{ij}|_{\Omega_i}$ is a Cauchy sequence in the Banach
  space $C^{m_i}(\Omega_i;T\Omega_i^*\odot T\Omega_i^*)$ with $\|\
  \|_{C^{m_i},\Omega_i,g_i}$, and therefore it has a limit
  $g_{i,\infty}$. For all nonzero $\xi\in T\Omega_i$, we have
  \[
  g_{i,\infty}(\xi,\xi)=\lim_jg_{ij}(\xi,\xi)\ge\frac{1}{\bar\lambda_i}\,g_i(\xi,\xi)>0\;,
  \]
  obtaining that $g_{i,\infty}$ is positive definite. This completes the
  proof of Claim~\ref{cl: hat g_j}.
	
  According to Claim~\ref{cl: hat g_j}, each $\hat g_{i,\infty}$ is a
  $C^{m_i}$ Riemannian metric on $\widehat\Omega_i$, and, obviously,
  $\hat g_{j,\infty}|_{\widehat\Omega_i}=\hat g_{i,\infty}$ for $j>i$. Hence the
  metric tensors $\hat g_{i,\infty}$ can be combined to define a $C^\infty$
  Riemannian metric $\hat g$ on $\widehat M$ by Claim~\ref{cl: widehat
    M = bigcup_i widehat Omega_i}.
	
  Let $|\ |_{i,\infty}^{(m_i)}$ be the norm defined by $g_{i,\infty}^{(m_i)}$ on
  $T\Omega_i^{(m_i)}$. By~\eqref{|g_i^(m_i) - g_ij^(m_i)|} and because
  $|\ |_{i,\infty}^{(m_i)}=\lim_j|\ |_{ij}^{(m_i)}$ on $T\Omega_i^{(m_i)}$,
  we get
  $\frac{1}{\bar\lambda_i}\,|\xi|_i^{(m_i)}\le|\xi|_{i,\infty}^{(m_i)}\le\bar\lambda_i\,|\xi|_i^{(m_i)}$
  for all $\xi\in T\Omega_i^{(m_i)}$. Thus, by Remark~\ref{r:
    quasi-isometries}-\eqref{i: quasi-isometry and distances},
  $\Omega_i$ contains the $g_{i,\infty}$-ball of center
  $x_i$ and radius $R'_i/\bar\lambda_i$ because it contains $B_i'$; in
  particular, $\widehat M$ is complete because
  $R'_i/\bar\lambda_i\to\infty$ and every $\Omega_i$ is compact. Since
  $g_{i,\infty}=\psi_i^*\hat g$, it also follows that 
  $\psi_{i*}^{(m_i)}:\Omega_i^{(m_i)}\to T^{(m_i)}\widehat M$ is a 
  $\bar\lambda_i$-quasi-isometry. So $\psi_i:(M_i,x_i)\rightarrowtail(\widehat M,\hat x)$ 
  is an $(m_i,R'_i,\bar\lambda_i)$-pointed local quasi-isometry, obtaining that 
  $([M_i,x_i],[\widehat M,\hat x])\in U^{m_i}_{R'_i,s_i}$ for any sequence
  $s_i\downarrow0$ with $\bar\lambda_i<e^{s_i}$, and therefore
  $[M_i,x_i]\to[\widehat M,\hat x]$ as $i\to\infty$ in $\MM_*^\infty(n)$.
\end{proof}

\begin{cor}\label{c: MM_*^infty(n) is Polish}
  $\MM_*^\infty(n)$ is Polish.
\end{cor}

\begin{proof}
  This is the content of Propositions~\ref{p: MM_*^infty(n) is
    separable} and~\ref{p: MM_*^infty(n) is completely metrizable}
  together.
\end{proof}
 
Corollaries~\ref{c: C^infty convergence in MM_*(n)} and~\ref{c:
  MM_*^infty(n) is Polish} give Theorem~\ref{t: C^infty convergence in
  MM_*(n)}.

\section{Some basic properties of $\MM_{*,\text{\rm lnp}}^\infty(n)$}\label{s:lemmas}

For each closed $C^\infty$ manifold $M$ of dimension $\ge2$, the
non-periodic metrics on $M$ form a residual subset of $\Met(M)$ with
the $C^\infty$ topology \cite[Corollary~3.5]{BandoUrakawa1983},
\cite[Proposition~1]{Sunada1985}. Then, since $\MM_{*,\text{\rm
    c}}^\infty(n)$ is dense in $\MM_*^\infty(n)$ (Remark~\ref{r: MM_*,c^infty(n) is dense}), it follows that
$\MM_{*,\text{\rm np}}^\infty(n)$ is dense in $\MM_*^\infty(n)$, and
therefore $\MM_{*,\text{\rm lnp}}^\infty(n)$ is dense in
$\MM_*^\infty(n)$ too. On the other hand, $\MM_{*,\text{\rm
    lnp}}^\infty(n)$ is $G_\delta$ in $\MM_*^\infty(n)$ by Lemmas~\ref{l:
  LL} and~\ref{l: [M i,y i]} below, and therefore it is a Polish subspace \cite[Theorem~I.3.11]{Kechris1995}. This proves Theorem~\ref{t: FF_*,lnp(n)}-\eqref{i: MM_*,lnp(n)
  is open and dense}.

\begin{lem}\label{l: LL} 
  For every $n\in\Z^+$ and $[M,x]\in\MM_{*,\text{\rm lnp}}^\infty(n)$,
  there is some $r>0$ such that, if
  \[
  \{\,h\in\Iso(M)\mid h(x)\in\ol B(x,r)\,\}=\{\id_M\}\;,
  \]
  then there is some neighborhood $\LL$ of $[M,x]$ in
  $\MM_{*,\text{\rm lnp}}^\infty(n)$ so that
  \[
  \{\,h\in\Iso(L)\mid h(y)\in\ol B(y,r)\,\}=\{\id_L\}
  \]
  for all $[L,y]\in\LL$.
\end{lem}

\begin{proof} 
  Suppose that the statement is false. Then there is some convergent
  sequence, $[M_i,x_i]\to[M,x]$, in $\MM_*^\infty(n)$ so that, for
  each $i$, some $h_i\in\Iso(M_i)\sm\{\id_{M_i}\}$ satisfies
  $h_i(x_i)\in\ol B_i(x_i,r)$. Choose any sequence of compact domains
  $\Omega_q$ of $M$ such that $\ol B(x,2r)\subset\Int(\Omega_q)$ and
  $d(x,\partial\Omega_q)\to\infty$ as $q\to\infty$. For each $q$ and
  $i$ large enough, there is some pointed smooth embedding
  $\phi_{q,i}:(\Omega_q,x)\to(M_i,x_i)$ so that $\phi_{q,i}^*g_i\to
  g|_{\Omega_q}$ as $i\to\infty$ with respect to the $C^\infty$
  topology. Thus $\ol B_i(x_i,2r)\subset\phi_{q,i}(\Int(\Omega_q))$
  for $i$ large enough.

  \begin{claim}\label{cl: delta}
    If $r$ is small enough, we can assume that there is some
    $\delta>0$ such that, for $i$ large enough, the maps $h_i$ can be
    chosen so that $d_i(z_i,h_i(z_i))\ge\delta$ for some $z_i\in
    B_i(x_i,r)$.
  \end{claim}

  Given any index $i$, suppose first that there is some
  $k\in\Z\sm\{0\}$ such that $h_i^k(x_i)\not\in\ol B_i(x_i,r/2)$. Then
  there is some $k\in\Z\sm\{0\}$ such that $h_i^k(x_i)\not\in\ol
  B_i(x_i,r/2)$ and $h_i^\ell(x_i)\in\ol B_i(x_i,r/2)$ if
  $|\ell|<|k|$. If $k=1$, then $d_i(x_i,h_i(x_i))\ge r/2$. If $k=-1$,
  then
  \[
  d_i(x_i,h_i(x_i))=d_i\left(h_i^{-1}(x_i),x_i\right)\ge r/2
  \]
  as well. If $|k|\ge2$, then there is some $\ell\in\Z$ such that
  $|\ell|,|k-\ell|<|k|$. Hence
  \[
  d_i\left(x_i,h_i^k(x_i)\right) \le
  d_i\left(x_i,h_i^\ell(x_i)\right)+d_i\left(h_i^\ell(x_i),h_i^k(x_i)\right)
  =d_i\left(x_i,h_i^\ell(x_i)\right)+d_i\left(x_i,h_i^{k-\ell}(x_i)\right)\le
  r\;.
  \]
  Therefore, by using $h_i^k$ instead of $h_i$, we can assume that
  $d_i(x_i,h_i(x_i))\ge r/2$ in this case.

  Now, suppose that $h_i^k(x_i)\in\ol B_i(x_i,r/2)$ for all
  $k\in\Z$. Consider the non-trivial abelian subgroup
  $A_i=\ol{\{\,h_i^k\mid k\in\Z\,\}}\subset\Iso(M_i)$. Since
  $a(x_i)\in\ol B_i(x_i,r/2)$ for any $a\in A_i$, it follows that
  $A_i$ is compact in the $C^\infty$ topology by Proposition~\ref{p:
    Arzela-Ascoli}, and thus $A_i$ is a non-trivial compact abelian
  Lie subgroup of $\Iso(M_i)$. Let $\mu_i$ be a bi-invariant
  probability measure on $A_i$, and let $f_i:A_i\to M$ be the mass
  distribution defined by $f_i(a)=a(x_i)$. By the $C^\infty$
  convergence $\phi_{q,i}^*g_i\to g|_{\Omega_q}$, we can suppose that
  $r$ is so small that the ball $B_i(x_i,2r/3)$ of $M_i$ satisfies the
  conditions of Proposition~\ref{p: Karcher} for $i$ large
  enough. Then, since $f_i(A_i)\subset\ol B_i(x_i,r/2)\subset
  B_i(x_i,2r/3)$, the center of mass $y_i=\CC_{f_i}$ is defined in
  $B_i(x_i,2r/3)$. Moreover $y_i$ is a fixed point of the canonical
  action of $A_i$ on $M$ \cite[Section~2.1]{Karcher1977}. Since there
  is a neighborhood of the identity in the orthogonal group $\Or(n)$
  which contains no non-trivial subgroup (simply because $\Or(n)$ is a
  Lie group), it follows that there is some $K>0$ such that, for any
  non-trivial subgroup $A\subset\Or(n)$, there is some $a\in A$ and
  some $v\in\R^n$ such that $|v|=1$ and $|a(v)-v|\ge K$. In our
  setting, the subgroup $\{\,a_{*y_i}\mid a\in A_i\,\}$ of the
  orthogonal group $\Or(T_{y_i}M_i)\equiv\Or(n)$ is non-trivial
  because $M_i$ is connected and $A_i$ is non-trivial. Hence there is
  some $a_i\in A_i$ and some $\xi_i\in T_{y_i}M_i$ such that
  $|\xi_i|=1$ and $|a_{i*}(\xi_i)-\xi_i|\ge K$. By the $C^\infty$
  convergence $\phi_{q,i}^*g_i\to g|_{\Omega_q}$, we can also assume
  that $r$ is so small that there exists some $C\ge1$ such that
  $\exp_{y_i}:B(0_{y_i},r)\to B(y_i,r)$ is $C$-quasi-isometric for $i$
  large enough. Then, for $z_i=\exp_{y_i}(\frac{r}{3}\,\xi_i)\in\ol
  B_i(y_i,r/3)\subset B_i(x_i,r)$, we get
  \[
  d_i(z_i,a_i(z_i))\ge\frac{r}{3C}\,|\xi_i-h'_{i*}(\xi_i)|\ge\frac{rK}{3C}\;.
  \]
  Thus, by using $a_i$ instead of $h_i$, we can assume in this case
  that $d_i(z_i,h_i(z_i))\ge rK/3C$. Therefore Claim~\ref{cl: delta}
  follows with $\delta=\min\{r/2,rK/3C\}$.

  For each $q$, we can assume that
  \[
  \ol B(x,\diam(\Omega_q)+r)\subset\Int(\Omega_{q+1})\;,
  \]
  obtaining
  \[
  \ol
  B_i(x_i,\diam(\phi_{q,i}(\Omega_q))+r)\subset\Int(\phi_{q+1,i}(\Omega_{q+1}))
  \]
  for all $i$ large enough by the $C^\infty$ convergence
  $\phi_{q,i}^*g_i\to g|_{\Omega_q}$. Then
  $h'_{q,i}:=\phi_{q+1,i}^{-1}\,h_i\,\phi_{q,i}:\Omega_q\to M$ is well
  defined for each $q$ and all $i$ large enough because
  $x_i\in\phi_{q,i}(\Omega_q)$ and $h_i(x_i)\in\ol B_i(x_i,r)$. On the
  one hand, from the $C^\infty$ convergence $\phi_{q,i}^*g_i\to
  g|_{\Omega_q}$ and since $h_i(x_i)\in\ol B_i(x_i,r)$, we get the
  $C^\infty$ convergence $h_{q,i}^{\prime*}g\to g|_{\Omega_q}$ and
  $\limsup_id(x,h'_{q,i}(x))\le r$; in particular, for each $q$, the
  maps $h'_{q,i}$ are equi-quasi-isometries of order
  $\infty$. Therefore, by Proposition~\ref{p: Arzela-Ascoli}, some
  subsequence of $h'_{q,i}$ is $C^\infty$ convergent to some
  $C^\infty$ map $h'_q:\Omega_q\to M$, which is an isometric embedding
  satisfying $h'_q(x)\in\ol B(x,r)$.

  For all $p\ge q$, the restrictions $h'_p|_{\Omega_q}$ form a
  sequence of isometric embeddings satisfying $h'_p(x)\in\ol
  B(x,r)$. Then, by Proposition~\ref{p: Arzela-Ascoli}, there is some
  sequence of positive integers $p(q,k)$ for each $q$ so that the
  subsequence $h'_{p(q,k)}|_{\Omega_q}$ of $h'_p|_{\Omega_q}$ is
  $C^\infty$ convergent as $k\to\infty$ to an isometric embedding
  $h''_q:\Omega_q\to M$ satisfying $h''_q(x)\in\ol B(x,r)$. We can
  assume that $p(q+1,k)$ is a subsequence of $p(q,k)$ for each $q$,
  yielding $h''_{q+1}|_{\Omega_q}=h''_q$. So the maps $h''_q$ can be
  combined to define an isometry $h:M\to M$ satisfying $h(x)\in\ol
  B(x,r)$.

  Now, fix any $q$ and let $z'_{p,i}=\phi_{p,i}^{-1}(z_i)$ for each
  $p\ge q$ and all $i$ large enough. From $z_i\in B_i(x_i,r)$ and the
  $C^\infty$ convergence $\phi_{p,i}^*g_i\to g|_{\Omega_p}$, it
  follows that $z'_{p,i}$ approaches the compact set $\ol B(x,r)$ as
  $i\to\infty$. Then, for each $p\ge q$, there is a sequence $z_{p,i}$
  in $B(x,r)$ so that $d(z_{p,i},z'_{p,i})\to0$.  Hence, by the
  $C^\infty$ convergence $\phi_{p,i}^*g_i\to g|_{\Omega_p}$ and
  Claim~\ref{cl: delta}, we get
  \begin{multline*}
    \sup\{\,d(z,h(z))\mid z\in B(x,r)\,\}=\sup\{\,d(z,h''_q(z))\mid z\in B(x,r)\,\}\\
    \ge\sup\left\{\liminf_pd(z,h'_p(z))\mid z\in B(x,r)\right\}
    \ge\sup\left\{\liminf_p\liminf_id(z,h'_{p,i}(z))\mid z\in B(x,r)\right\}\\
    \ge\liminf_p\liminf_id(z_{p,i},h'_{p,i}(z_{p,i}))
    =\liminf_p\liminf_id(z'_{p,i},h'_{p,i}(z'_{p,i}))\ge\liminf_id_i(z_i,h_i(z_i))\ge\delta\;.
  \end{multline*}
  So $h\ne\id_M$, which is a contradiction because $h(x)\in\ol
  B(x,r)$.
\end{proof}

\begin{lem}\label{l: NN} 
  For $n\ge2$ and each point $[M,x]\in\MM_{*,\text{\rm
      lnp}}^\infty(n)$, there is some $r>0$ such that, for each
  $\epsilon\in(0,r)$, there is some neighborhood $\NN$ of $[M,x]$ in
  $\MM_{*,\text{\rm lnp}}^\infty(n)$ so that, if an equivalence class
  $\iota(L)$ of $\MM_{*,\text{\rm lnp}}^\infty(n)$ meets $\NN$ at
  points $[L,y]$ and $[L,z]$, then either $d_L(y,z)<\epsilon$ or
  $d_L(y,z)>r$.
\end{lem}

\begin{proof} 
  Since $M$ is locally non-periodic, there is some $r>0$ such that
  \begin{equation}\label{e:... le r}
    \{\,h\in\Iso(M)\mid d(x,h(x))\le r\,\}=\{\id_M\}\;.
  \end{equation} Suppose that the statement is false for this $r$. Then,
  given any $\epsilon\in(0,r)$, there are sequences $[L_i,y_i]$ and
  $[L_i,z_i]$ in $\MM_{*,\text{\rm lnp}}^\infty(n)$ converging to $[M,x]$
  in $\MM_{*,\text{\rm lnp}}^\infty(n)$ such that $\epsilon\le
  d_i(y_i,z_i)\le r$ for all $i$.

  Take a sequence of compact domains $\Omega_q$ of $M$ such that
  $x\in\Omega_q$ and $d(x,\partial\Omega_q)\to\infty$ as
  $q\to\infty$. For each $q$, there are $C^\infty$ embeddings
  $\phi_{q,i}:\Omega_q\to M_i$ and $\psi_{q,i}:\Omega_q\to M_i$ for
  $i$ large enough so that $\phi_{q,i}(x)=y_i$, $\psi_{q,i}(x)=z_i$,
  and $\phi_{q,i}^*g_i,\psi_{q,i}^*g_i\to g|_{\Omega_q}$ as
  $i\to\infty$ with respect to the $C^\infty$ topology. We can also
  assume that, for each $q$,
  \[
  \ol B(x,\diam(\Omega_q)+r)\subset\Int(\Omega_{q+1})\;,
  \]
  giving
  \[
  \phi_{q,i}(\Omega_q)\subset\ol B_i(y_i,\diam(\phi_{q,i}(\Omega_q)))
  \subset\ol B_i(z_i,\diam(\phi_{q,i}(\Omega_q))+r)
  \subset\Int(\psi_{q+1,i}(\Omega_{q+1}))
  \]
  for $i$ large enough by the $C^\infty$ convergence
  $\phi_{q,i}^*g_i,\psi_{q,i}^*g_i\to g|_{\Omega_q}$ and since
  $d_i(y_i,z_i)\le r$. So
  $h_{q,i}:=\psi_{q+1,i}^{-1}\,\phi_q:\Omega_q\to M$ is well defined
  for each $q$ and all $i$ large enough.  From the $C^\infty$
  convergence $\phi_{q,i}^*g_i,\psi_{q,i}^*g_i\to g|_{\Omega_q}$, we
  also get the $C^\infty$ convergence $h_{q,i}^*g\to g|_{\Omega_q}$,
  and moreover
  \[
  \liminf_id(x,h_{q,i}(x))\ge\epsilon\;,\qquad\limsup_id(x,h_{q,i}(x))\le
  r\;,
  \]
  because $\phi_{q,i}(x)=y_i$, $\psi_{q,i}(x)=z_i$ and $\epsilon\le
  d_i(y_i,z_i)\le r$. Then, like in the proof of Lemma~\ref{l: LL}, an
  isometry $h:M\to M$ can be constructed so that $\epsilon\le
  d(x,h(x))\le r$, which contradicts~\eqref{e:... le r}.
\end{proof}

\begin{lem}\label{l: [M i,y i]} 
  Let $n\in\N$ and $r>0$. For any convergent sequence
  $[M_i,x_i]\to[M,x]$ in $\MM_*^\infty(n)$ and each $y\in B(x,r)$,
  there are points $y_i\in B_i(x_i,r)$ such that $[M_i,y_i]\to[M,y]$
  in $\MM_*^\infty(n)$.
\end{lem}

\begin{proof} 
  Take a sequence of compact domains $\Omega_q$ of $M$ such that
  $x,y\in\Omega_q$ and $d(x,\partial\Omega_q)\to\infty$ as
  $q\to\infty$. For each $q$, there is some index $i_q$ such that, for
  each $i\ge i_q$ there is a $C^\infty$ embedding
  $\phi_{q,i}:\Omega_q\to M_i$ satisfying $\phi_{q,i}(x)=x_i$ and
  $\phi_{q,i}^*g_i\to g|_{\Omega_q}$ as $i\to\infty$ with respect to
  the $C^\infty$ topology. Let $y_{q,i}=\phi_{q,i}(y)$ for all $i\ge
  i_q$. Then, for each $q$ and every $m\in\Z^+$, there is some index
  $i_{q,m}\ge i_q$ such that $d_i(x_i,y_{q,i})<r$ and
  $\|\phi_{q,i}^*g_i-g\|_{C^m,\Omega_q,g}<1/m$ for all $i\ge
  i_{q,m}$. Moreover we can assume that $i_{q,q}<i_{q+1,q+1}$ for all
  $q$.  Now, let $y_i$ be any point of $B_i(x_i,r)$ for $i<i_{0,0}$,
  and let $y_i=y_{q,i}$ for $i_{q,q}\le i<i_{q+1,q+1}$. Let us check
  that $[M_i,y_i]\to[M,y]$ in $\MM_*^\infty(n)$. Fix any compact
  domain $\Omega$ of $M$ containing $y$, and let $m\in\N$. We have
  $d(y,\partial\Omega_q)\to\infty$ as $q\to\infty$ because
  $d(x,\partial\Omega_q)\to\infty$ and $d(x,y)<r$. So there is some
  $q_0\ge m$ such that $\Omega\subset\Omega_q$ for all $q\ge q_0$. For
  $i\ge i_{q_0,q_0}$, let $\phi_i=\phi_{q,i}|_\Omega$ if $i_{q,q}\le
  i<i_{q+1,q+1}$ with $q\ge q_0$. Then $\phi_i(y)=y_i$ and
  \[
  \|\phi_i^*g_i-g\|_{C^m,\Omega_q,g}\le\|\phi_{q,i}^*g_i-g\|_{C^q,\Omega_q,g}<\frac{1}{q}
  \]
  for $i_{q,q}\le i<i_{q+1,q+1}$, obtaining $\phi_i^*g_i\to g|_\Omega$
  as $i\to\infty$.
\end{proof}

\begin{lem}\label{l: delta} 
  For $n\in\N$, let $[M,x]\in\MM_*^\infty(n)$, and let $\NN$ be a
  neighborhood of $[M,x]$ in $\MM_*^\infty(n)$. Then there is some
  $\delta>0$ and some neighborhood $\LL$ of $[M,x]$ in
  $\MM_*^\infty(n)$ such that $[L,z]\in\NN$ for all $[L,y]\in\LL$ and
  all $z\in B_L(y,\delta)$.
\end{lem}

\begin{proof} 
  There are some $m\in\Z^+$ and $\epsilon>0$, and a compact domain
  $\Omega$ of $M$ containing $x$ such that, for all
  $[L,z]\in\MM_*^\infty(n)$, if there is some $C^\infty$ embedding
  $\phi:\Omega\to L$ so that $\phi(x)=z$ and
  $\|\phi^*g_L-g_M\|_{C^m,\Omega,g_M}<\epsilon$, then
  $[L,z]\in\NN$. Take any compact domain $\Omega'$ of $M$ whose
  interior contains $\Omega$. There is some $\epsilon_0>0$ and some
  neighborhood $\HH$ of $\id_M$ in the group of diffeomorphisms of $M$
  with the weak $C^m$ topology such that, for all $h\in\HH$ and any
  metric tensor $g'$ on $\Omega'$ satisfying
  $\|g'-g_M\|_{C^m,\Omega',g_M}<\epsilon_0$, we have
  $h(\Omega)\subset\Omega'$ and
  $\|h^*g'-g_M\|_{C^m,\Omega,g_M}<\epsilon$. Moreover there is some $\delta'>0$ such that, 
  for each $z'\in B_M(x,\delta')$, there is some
  $h\in\HH$ so that $h(x)=z'$. Let $\LL$ be the neighborhood of $[M,x]$
  in $\MM_*^\infty(n)$ that consists of the points
  $[L,y]\in\MM_*^\infty(n)$ such that there is some $C^\infty$
  embedding $\psi:\Omega'\to L$ so that $\psi(x)=y$ and
  $\|\psi^*g_L-g_M\|_{C^m,\Omega',g_M}<\epsilon_0$. There is some
  $\delta>0$ such that $B_L(y,\delta)\subset\psi(\Omega')$ and
  $\psi^{-1}(B_L(y,\delta))\subset B_M(x,\delta')$ for all
  $[L,y]\in\LL$ and $\psi:\Omega'\to L$ as above. Hence
  $z'=\psi^{-1}(z)\in B_M(x,\delta')$ for each $z\in B_L(y,\delta)$, and therefore there is some
  $h\in\HH$ such that $h(x)=z'$. Then $\phi:=\psi h$ is defined on
  $\Omega$ and satisfies $\phi(x)=\psi(z')=z$. Moreover
		\[
 			 \|\phi^*g_L-g_M\|_{C^m,\Omega,g_M}=\|h^*\psi^*g_L-g_M\|_{C^m,\Omega,g_M}<\epsilon
		\]
	because $\|\psi^*g_L-g_M\|_{C^m,\Omega',g_M}<\epsilon_0$ and $h\in\HH$.
\end{proof}

\section{Canonical bundles over $\MM_{*,\text{\rm lnp}}^\infty(n)$}\label{s: bundles}

For each $n\in\N$, consider the set of pairs
$(M,\xi)$, where $M$ is a complete connected Riemannian manifold
without boundary of dimension $n$, and $\xi\in TM$. Like in the case of
  $\MM_*(n)$, we can assume that the underlying set of each complete
  connected Riemannian $n$-manifold is contained in \(\R\), obtaining
  that these pairs $(M,\xi)$ form a well defined set. Define an
equivalence relation on this set by declaring that $(M,\xi)$ is
equivalent to $(N,\zeta)$ if there is an isometric diffeomorphism
$\phi:M\to N$ such that $\phi_*(\xi)=\zeta$. The class of a pair
$(M,\xi)$ will be denoted by $[M,\xi]$, and the corresponding set of
equivalence classes will be denoted by $\TT_*(n)$.  If orthonormal
tangent frames are used instead of tangent vectors in the above
definition, we get a set denoted by $\QQ_*(n)$. Let
$\pi_{\TT_*(n)}:\TT_*(n)\to\MM_*(n)$ and
$\pi_{\QQ_*(n)}:\QQ_*(n)\to\MM_*(n)$ be the maps defined by
$\pi([M,\xi])=[M,\pi_M(\xi)]$ and $\pi([M,f])=[M,\pi_M(f)]$ for
$[M,\xi]\in\TT_*(n)$ and $[M,f]\in\QQ_*(n)$; the simpler notation
$\pi$ will be used for $\pi_{\TT_*(n)}$ and $\pi_{\QQ_*(n)}$ if there
is no danger of misunderstanding. For each $[M,x]\in\MM_*(n)$, there
are canonical surjections $T_xM\to\pi_{\TT_*(n)}^{-1}([M,x])$,
$\xi\mapsto[M,\xi]$, and $Q_xM\to\pi_{\QQ_*(n)}^{-1}([M,x])$,
$f\mapsto[M,f]$. Via the canonical surjection
$Q_xM\to\pi_{\QQ_*(n)}^{-1}([M,x])$, the canonical right action of
$\Or(n)$ on $Q_xM$ induces a right action on
$\pi_{\QQ_*(n)}^{-1}([M,x])$; in this way, we get a canonical action
of $\Or(n)$ on $\QQ_*(n)$ whose orbits are the fibers of
$\pi_{\QQ_*(n)}$. The operation of multiplication by scalars on $T_xM$
also induces an action of $\R$ on
$\pi_{\TT_*(n)}^{-1}([M,x])$. However the sum operation of $T_xM$ may
not induce an operation on $\pi_{\TT_*(n)}^{-1}([M,x])$. The following
definition is analogous to Definition~\ref{d: C^infty convergence in
  MM_*(n)}.

\begin{defn}\label{d: C^infty convergence in TT_*(n) and QQ_*(n)}
  For each $m\in\N$, a sequence $[M_i,\xi_i]\in\TT_*(n)$
  (respectively, $[M_i,f_i]\in\QQ_*(n)$) is said to be \emph{$C^m$
    convergent} to $[M,\xi]\in\TT_*(n)$ (respectively,
  $[M,f]\in\QQ_*(n)$) if, with the notation $x=\pi(\xi)$ and
  $x_i=\pi_i(x_i)$ (respectively, $x=\pi(f)$ and $x_i=\pi_i(f_i)$),
  for each compact domain $\Omega\subset M$ containing $x$, there are
  pointed $C^{m+1}$ embeddings $\phi_i:(\Omega,x)\to(M_i,x_i)$ for
  large enough $i$ such that $\phi_{i*}(\xi)=\xi_i$ (respectively,
  $\phi_{i*}(f)=f_i$), and $\phi_i^*g_i\to g|_\Omega$ as $i\to\infty$
  with respect to the $C^m$ topology. If $[M_i,\xi_i]$ (respectively,
  $[M_i,f_i]$) is $C^m$ convergent to $[M,\xi]$ (respectively,
  $[M,f]$) for all $m$, then it is said that $[M_i,\xi_i]$
  (respectively, $[M_i,f_i]$) is \emph{$C^\infty$ convergent} to
  $[M,\xi]$ (respectively, $[M,f]$).
\end{defn}

\begin{thm}\label{t: C^infty convergence in TT_*(n) and QQ_*(n)}
  The $C^\infty$ convergence in $\TT_*(n)$ and $\QQ_*(n)$ describes a
  Polish topology.
\end{thm}

To prove Theorem~\ref{t: C^infty convergence in TT_*(n) and QQ_*(n)},
we follow the steps of Sections~\ref{s: C^infty topology}--\ref{s:
  MM_*^infty(n) is Polish}.

\begin{defn}\label{d: V^m_R,r and W^m_R,r}
  For $m\in\N$ and $R,r>0$, let $V^m_{R,r}$ (respectively,
  $W^m_{R,r}$) be the set of pairs
  $([M,\xi],[N,\zeta])\in\TT_*(n)\times\TT_*(n)$ (respectively,
  $([M,f],[N,h])\in\QQ_*(n)\times\QQ_*(n)$) such that there is some
  $(m,R,\lambda)$-pointed local quasi-isometry
  $\phi:(M,x)\rightarrowtail (N,y)$ for some $\lambda\in[1,e^r)$ so
  that $\phi_*(\xi)=\zeta$ (respectively, $\phi_*(f)=h$).
\end{defn}

The following proposition is proved like Proposition~\ref{p: C^infty
  uniformity in MM_*(n)}.

\begin{prop}\label{p: C^infty uniformity in TT_*(n) and QQ_*(n)}
  The following properties hold for all $m,m'\in\N$ and $R,S,r,s>0$:
  \begin{enumerate}[{\rm(}i{\rm)}]
  	
  \item $(V^m_{e^rR,r})^{-1} \subset V^m_{R,r}$ and
    $(W^m_{e^rR,r})^{-1} \subset W^m_{R,r}$.
		
  \item $V^{m_0}_{R_0,r_0}\subset V^m_{R,r}\cap V^{m'}_{S,s}$ and
    $W^{m_0}_{R_0,r_0}\subset W^m_{R,r}\cap W^{m'}_{S,s}$, where
    $m_0=\max\{m,m'\}$, $R_0=\max\{R,S\}$ and $r_0=\min\{r,s\}$.
		
  \item $\Delta\subset V^m_{R,r}$ and $\Delta\subset W^m_{R,r}$.
		
  \item $V^m_{e^{r+s}R,r} \circ V^m_{e^{r+s}R,s}\subset V^m_{R,r+s}$
    and $W^m_{e^{r+s}R,r} \circ W^m_{e^{r+s}R,s}\subset W^m_{R,r+s}$.
		
  \end{enumerate}
\end{prop}

\begin{prop}\label{p: Hausdorff TT_*(n) and QQ_*(n)}
  $\bigcap_{R,r>0}V^m_{R,r}=\Delta$ and
  $\bigcap_{R,r>0}W^m_{R,r}=\Delta$ for all $m\in\N$.
\end{prop}

\begin{proof}
  We only prove the first equality because the proof of the second one
  is analogous. The inclusion ``$\supset$'' is obvious; thus let us
  prove ``$\subset$''. Let
  $([M,\xi],[N,\zeta])\in\bigcap_{R,r>0}V^m_{R,r}$, and let
  $x=\pi_M(\xi)$ and $y=\pi_N(\zeta)$. Then there is a sequence of
  pointed local quasi-isometries $\phi_i:(M,x)\rightarrowtail(N,y)$,
  with corresponding types $(m,R_i,\lambda_i)$, such that
  $\phi_{i*}(\xi)=\zeta$, and $R_i\uparrow\infty$ and
  $\lambda_i\downarrow1$ as $i\to\infty$. According to the proof of
  Proposition~\ref{p: Hausdorff MM_*(n)}, there is a pointed isometric
  immersion $\psi:(M,x)\to(N,y)$ so that, for any $i$, the restriction
  $\psi:B_M(x,R_i)\to N$ is the limit of the restrictions of a
  subsequence $\phi_{k(i,l)}$ in the weak $C^m$ topology. Hence
  $\psi_*(\xi)=\lim_l\phi_{k(i,l)*}(\xi)=\zeta$, obtaining
  $[M,\xi]=[N,\zeta]$.
\end{proof}

By Propositions~\ref{p: C^infty uniformity in TT_*(n) and QQ_*(n)}
and~\ref{p: Hausdorff TT_*(n) and QQ_*(n)}, the sets $V_{R,r}^m$
(respectively, $W_{R,r}^m$) form a base of entourages of a Hausdorff
uniformity on $\TT_*(n)$ (respectively, $\QQ_*(n)$), which is also
called the \emph{$C^\infty$ uniformity}. The corresponding topology is
also called the \emph{$C^\infty$ topology}, and the corresponding
space is denoted by $\TT_*^\infty(n)$ (respectively,
$\QQ_*^\infty(n)$).

\begin{rem}\label{r: pi is cont}
  \begin{enumerate}[(i)]
	
  \item\label{i: pi is cont} The maps
    $\pi:\TT_*^\infty(n)\to\MM_*^\infty(n)$ and
    $\pi:\QQ_*^\infty(n)\to\MM_*^\infty(n)$ are uniformly continuous
    and open because
    $(\pi\times\pi)(V^m_{R,r})=(\pi\times\pi)(W^m_{R,r})=U^m_{R,r}$
    for all $m\in\N$ and $R,r>0$.
		
  \item\label{i: the O(n)-action is cont} The canonical right
    $\Or(n)$-action on $\QQ_*^\infty(n)$ is continuous. This follows
    easily by using that the composite of maps is continuous in the
    weak $C^\infty$ topology \cite[p.~64, Exercise~10]{Hirsch1976},
    and the following property that can be easily verified: for each
    $[M,f]\in\QQ_*^\infty(n)$ and any neighborhood $\NN$ of $\id_M$ in
    the space of $C^\infty$ diffeomorphisms of $M$ with the weak
    $C^\infty$ topology, there is a neighborhood $O$ of the identity
    element $e$ in $\Or(n)$ such that, for all $a\in O$, there is some
    $\phi\in\NN$ so that $\phi(x)=x$ and $\phi_*(f)=h$.
		
  \end{enumerate}
\end{rem}

\begin{defn}\label{d: E^m_R,r and F^m_R,r}
  For $R,r>0$ and $m\in\N$, let $E^m_{R,r}$ (respectively,
  $F^m_{R,r}$) be the set of pairs
  $([M,\xi],[N,\zeta])\in\TT_*(n)\times\TT_*(n)$ (respectively,
  $([M,f],[N,h])\in\QQ_*(n)\times\QQ_*(n)$) such that, with the
  notation $x=\pi_M(\xi)$ and $y=\pi_N(\zeta)$, there is some
  $C^{m+1}$ pointed local diffeomorphism
  $\phi\colon(M,x)\rightarrowtail(N,y)$ so that $\phi_*(\xi)=\zeta$
  (respectively, $\phi_*(f)=h$), and
  $\|g_M-\phi^*g_N\|_{C^m,\Omega,g_M}<r$ for some compact domain
  $\Omega\subset\dom\phi$ with $B_M(x,R)\subset\Omega$.
\end{defn} 

Like in the case of relations on $\MM_*(n)$, for $V\subset\TT_*(n)\times\TT_*(n)$,
     $W\subset\QQ_*(n)\times\QQ_*(n)$, $[M,\xi]\in\TT_*(n)$ and
     $[M,f]\in\QQ_*(n)$, the simpler notation $V(M,\xi)$ and $W(M,f)$
     is used instead of $V([M,\xi])$ and $W([M,f])$.
 
 \begin{rem}\label{r: E^m_R,r}
   By~\eqref{norm equiv}, a sequence $[M_i,\xi_i]\in\TT_*(n)$
   (respectively, $[M_i,f_i]\in\QQ_*(n)$) is $C^\infty$ convergent to
   $[M,\xi]\in\TT_*(n)$ (respectively, $[M,f]\in\QQ_*(n)$) if and only
   if it is eventually in
   $E^m_{R,r}(M,\xi)$ (respectively, $F^m_{R,r}(M,f)$) for arbitrary
   $m\in\N$ and $R,r>0$.
 \end{rem}
 
\begin{prop}\label{p: E^m_R,epsilon(M,xi) subset V^m_R,r(M,xi)}
  \begin{enumerate}[{\rm(}i{\rm)}]

  \item\label{i: E^0_R,epsilon subset V^0_R,r} For $R,r>0$, if
    $0<\epsilon\le\min\{1-e^{-2r},e^{2r}-1\}$, then
    $E^0_{R,\epsilon}\subset V^0_{R,r}$ and $F^0_{R,\epsilon}\subset
    W^0_{R,r}$.
		
  \item\label{i: E^m_R,epsilon(M,xi) subset V^m_R,r(M,xi)} For all
    $m\in\Z^+$, $R, r>0$ and $[M,\xi]\in\TT_*(n)$ {\rm(}respectively,
    $[M,f]\in\PP_*(n)${\rm)}, there is some $\epsilon>0$ such that
    $E^m_{R,\epsilon}(M,\xi)\subset V^m_{R,r}(M,\xi)$
    {\rm(}respectively, $F^m_{R,\epsilon}(M,\xi)\subset
    W^m_{R,r}(M,\xi)${\rm)}.

  \end{enumerate}
\end{prop}

\begin{proof}
  Let us show~\eqref{i: E^0_R,epsilon subset V^0_R,r} for the case of
  $V^0_{R,r}$, the case of $W^0_{R,r}$ being analogous. Let
  $([M,\xi],[N,\zeta])\in E^0_{R,\epsilon}$, and let $x=\pi_M(\xi)$
  and $y=\pi_N(\zeta)$. Then there is a $C^1$ pointed local
  diffeomorphism $\phi:(M,x)\rightarrowtail(N,y)$ such that
  $\phi_*(\xi)=\zeta$, and
  $\epsilon_0:=\|g_M-\phi^*g_N\|_{C^0,\Omega,g_M}<\epsilon$ for some
  compact domain $\Omega\subset\dom\phi$ with
  $B_M(x,R)\subset\Omega$. According to the proof of
  Proposition~\ref{p: D^m_R,epsilon(M,x) subset
    U^m_R,r(M,x)}-\eqref{i: D^0_R,epsilon subset U^0_R,r}, $\phi$ is a
  $(0,R,\lambda)$-pointed local quasi-isometry if $1\le\lambda<e^r$
  and $\epsilon_0\le\min\{1-\lambda^{-2},\lambda^2-1\}$, obtaining
  that $([M,\xi],[N,\zeta])\in V^0_{R,r}$.

  As above, let us prove~\eqref{i: E^m_R,epsilon(M,xi) subset
    V^m_R,r(M,xi)} only for the case of $V^m_{R,r}(M,\xi)$. Take
  $m\in\Z^+$, $R, r>0$ and $[M,\xi],[N,\zeta]\in\TT_*(n)$, and let
  $x=\pi_M(\xi)$ and $y=\pi_N(\zeta)$.  According to the proof of
  Proposition~\ref{p: D^m_R,epsilon(M,x) subset
    U^m_R,r(M,x)}-\eqref{i: D^m_R,epsilon(M,x) subset U^m_R,r(M,x)},
  there is some $\epsilon>0$ such that, for every $C^{m+1}$ pointed
  local diffeomorphism $\phi\colon(M,x)\rightarrowtail(N,y)$, if
  $\|g_M-\phi^*g_N\|_{C^m,\Omega,g_M}<\epsilon$ for some compact
  domain $\Omega\subset\dom\phi\cap\Int(K)$ with
  $B_M(x,R)\subset\Omega$, then $\phi$ is an $(m,R,\lambda)$-pointed
  local quasi-isometry $(M,x)\rightarrowtail (N,y)$ for some
  $\lambda\in[1,e^r)$. Therefore $[N,\zeta]\in V^m_{R,r}(M,\xi)$ if
  $[N,\zeta]\in E^m_{R,\epsilon}(M,\xi)$.
\end{proof}

\begin{prop}\label{p: V^m_R,epsilon(M,xi) subset E^m_R,r(M,xi)}
  \begin{enumerate}[{\rm(}i{\rm)}]

  \item For all $R,r>0$, if $e^{2\epsilon}-e^{-2\epsilon}\le r$, then
    $V^0_{R,\epsilon}\subset E^0_{R,r}$ and $W^0_{R,\epsilon}\subset
    F^0_{R,r}$.
		
  \item For all $m\in\Z^+$, $R, r>0$ and $[M,\xi]\in\TT_*(n)$
    {\rm(}respectively, $[M,f]\in\QQ_*(n)${\rm)}, there is some
    $\epsilon>0$ such that $V^m_{R,\epsilon}(M,\xi)\subset
    E^m_{R,r}(M,\xi)$ {\rm(}respectively,
    $W^m_{R,\epsilon}(M,f)\subset F^m_{R,r}(Mf)${\rm)}.

  \end{enumerate}
\end{prop}

\begin{proof}
  This result follows from the proof of Proposition~\ref{p:
    U^m_R,epsilon(M,x) subset D^m_R,r(M,x)} in the same way as
  Proposition~\ref{p: E^m_R,epsilon(M,xi) subset V^m_R,r(M,xi)}
  follows from Proposition~\ref{p: D^m_R,epsilon(M,x) subset
    U^m_R,r(M,x)}.
\end{proof}

As a direct consequence of Remark~\ref{r: E^m_R,r}, and
Propositions~\ref{p: E^m_R,epsilon(M,xi) subset V^m_R,r(M,xi)}
and~\ref{p: V^m_R,epsilon(M,xi) subset E^m_R,r(M,xi)}, we get that the
$C^\infty$ convergence in $\TT_*(n)$ and $\QQ_*(n)$ describes the
$C^\infty$ topology.

\begin{prop}\label{p: TT_*^infty(n) and QQ_*^infty(n) are separable}
$\TT_*^\infty(n)$ and $\QQ_*^\infty(n)$ are separable
\end{prop}

\begin{proof}
  With the notation of Proposition~\ref{p: MM_*^infty(n) is
    separable}, for every $M\in\CC$, let $\DD'_M$ and $\DD''_M$ be
  countable dense subsets of $TM$ and $QM$, respectively. Then the
  countable sets
  \[
  \{\,[(M,g),\xi]\mid M\in\mathcal{C},\ g\in\GG_M,\
  \xi\in\DD'_M\,\}\quad\text{and}\quad \{\,[(M,g),f]\mid
  M\in\mathcal{C},\ g\in\GG_M,\ f\in\DD''_M\,\}
  \]
  are dense in $\TT_*^\infty(n)$ and $\QQ_*^\infty(n)$, respectively.
\end{proof}

\begin{prop}\label{p: TT_*^infty(n) and QQ_*^infty(n) are completely
    metrizable}
$\TT_{*}^\infty(n)$ and $\QQ_{*}^\infty(n)$ are completely metrizable
\end{prop}

\begin{proof}
  Only the case of $\TT_{*}^\infty(n)$ is proved, the other case being
  similar. The $C^\infty$ uniformity on $\TT_{*}^\infty(n)$ is
  metrizable because it has a countable base of entourages. Thus it is
  enough to check that this uniformity is complete.
	
  Consider an arbitrary Cauchy sequence $[M_i,\xi_i]$ in $\TT_*(n)$
  with respect to the $C^\infty$ uniformity, and let
  $x_i=\pi_i(\xi_i)\in M_i$. We have to prove that $[M_i,\xi_i]$ is
  convergent in $\TT_*^\infty(n)$. By taking a subsequence if
  necessary, we can suppose that $([M_i,\xi_i],[M_{i+1},\xi_{i+1}])\in
  V^{m_i}_{R_i,r_i}$ for sequences $m_i$, and $R_i$ and $r_i$
  satisfying the conditions of the proof of Proposition~\ref{p:
    MM_*^infty(n) is completely metrizable}. Thus, for each $i$, there
  is some $\lambda_i\in(1,e^{r_i})$ and some
  $(m_i,R_i,\lambda_i)$-pointed local quasi-isometry
  $\phi_i\colon(M_i,x_i)\rightarrowtail(M_{i+1},x_{i+1})$, which can
  be assumed to be $C^\infty$ (Remark~\ref{r:
    (m,R,lambda)-...}-\eqref{i: C^infty approximation of
    (m,R,lambda)-...}), such that $\phi_{i*}(\xi_i)=\xi_{i+1}$. Then,
  with the notation of the proof of Proposition~\ref{p: MM_*^infty(n)
    is completely metrizable}, we have $\psi_{ij*}(\xi_i)=\xi_j$ for
  $i<j$. Therefore there is some $\hat\xi\in T_{\hat x}\widehat M$ so
  that $\psi_{i*}(\xi_i)=\hat\xi$ for all $i$, obtaining that
  $([M_i,\xi_i],[\widehat M,\hat\xi])\in
  U^{m_i}_{R'_i/\bar\lambda_i,s_i}$ for all $i$ according to the proof
  of Proposition~\ref{p: MM_*^infty(n) is completely
    metrizable}. Hence $[M_i,\xi_i]\to[\widehat M,\hat\xi]$ as
  $i\to\infty$ in $\TT_*^\infty(n)$.
\end{proof}

Propositions~\ref{p: TT_*^infty(n) and QQ_*^infty(n) are separable}
and~\ref{p: TT_*^infty(n) and QQ_*^infty(n) are completely metrizable}
together mean that $\TT_*^\infty(n)$ and $\QQ_*^\infty(n)$ are Polish,
completing the proof of Theorem~\ref{t: C^infty convergence in TT_*(n)
  and QQ_*(n)}.

Let $\TT_{*,\text{\rm lnp}}^\infty(n)\subset\TT_*^\infty(n)$ and
$\QQ_{*,\text{\rm lnp}}^\infty(n)\subset\QQ_*^\infty(n)$ be the
subspaces defined by locally non-periodic manifolds.

\begin{prop}\label{p: TT} 
  \begin{enumerate}[{\rm(}i{\rm)}]

  \item\label{i: TT} The projection $\pi:\TT_{*,\text{\rm
        lnp}}^\infty(n)\to\MM_{*,\text{\rm lnp}}^\infty(n)$ admits the
    structure of a Riemannian vector bundle of rank $n$ so that the
    canonical map $T_xM\to\pi^{-1}([M,x])$ is a orthogonal isomorphism
    for each $[M,x]\in\MM_{*,\text{\rm lnp}}^\infty(n)$.

  \item\label{i: QQ} The projection $\pi:\QQ_{*,\text{\rm
        lnp}}^\infty(n)\to\MM_{*,\text{\rm lnp}}^\infty(n)$ admits the
    structure of a $\Or(n)$-principal bundle canonically isomorphic to
    the $\Or(n)$-principal bundle of orthonormal references of
    $\TT_{*,\text{\rm lnp}}^\infty(n)$.

  \end{enumerate}
\end{prop}

\begin{proof} 
  Obviously, the canonical $\Or(n)$-action on $\QQ_*^\infty(n)$
  preserves $\QQ_{*,\text{\rm lnp}}^\infty(n)$, and the
  $\Or(n)$-orbits in $\QQ_{*,\text{\rm lnp}}^\infty(n)$ are the fibers
  of $\pi:\QQ_{*,\text{\rm lnp}}^\infty(n)\to\MM_{*,\text{\rm
      lnp}}^\infty(n)$.
	
  \begin{claim}\label{cl: free}
    For all $[M,x]\in\MM_{*,\text{\rm lnp}}^\infty(n)$, the canonical
    maps $T_xM\to\pi_{\TT_*(n)}^{-1}([M,x])$ and
    $Q_xM\to\pi_{\QQ_*(n)}^{-1}([M,x])$ are bijections.
  \end{claim}
		
  Let us show the case of the first map in Claim~\ref{cl: free}, the
  case of the second one being similar. It was already pointed out
  that the canonical map $T_xM\to\pi_{\TT_*(n)}^{-1}([M,x])$ is
  surjective, and let us to prove that it is also injective. If
  $[M,\xi]=[M,\zeta]$ for some $\xi,\zeta\in T_xM$, then
  $\phi_*(\xi)=\zeta$ for some $\phi\in\Iso(M)$ with $\phi(x)=x$. But
  $\phi=\id_M$ because $M$ is locally non-periodic, obtaining
  $\xi=\zeta$.
	
  Let $X$ be a completely
    regular space with a right action of a Lie group $G$, and let
    $G_x\subset G$ denote the isotropy subgroup at some point $x\in
    X$. Recall that a {\em slice\/} at $x$ is a subspace $S\subset X$ containing
    $x$ such that $S\cdot G$ is open in $X$, and there is a
    $G$-equivariant continuous map $\kappa:S\cdot G\to G_x\backslash
    G$ with $\kappa^{-1}(G_x)=S$ \cite[Definition~2.1.1]{Palais1961}. Since $\QQ_{*,\text{\rm lnp}}^\infty(n)$ is completely regular and
  $\Or(n)$ is compact, the $\Or(n)$-action on $\QQ_{*,\text{\rm
      lnp}}^\infty(n)$ has a slice
  $\SSS$ at each point $[M,f]\in\QQ_{*,\text{\rm lnp}}^\infty(n)$
  \cite[Theorem~2.3.3]{Palais1961} (see also
  \cite{HofmannMostert1966}, \cite[Theorems~5.1 and~5.2]{Ramsay1991}
  and \cite[Theorems~11.3.9 and~11.3.14]{CandelConlon2000-I}). Then
  $\Theta:=\pi(\SSS)=\pi(\SSS\cdot\Or(n))$ is open in
  $\MM_{*,\text{\rm lnp}}^\infty(n)$ by Remark~\ref{r: pi is
    cont}-\eqref{i: pi is cont}.
	
  \begin{claim}\label{cl: pi:SSS to Theta is a homeomorphism}
    $\pi:\SSS\to\Theta$ is a homeomorphism.
  \end{claim}
	
  This is the restriction of a continuous map (Remark~\ref{r: pi is
    cont}-\eqref{i: pi is cont}), and therefore it is continuous. This
  map is also open because, for every open $W\subset\SSS$, the set
  $W\cdot\Or(n)$ is open in $\QQ_{*,\text{\rm lnp}}^\infty(n)$
  \cite[Corollary of Proposition~2.1.2]{Palais1961}, and thus
  $\pi(W)=\pi(W\cdot\Or(n))$ is open in $\MM_*^\infty(n)$
  (Remark~\ref{r: pi is cont}-\eqref{i: pi is cont}). Obviously,
  $\pi:\SSS\to\Theta$ is surjective, and let us show that it is also
  injective. Take $[N,p],[L,q]\in\SSS$ such that
  $\pi([N,p])=\pi([N,q])=:x$. Thus there is some $a\in\Or(n)$ so that
  $[L,q]=[N,p]\cdot a$. Since the isotropy group at $[M,f]$ is trivial
  by Claim~\ref{cl: free}, there is an $\Or(n)$-equivariant continuous
  map $\kappa:\SSS\cdot\Or(n)\to\Or(n)$ so that
  $\kappa^{-1}(e)=\SSS$. It follows that
  $e=\kappa([L,q])=\kappa([N,p]\cdot a)=\kappa([N,p])\,a=a$, obtaining
  $[L,q]=[N,p]$, which completes the proof of Claim~\ref{cl: pi:SSS to
    Theta is a homeomorphism}.
	
  According to Claim~\ref{cl: pi:SSS to Theta is a homeomorphism}, the
  inverse of $\pi:\SSS\to\Theta$ defines a continuous local section
  $\sigma:\Theta\to\QQ_{*,\text{\rm lnp}}^\infty(n)$ of
  $\pi:\QQ_{*,\text{\rm lnp}}^\infty(n)\to\MM_{*,\text{\rm
      lnp}}^\infty(n)$. By the existence of continuous local sections,
  and since the $\Or(n)$-action on $\QQ_{*,\text{\rm lnp}}^\infty(n)$
  is continuous and free (Remark~\ref{r: pi is cont}-\eqref{i: the
    O(n)-action is cont} and Claim~\ref{cl: free}), it easily follows
  that $\pi:\QQ_{*,\text{\rm lnp}}^\infty(n)\to\MM_{*,\text{\rm
      lnp}}^\infty(n)$ admits the structure of an $\Or(n)$-principal
  bundle.
	
  By Claim~\ref{cl: free}, $\pi_{\TT_*(n)}^{-1}([M,x])$ canonically
  becomes an orthogonal vector space for each
  $[M,x]\in\MM_{*,\text{\rm lnp}}^\infty(n)$, and we can canonically
  identify $\pi_{\QQ_*^{-1}(n)}([M,x])$ to the set of linear
  isometries $\pi_{\TT_*(n)}^{-1}([M,x])\to\R^n$. The continuity of
  the mapping $([M,f],[M,\xi])\mapsto[M,f]([M,\xi])$ is easy to check.
  By using this identity, we get a homeomorphism
  $\theta:\pi_{\TT_*(n)}^{-1}(\Theta)\to\R^n\times\Theta$ defined by
  $\theta([M,\xi])=(\sigma([M,x])([M,\xi]),[M,x])$, where
  $\pi([M,\xi])=[M,x]$, whose inverse map is given by
  $\theta^{-1}(v,[M,x])=[M,\sigma([M,x])^{-1}(v)]$. If
  $\sigma':\Theta'\to\QQ_{*,\text{\rm lnp}}^\infty(n)$ is another
  local section of $\pi:\QQ_{*,\text{\rm
      lnp}}^\infty(n)\to\MM_{*,\text{\rm lnp}}^\infty(n)$ defining a
  map $\theta':\pi^{-1}(\Theta')\to\R^n\times\Theta'$ as above, and
  $[M,x]\in\Theta\cap\Theta'$, then the composite
  \[
  \begin{CD}
    \R^n\equiv\R^n\times\{[M,x]\} @>{\theta^{-1}}>>
    \pi_{\TT_*(n)}^{-1}([M,x]) @>{\theta'}>>
    \R^n\times\{[M,x]\}\equiv\R^n
  \end{CD}
  \]
  is the orthogonal isomorphism
  $\sigma'([M,x])\circ\sigma([M,x])^{-1}$. It follows that
  $\pi:\TT_{*,\text{\rm lnp}}^\infty(n)\to\MM_{*,\text{\rm
      lnp}}^\infty(n)$, with these local trivializations, becomes an
  orthogonal vector bundle of rank $n$ so that the canonical map
  $T_xM\to\pi^{-1}([M,x])$ is a orthogonal isomorphism for all
  $[M,x]\in\MM_{*,\text{\rm lnp}}^\infty(n)$. Moreover, by
  Claim~\ref{cl: free}, there is a canonical isomorphism between
  $\QQ_{*,\text{\rm lnp}}^\infty(n)$ and the $\Or(n)$-principal bundle
  of orthonormal frames of $\TT_{*,\text{\rm lnp}}^\infty(n)$.
\end{proof}

By the compatibility of exponential maps and isometries, a map
$\exp:\TT_*^\infty(n)\to\MM_*^\infty(n)$ is well defined by setting
$\exp([M,\xi])=[M,\exp_M(\xi)]$. For each $[M,x]\in\MM_*^\infty(n)$,
the restriction $\exp:\pi^{-1}([M,x])\to\MM_*^\infty(n)$ may be
denoted by $\exp_{[M,x]}$.

\begin{lem}\label{l: double convergence} 
  Consider convergent sequences $[M_i,f_i]\to[M,f]$ and
  $[M_i,f'_i]\to[M,f']$ in $\QQ_*^\infty(n)$ for some $n\in\Z^+$.  Let
  $x=\pi(f)$, $x'=\pi(f')$, $x_i=\pi_i(f_i)$ and
  $x'_i=\pi_i(f'_i)$. Suppose that there is some $r>0$ such that
  \begin{equation}\label{double convergence}
    \left\{h\in\Iso(M)\ |\ h(x)\in\ol B(x,2r)\right\}=\{\id_M\}\;,
  \end{equation} 
  and $d(x,x'),d_i(x_i,x'_i)\le r$ for all $i$. Then there is some
  compact domain $\Omega$ in $M$ whose interior contains $x$ and $x'$,
  and there are $C^\infty$ embeddings $\phi_i:\Omega\to M_i$ for $i$
  large enough so that $\phi_{i*}(f)=f_i$ and
  $\lim_i\phi_{i*}^{-1}(f'_i)=f'$ in $PM$, and
  $\lim_i\phi_i^*g_i=g|_\Omega$ with respect to the $C^\infty$
  topology.
\end{lem}

\begin{proof} 
  Let $\Omega_q$ be a sequence of compact domains in $M$ such that
  \[
  \ol B(x,r)\subset\Int(\Omega_q)\;,\quad
  \Pen(\Omega_q,\diam(\Omega_q))\subset\Int(\Omega_{q+1})\;;
  \]
  in particular, $x'\in\Int(\Omega_q)$. By the convergence
  $[M_i,f_i]\to[M,f]$ and $[M_i,f'_i]\to[M,f']$ in $\QQ_*^\infty(n)$,
  for each $q$, there are $C^\infty$ embeddings
  $\phi_{q,i},\psi_{q,i}:\Omega_q\to M_i$ for $i$ large enough so that
  $\phi_{q,i*}(f)=f_i$, $\psi_{q,i*}(f')=f'_i$, and
  $\lim_i\phi_{q,i}^*g_i=g|_{\Omega_q}$ and
  $\lim_i\psi_{q,i}^*g_i=g|_{\Omega_q}$ with respect to the $C^\infty$
  topology; in particular, $\phi_{q,i}(x)=x_i$ and
  $\psi_{q,i}(x')=x'_i$. We have $x'_i\in\ol
  B_i(x_i,r)\subset\Int(\phi_{q,i}(\Omega_q))$ for $i$ large enough,
  depending on $q$, and therefore
  $\phi_{q,i}(\Omega_q)\cap\psi_{q,i}(\Omega_q)\ne\emptyset$. Hence
  \[
  \psi_{q,i}(\Omega_q)\subset\Pen_i(\phi_{q,i}(\Omega_q),\diam(\phi_{q,i}(\Omega_q)))
  \subset\Int(\phi_{q,i}(\Omega_{q+1}))
  \]
  for $i$ large enough, depending on $q$. It follows that
  $h_{q,i}:=\phi_{q+1,i}^{-1}\psi_{q,i}$ is a well defined $C^\infty$
  embedding $\Omega_q\to M$. Observe that
  $\lim_ih_{q,i}^*g=g|_{\Omega_q}$ with respect to the $C^\infty$
  topology. Moreover
  \begin{multline*}
    \limsup_id(x,h_{q,i}(x))=\limsup_id(x,\phi_{q+1,i}^{-1}\psi_{q,i}(x))
    =\limsup_id_i(x_i,\psi_{q,i}(x))\\
    \le\limsup_id_i(x_i,x'_i)+\limsup_id_i(x'_i,\psi_{q,i}(x))\le
    r+d(x',x)\le2r\;.
  \end{multline*}
  If the statement is not true, then some neighborhood $U$ of $f'$ in
  $PM$ contains no accumulation point of the sequence
  $\phi_{q+1,i*}^{-1}(f'_i)=\phi_{q+1,i*}^{-1}\psi_{q,i*}(f')=h_{q,i*}(f')$
  for each $q$. With the arguments of the proof of Lemma~\ref{l: LL},
  it follows that there is some $h\in\Iso(M)$ such that
  $d(x,h(x))\le2r$ and $h_*(f')\not\in U$, which
  contradicts~\eqref{double convergence}.
\end{proof}

\section{Center of mass}\label{s: center of mass}

The main tool used to prove Theorem~\ref{t: FF_*,lnp(n)}-\eqref{i:
  FF_*,lnp(n) is foliated structure}--\eqref{i: MM_*,np(n)} is the
Riemannian center of mass of a mass distribution on a Riemannian
manifold $M$ \cite{Karcher1977}, \cite[Section~IX.7]{Chavel2006};
especially, we will use the continuous dependence of the center of
mass on the mass distribution and the metric tensor.

Recall that a domain $\Omega\subset M$ is
  said to be \emph{convex} when, for all $x,y\in\Omega$, there is a unique 
  minimizing geodesic segment from $x$ to $y$ in $M$ that 
  lies in $\Omega$ (see e.g.\ \cite[Section~IX.6]{Chavel2006}). 
  For example, sufficiently small balls are convex. For a fixed convex compact domain
$\Omega$ in $M$, let $\CC(\Omega)$ be the set of functions $f\in
C^2(\Omega)$ such that the gradient $\grad f$ is an outward pointing
vector field on $\partial\Omega$ and $\Hess f$ is positive definite on
the interior $\Int(\Omega)$ of $\Omega$. Notice that $\CC(\Omega)$ is
open in the Banach space $C^2(\Omega)$ with the norm $\|\
\|_{C^2,\Omega,g}$, and thus it is a $C^\infty$ Banach
manifold. Moreover $\CC(\Omega)$ is preserved by the operations of sum
and product by positive numbers. Any $f\in\CC(\Omega)$ attains its
minimum value at a unique point $\bm(f)\in\Int(\Omega)$, defining a
function $\bm:\CC(\Omega)\to\Int(\Omega)$.

\begin{lem}\label{l: convex}
	$\bm$ is continuous.
\end{lem}

\begin{proof} 
  Consider the map $\bv:\CC(\Omega)\times\Int(\Omega)\to T\Omega$
  defined by $\bv(f,x)=\grad f(x)$, and let $Z\subset T\Omega$
  denote the image of the zero section.  Since the graph of $\bm$ is
  equal to $\bv^{-1}(Z)$, it is enough to prove the following.

\begin{claim}\label{cl: v is C^1 and transv to Z}
	$\bv$ is $C^1$ and transverse to $Z$.
\end{claim}

Here, smoothness and transversality refer to \(\bv\) considered as a
map between $C^\infty$ Banach manifolds
\cite[p.~45]{AbrahamRobbin1967}.

Let $\pi_{\HH}$ and $\pi_{\VV}$ denote the orthogonal projections of
$T^{(2)}\Omega$ onto $\HH$ and $\VV$, respectively. Let ${\mathfrak
  X}^1(\Omega)$ denote the Banach space of $C^1$ vector fields over
$\Omega$ with the norm $\|\ \|_{C^1,\Omega,g}$, which is equivalent to
the norm $\|\ \|_1$ defined by
\[
\|X\|_1=\sup\left\{\,|X(x)|+|\nabla X(x)|\mid x\in\Omega\,\right\}\;.
\]

The gradient map, $\grad:C^2(\Omega)\to{\mathfrak X}^1(\Omega)$, is a
continuous linear map between Banach spaces, and therefore it is
$C^\infty$.  The evaluation map, $\ev:{\mathfrak
  X}^1(\Omega)\times\Omega\to T\Omega$, is $C^1$ because, if
$X\in{\mathfrak X}^1(\Omega)$, $Y\in T_X{\mathfrak
  X}^1(\Omega)\equiv{\mathfrak X}^1(\Omega)$, $x\in\Omega$ and $\xi\in
T_x\Omega$, then $\ev_*(Y,\xi)\in T_\xi T\Omega$ is easily seen to be
determined by the conditions $\pi_{\HH}(\ev_*(Y,\xi))\equiv\xi$ in
$\HH_\xi\equiv T_x\Omega$ and
$\pi_{\VV}(\ev_*(Y,\xi))\equiv Y(x)+\nabla_\xi X$ in $\VV_\xi\equiv
T_x\Omega$. Therefore $\bv$ is $C^1$ because it is the restriction to
$\CC(\Omega)\times\Int(\Omega)$ of the composition
\[
\begin{CD} 
C^2(\Omega)\times\Omega @>{\grad\times\id_{\Omega}}>> {\mathfrak
X}^1(\Omega)\times\Omega @>{\ev}>> T\Omega\;.
\end{CD}
\]

Fix any $f\in\CC(\Omega)$ and $x\in\Int(\Omega)$ with $\bv(f,x)\in Z$;
thus $\grad f(x)=0_x$.

\begin{claim}\label{cl: pi_VV: v_*(0_f times T_xOmega) to VV_0_x is an iso}
  $\pi_{\VV}:\bv_*(\{0_f\}\times T_x\Omega)\to\VV_{0_x}$ is an
  isomorphism.
\end{claim}

For any $\xi\in T_x\Omega$,
\[
\pi_{\VV}\,\bv_*(0_f,\xi)=\pi_{\VV}\,(\grad f)_*(\xi)\equiv\nabla_\xi\grad f
\]
in $\VV_{0_x}\equiv T_x\Omega$. Then Claim~\ref{cl: pi_VV: v_*(0_f
  times T_xOmega) to VV_0_x is an iso} follows because the mapping
$\xi\mapsto\nabla_\xi\grad f$ is an automorphism of $T_x\Omega$ since
$\Hess f$ is positive definite at $x$ and $\Hess
f(\xi,\cdot)=g(\nabla_\xi\grad f,\cdot)$ on $T_xM$.

From Claim~\ref{cl: pi_VV: v_*(0_f times T_xOmega) to VV_0_x is an
  iso}, it follows that $\bv_*(\{0_f\}\times T_x\Omega)$ is a linear
complement to $\HH_{0_x}=T_{0_x}Z$ in $T_{0_x}T\Omega$; in particular,
it is closed in $T_{0_x}T\Omega$ because $T_{0_x}T\Omega$ is Hausdorff
of finite dimension.

Since $\bv_*:T_f\CC(\Omega)\times T_x\Omega\to T_{0_x}T\Omega$ is
linear and continuous, and $T_{0_x}T\Omega$ is Hausdorff of finite
dimension, we get that the space
$\left(\bv_{*(f,x)}\right)^{-1}(T_{0_x}Z)$ is closed and of finite
codimension in the Banach space $T_f\CC(\Omega)\times T_x\Omega$, and
therefore it has a closed linear complement in $T_f\CC(\Omega)\times
T_x\Omega$ (see e.g.\ \cite[p.~22]{Schaefer1971}), which completes the
proof of Claim~\ref{cl: v is C^1 and transv to Z}.
\end{proof}

\begin{rem}
\begin{enumerate}[(i)]

\item In the last part of the above proof, the space
  $\left(\bv_{*(f,x)}\right)^{-1}(T_{0_x}Z)$ can be described as
  follows.  Since $h\mapsto\grad h(x)$ defines a continuous linear map
  $C^2(\Omega)\to T_x\Omega$, we have
  $\bv_*(T_f\CC(\Omega)\times\{0_x\})\subset\VV_{0_x}$ and
  $\bv_*(h,0_x)\equiv\grad h(x)$ in $\VV_{0_x}\equiv T_x\Omega$ for
  any $h\in C^2(\Omega)\equiv T_f\CC(\Omega)$, giving
\[
	\left(\bv_{*(f,x)}\right)^{-1}(T_{0_x}Z)
	\equiv\{\,(h,\xi)\in C^2(\Omega)\times T_x\Omega\mid \grad h(x)+\nabla_\xi\grad f=0\,\}\;,
\]
which is obviously closed and of finite codimension in
$C^2(\Omega)\times T_x\Omega$.

\item In Lemma~\ref{l: convex}, the map $\bm$ is $C^m$ if the Banach space $C^{m+2}(\Omega)$ is used instead of $C^2(\Omega)$.

\end{enumerate}
\end{rem}

Suppose that the Riemannian manifold $M$
is connected and complete.  Let $(A,\mu)$ be a probability space, $B$ 
a convex open ball of radius $r>0$ in $M$, and
$f:A\to B$ a measurable map, which is called a \emph{mass
distribution} on $B$. Consider the $C^\infty$ function $P_f:B\to\R$ defined by
\[
P_f(x)=\frac{1}{2}\,\int_Ad(x,f(a))^2\,\mu(a)\;.
\]

\begin{prop}[{H.~Karcher  \cite[Theorem~1.2]{Karcher1977}}]\label{p: Karcher}
With the above notation and conditions, the following properties hold:
\begin{enumerate}[{\rm(}i{\rm)}]

\item\label{i: grad P_f} $\grad P_f$ is an outward pointing vector 
field on the boundary $\partial\ol B$.

\item\label{i: Hess P_f} If $\delta>0$ is an upper bound for the 
sectional curvatures of $M$ in $B$, and $2 r<\pi/2\sqrt{\delta}$, then 
$\Hess P_f$ is positive definite on $B$.

\end{enumerate}
\end{prop}

If the hypotheses of Proposition~\ref{p: Karcher} are satisfied, then 
$P_f\in\CC(\ol B)$, and therefore $P_f$ reaches its minimum on $\ol B$ 
at a unique point $\CC_f\in B$, which is called the \emph{center of 
mass} of $f$. It is known that $\CC_f$ depends continuously on $f$ 
with respect to the supremum distance when $(A,\mu)$ is fixed 
\cite[Corollary~1.6]{Karcher1977}; indeed, the following result 
follows directly from Lemma~\ref{l: convex}.

\begin{cor}\label{c:center of mass}
        \begin{enumerate}[{\rm(}i{\rm)}]

                \item $\CC_f$ depends continuously on $f$ and the metric tensor of $M$.

                \item If $A$ is the Borel $\sigma$-algebra of a metric space, then 
$\CC_f$ depends continuously on $\mu$ in the weak-$*$ topology.
\end{enumerate}
\end{cor}

\section{Foliated structure of $\MM_{*,\text{\rm
      lnp}}^\infty(n)$}\label{s: foliated structure}

The goal of this section is to prove 
Theorem~\ref{t: FF_*,lnp(n)}-\eqref{i: FF_*,lnp(n) is foliated structure}--\eqref{i: MM_*,np(n)}.

For any point $[M,x]\in\MM_{*,\text{\rm lnp}}^\infty(n)$, choose some
$r,\epsilon>0$ and some neighborhood $\NN_0$ of $[M,x]$ in
$\MM_{*,\text{\rm lnp}}^\infty(n)$ satisfying the statement of
Lemma~\ref{l: NN} with $\epsilon\le r/5$. Using \cite[Chapter~6, Theorem~3.6]{Petersen1998}, we can assume that
$\epsilon$ and $\NN_0$ are so small that $B_L(y,\epsilon)$ satisfies
the conditions of Proposition~\ref{p: Karcher} in $L$ for all
$[L,y]\in\NN_0$.  Take any continuous function
$\lambda:\MM_*^\infty(n)\to[0,1]$ supported in $\NN_0$ and with
$\lambda([M,x])=1$, whose existence is a simple consequence of the
metrizability of $\MM_*^\infty(n)$ (Theorem~\ref{t: C^infty
  convergence in MM_*(n)}). For $[L,y]\in\NN_0$, let $\omega_L$ denote
the Riemannian density of $L$, and let $\lambda_{L,y}:L\to[0,1]$ be the
function defined by
\[
\lambda_{L,y}(z)=
\begin{cases}
  \lambda([L,z]) & \text{if $d_L(y,z)\le\epsilon$}\\
  0 & \text{if $d_L(y,z)\ge\epsilon$}\;,
\end{cases}
\]
which is well defined and continuous by Lemma~\ref{l: NN}. Take
another neighborhood $\NN\subset\NN_0$ of $[M,x]$ where
$\lambda>0$. For $[L,y]\in\NN$, we have
$\int_L\lambda_{L,y}\,\omega_L>0$, and set
\[
\bar\lambda_{L,y}=\frac{\lambda_{L,y}}{\int_L\lambda_{L,y}\,\omega_L}\;.
\]
Then $\mu_{L,y}=\bar\lambda_{L,y}\,\omega_L$ is a continuous  density 
defining a probability measure on $L$, and the identity map
$(L,\mu_{L,y})\to L$ is a distribution of mass on $L$ satisfying the
conditions of Proposition~\ref{p: Karcher} with
$B_L(y,\epsilon)$. Thus its center of mass, $\CC_{L,y}$, is defined in
$B_L(y,\epsilon)$. Let $\bc:\NN\to\MM_*^\infty(n)$ be the map given by
$\bc([L,y])=[L,\CC_{L,y}]$.

\begin{lem}\label{l: d_L(y,y') le epsilon} 
  If $[L,y],[L,y']\in\NN$ and $d_L(y,y')\le\epsilon$, then
  $\bc([L,y])=\bc([L,y'])$.
\end{lem}

\begin{proof}
  Take any point $z\in L$. If $[L,z]\not\in\NN_0$ or
  $d_L(y,z),d_L(y',z)>\epsilon$, then
  $\lambda_{L,y}(z)=\lambda_{L,y'}(z)=0$. If $[L,z]\in\NN_0$ and
  $d_L(y,z)\le\epsilon$, then $d_L(y',z)\le2\epsilon$, obtaining
  $d_L(y',z)\le\epsilon$ by Lemma~\ref{l: NN} since $5\epsilon\le r$,
  and therefore
  $\lambda_{L,y}(z)=\lambda_{L,y'}(z)=\lambda([L,z])$. If
  $[L,z]\in\NN_0$ and $d_L(y',z)\le\epsilon$, we similarly get
  $\lambda_{L,y}(z)=\lambda_{L,y'}(z)$. Thus
  $\lambda_{L,y}=\lambda_{L,y'}$, obtaining $\CC_{L,y}=\CC_{L,y'}$,
  and therefore $\bc([L,y])=\bc([L,y'])$.
\end{proof}

\begin{lem}\label{l: bc is cont} 
	$\bc$ is continuous.
\end{lem}

\begin{proof}
  Take any convergent sequence $[L_i,y_i]\to[L,y]$ in $\NN$. Let
  $\Omega$ be a compact domain in $L$ whose interior contains $\ol
  B_L(y,\epsilon)$. Then there is a $C^\infty$ embedding
  $\phi_i:\Omega\to L_i$ for each $i$ large enough so that
  $\lim_i\phi_i^*g_i=g|_{\Omega}$ with respect to the $C^\infty$
  topology. It follows that
  $\lim_i\phi_i^*\mu_{L_i,y_i}=\mu_{L,y}|_{\Omega}$ with respect to
  the $C^0$ topology by the continuity of $\lambda$, and thus this
  convergence also holds in the space of probability measures on
  $\Omega$ with the weak-$*$ topology.  Since
  $\phi_i^{-1}(\CC_{L_i,y_i})$ is the center of mass of the mass
  distribution on $\Omega$ defined by the probability measure
  $\phi_i^*\mu_{L_i,y_i}$, it follows from Corollary~\ref{c:center of
    mass} that $\lim_i\phi_i^{-1}(\CC_{L_i,y_i})=\CC_{L,y}$ in
  $L$. Therefore $\lim_i\bc([L_i,y_i])=\bc([L,y])$ in
  $\MM_*^\infty(n)$ because $\Omega$ is arbitrary.
\end{proof}

Let $\ZZ=\bc(\NN)$, and let $\NN'=\bigcup_{[L,c]\in\ZZ}\iota_L(B_L(c,\epsilon))$, which
contains $\NN$ because $d_M(y,\CC_{L,y})<\epsilon$ for all
$[L,y]\in\NN$. Also, let $\bc':\NN'\to\ZZ$ be defined by the condition
$\bc'([L,z])=[L,c]$ if $[L,c]\in\ZZ$ and $d_L(c,z)<\epsilon$. To prove
that $\bc'$ is well defined, take another point $c'\in L$ satisfying
$[L,c']\in\ZZ$ and $d_L(c',z)<\epsilon$, and let us check that
$[L,c]=[L,c']$. Choose points $y,y'\in L$ such that
$[L,y],[L,y']\in\NN$, $\bc([L,y])=[L,c]$ and
$\bc([L,y'])=[L,c']$. Then
\[
d_L(y,y')\le d_L(y,c)+d_L(c,z)+d_L(z,c')+d_L(c',y')<4\epsilon\;,
\]
giving $d_L(y,y')\le\epsilon$ by Lemma~\ref{l: NN} since $5\epsilon\le
r$, which implies $[L,c]=[L,c']$ by Lemma~\ref{l: d_L(y,y') le
  epsilon}. Furthermore $\bc'$ is an extension of $\bc$ because
$d_L(y,\CC_{L,y})<\epsilon$ for all $[L,y]\in\NN$.  Note also that
$\bc'([L,c])=[L,c]$ for all $[L,c]\in\ZZ$.

\begin{lem}\label{l: d_L(z,z') le 2epsilon}
  If $[L,z],[L,z']\in\NN'$ and $d_L(z,z')\le2\epsilon$, then
  $\bc'([L,z])=\bc'([L,z'])$.
\end{lem}

\begin{proof}
  Let $\bc'([L,z])=[L,c]$ and $\bc'([L,z'])=[L,c']$. Choose points
  $[L,y],[L,y']\in\NN$ with $\bc([L,y])=[L,c]$ and
  $\bc([L,y'])=[L,c']$. Then
  \[
  d_L(y,y')\le
  d_L(y,c)+d_L(c,z)+d_L(z,z')+d_L(z',c')+d_L(c',y')<5\epsilon\;.
  \]
  From Lemma~\ref{l: NN} and since $5\epsilon\le r$, it follows that
  $[L,c]=[L,c']$.
\end{proof}

\begin{lem}\label{l: bc' is cont}
  $\bc'$ is continuous.
\end{lem}

\begin{proof}
  Take any convergent sequence $[L_i,z_i]\to[L,z]$ in $\NN'$. Let
  $\bc'([L_i,z_i])=[L_i,c_i]$ and $\bc'([L,z])=[L,c]$, and choose
  points $[L_i,y_i],[L,y]\in\NN$ so that $\bc([L_i,y_i])=[L_i,c_i]$
  and $\bc([L,y])=[L,c]$. We have
  \[
  d_i(y_i,z_i)\le d_i(y_i,c_i)+d_i(c_i,z_i)<2\epsilon\;,\quad
  d_L(y,z)\le d_L(y,c)+d_L(c,z)<2\epsilon\;.
  \]
  Then, by Lemma~\ref{l: [M i,y i]}, there are points $y'_i\in
  B_i(z_i,2\epsilon)$ such that $\lim_i[L_i,y'_i]=[L,y]$ in
  $\MM_*^\infty(n)$ as $i\to\infty$. Thus $[L_i,y'_i]\in\NN$ for $i$
  large enough, and moreover
  \[
  d_i(y_i,y'_i)\le d_i(y_i,z_i)+d_i(z_i,y'_i)<4\epsilon\;,
  \]
  obtaining $d_i(y_i,y'_i)\le\epsilon$ by Lemma~\ref{l: NN} since
  $5\epsilon\le r$. By Lemma~\ref{l: d_L(y,y') le epsilon}, it follows
  that $\bc([L_i,y'_i])=\bc([L_i,y_i])=[L_i,c_i]$ for $i$ large
  enough, giving $\lim_i[L_i,c_i]=[L,c]$ in $\MM_*^\infty(n)$ by
  Lemma~\ref{l: bc is cont}.
\end{proof}

We can assume that $\epsilon$ and $\NN$ are so small that the
following properties hold for all $[L,y]\in\NN$ and $z\in
B_L(y,\epsilon)$:
\begin{enumerate}[(a)]

\item\label{i: exp} $\exp_L:B_{T_zL}(0_z,\epsilon)\to B_L(z,\epsilon)$
  is a diffeomorphism; and

\item\label{i: id_L} $\left\{h\in\Iso(L)\ |\ h(z)\in\ol
    B(z,4\epsilon)\right\}=\{\id_L\}$.

\end{enumerate}
Observe that~\eqref{i: id_L} can be assumed by Lemma~\ref{l: LL}.
Notice also that~\eqref{i: exp} and~\eqref{i: id_L} hold for all
$[L,z]\in\ZZ$.  Let
\[
\widehat{\NN}'=\{\,[L,\xi]\in\TT_*^\infty(n)\mid\pi([L,\xi])\in\ZZ,\
|\xi|<\epsilon\,\}\;.
\]

\begin{lem}\label{l: exp}
	$\exp:\widehat{\NN}'\to\NN'$ is a homeomorphism.
\end{lem}

\begin{proof}
  This map is obviously surjective; we will prove that it also
  injective. For $i\in\{1,2\}$, take points
  $[L_i,\xi_i]\in\widehat{\NN}'$; thus $\xi_i\in T_{c_i}L_i$ for some
  points $[L_i,c_i]\in\ZZ$, and we have $\exp([L_i,\xi_i])=[L_i,z_i]$
  for $z_i=\exp_i(\xi_i)$. Suppose that $[L_1,z_1]=[L_2,z_2]$, which
  means that there is a pointed isometry
  $\phi:(L_1,z_1)\to(L_2,z_2)$. Then
  \begin{gather}
    \exp_2\,\phi_*(\xi_1)=\phi\,\exp_1(\xi_1)=\phi(z_1)=z_2=\exp_2(\xi_2)\;,
    \label{exp_2 phi*(xi1) = exp_2(xi2)}\\
    d_2(\phi(c_1),c_2)\le
    d_2(\phi(c_1),z_2)+d_2(z_2,c_2)=d_1(c_1,z_1)+d_2(z_2,c_2)<2\epsilon\;.
    \label{d_2(phi(c_1),c_2) < 2epsilon}
  \end{gather}
  We get
  \[
  [L_1,c_1]=\bc'([L_1,c_1])=\bc'([L_2,\phi(c_1)])=[L_2,c_2]
  \]
  by Lemma~\ref{l: d_L(z,z') le 2epsilon} and~\eqref{d_2(phi(c_1),c_2)
    < 2epsilon}. So there is an isometry $\psi:L_1\to L_2$ such that
  $\psi(c_1)=c_2$. Then the isometry $h=\psi^{-1}\phi:L_1\to L_1$
  satisfies
		\[
			d_1(c_1,h(c_1))=d_2(c_2,\phi(c_1))<2\epsilon
		\]
	by~\eqref{d_2(phi(c_1),c_2) < 2epsilon}, obtaining $h=\id_{L_1}$ by~\eqref{i: exp}. Hence $\phi(c_1)=\psi(c_1)=c_2$, giving $\phi_*(\xi_1)=\xi_2$ by~\eqref{exp_2 phi*(xi1) = exp_2(xi2)} and~\eqref{i: exp} since $\xi_i\in T_{c_i}L_i$. Therefore $\exp:\widehat{\NN}'\to\NN'$ is bijective. 
	
	The continuity of $\exp^{-1}:\NN'\to\widehat{\NN}'$ is a
        simple exercise using lemma~\ref{l: bc' is cont}.
\end{proof}

By Proposition~\ref{p: TT}-\eqref{i: TT}, there is some neighborhood
$\Theta$ of $[M,x]$ in $\MM_*^\infty(n)$ and some local trivialization
$\theta:\pi^{-1}(\Theta)\to\R^n\times\Theta$ of the Riemannian vector
bundle $\pi:\TT_*^\infty(n)\to\MM_*^\infty(n)$; in particular,
$\theta:\pi^{-1}([L,y])\to\R^n\times\{[L,y]\}\equiv\R^n$ is a linear
isometry for all $[L,y]\in\Theta$. More precisely, according to the
proof of Proposition~\ref{p: TT}, we can suppose that there is a local
section $\sigma:\Theta\to\QQ_*^\infty(n)$ of
$\pi:\QQ_*^\infty(n)\to\MM_*^\infty(n)$ so that
$\theta([L,\xi])=(\sigma([L,y])([L,\xi]),[L,y])$ if
$\pi_L(\xi)=y$. We can assume that $\ZZ\subset \Theta$ by
Lemma~\ref{l: delta}. Hence, by Lemma~\ref{l: exp}, the composite
\[
\begin{CD}
  \NN' @>{\exp^{-1}}>> \widehat{\NN}' @>{\theta}>>
  B^n_\epsilon\times\ZZ
\end{CD}
\]
is a homeomorphism $\Phi:\NN'\to B^n_\epsilon\times\ZZ$, where
$B^n_\epsilon$ denotes the open ball of radius $\epsilon$ centered at
the origin in $\R^n$. This shows that $\FF_{*,\text{\rm lnp}}(n)$ is a
foliated structure of dimension $n$ on $\MM_{*,\text{\rm lnp}}^\infty(n)$, completing the proof of 
Theorem~\ref{t: FF_*,lnp(n)}-\eqref{i: FF_*,lnp(n) is foliated structure}.

Recall that a Riemannian manifold $M$ (or its metric tensor) is called
\emph{nowhere locally homogenous} if there is no isometry between
distinct open subsets of $M$. It is easy to see that the proof of
\cite[Proposition~1]{Sunada1985} can be adapted to the case of open
manifolds, obtaining the following.

\begin{prop}\label{p: nowhere locally homogenous metrics}
  For any $C^\infty$ manifold $M$, the set of nowhere locally
  homogenous metrics on $M$ is residual in $\Met(M)$ with the weak and
  strong $C^\infty$ topologies.
\end{prop}

\begin{lem}\label{l: iota(M) is dense}
  There is a nowhere locally homogenous complete Riemannian manifold
  $M$ such that $\iota(M)$ is dense in $\MM_{*,\text{\rm o}}^\infty(n)$.
\end{lem}

\begin{proof}
  According to the proof of Proposition~\ref{p: MM_*^infty(n) is
    separable}, there is a countable dense set of points $[M_i,x_i]$
  in $\MM_{*,\text{\rm lnp,c}}^\infty(n)$ ($i\in\N$). For each $i$, take some $y_i\in M_i$ so that
  $d_i(x_i,y_i)=\max_{y\in M_i}d_i(x_i,y)$. For all $i\in\N$ and $j,k\in\Z^+$ with
  $1/j,1/k<\diam M_i$, let $(M_{ijk},x_{ijk},y_{ijk})$ be a copy of
  $(M_i,x_i,y_i)$, let $g_{ijk}$ be the metric of $M_{ijk}$, and let $\Omega_{ijk}$ be a compact domain in
  $M_{ijk}$ containing $y_{ijk}$ and with diameter $<1/j$. Observe
  that $\widehat\Omega_{ijk}:=M_{ijk}\sm\Int(\Omega_{ijk})$ is also a
  compact domain. Take also corresponding mutually disjoint compact
  domains $\Omega'_{ijk}$ in $\R^n$ so that every bounded subset of
  $\R^n$ only meets a finite number of them. Let $M$ be the $C^\infty$
  connected sum of $\R^n$ with all manifolds $M_{ijk}$ so that the
  connected sum with each $M_{ijk}$ only involves perturbations inside the interiors of
  $\Omega_{ijk}$ and $\Omega'_{ijk}$. Let $g$ be any Riemannian metric
  on $M$ whose restriction to each $\widehat\Omega_{ijk}$ equals
  $g_{ijk}$, and whose restriction to
  $\R^n\sm\bigcup_{ijk}\Omega'_{ijk}$ equals the Euclidean
  metric. Then $g$ is complete and $\iota(M,g)$ is dense in
  $\MM_*^\infty(n)$. With the strong $C^\infty$ topology,
  $C^\infty(M;TM^*\odot TM^*)$ is a Baire space by
  \cite[Theorem~4.4-(b)]{Hirsch1976}. Since $\Met(M)$ is open in
  $C^\infty(M;TM^*\odot TM^*)$, and the complete metrics on $M$ form
  an open subspace $\Met_{\text{\rm com}}(M)\subset\Met(M)$, it
  follows that $\Met_{\text{\rm com}}(M)$ is a Baire space with the
  strong $C^\infty$ topology. Hence, by Proposition~\ref{p: nowhere
    locally homogenous metrics}, there is a nowhere locally homogenous
  complete metric $g'$ on $M$ so that
  $\|g-g'\|_{C^k,\widehat\Omega_{ijk},g}<1/k$ for all $i$, $j$ and
  $k$. Then $\iota(M,g')$ is also dense in $\MM_{*,\text{\rm o}}^\infty(n)$.
\end{proof}

By Lemma~\ref{l: iota(M) is dense}, $\FF_{*,\text{\rm lnp,o}}(n)$ is
transitive, showing Theorem~\ref{t: FF_*,lnp(n)}-\eqref{i: FF_*,lnp,o(n) is transitive}.

Now, for $k\in\{1,2\}$, let $\Phi_k:\NN'_k\to
B^n_{\epsilon_k}\times\ZZ_k$ be two homeomorphisms constructed
as above with maps $\bc'_k:\NN'_k\to\ZZ_k$,
$\exp:\widehat{\NN}'_k\to\NN'_k$ and
$\sigma_k:\Theta_k\to\QQ_*^\infty(n)$.

\begin{lem}\label{l: change of coordinates} 
  $\Phi_2\Phi_1^{-1}:\Phi_1(\NN'_1\cap\NN'_2)\to\Phi_1(\NN'_1\cap\NN'_2)$
  is $C^\infty$ {\rm(}in the sense of Section~\ref{ss: fol sps}{\rm)}.
\end{lem}

\begin{proof}
  This map has the expression
  \[
  \Phi_2\Phi_1^{-1}(v,[L,c])=(\Psi(v,[L,c]),\Gamma([L,c]))\;,
  \]
  where $\Gamma:\bc'_1(\NN'_1\cap\NN'_2)\to\bc'_2(\NN'_1\cap\NN'_2)$
  is the corresponding holonomy transformation, and
  $\Psi:\Phi_1(\NN'_1\cap\NN'_2)\to\R^n$ is defined by
  \[
  \Psi(v,[L,c])=
  \sigma_2([L,c'])\,\exp_{[L,c']}^{-1}\,\exp_{[L,c]}\,\sigma_1([L,c])^{-1}(v)\;,
  \]
  where $[L,c']=\Gamma([L,c])$. Let $[L,f]=\sigma_1([L,c])$ and
  $[L,f']=\sigma_2([L,c'])$.  We can take $c'$ so that
  $d(c,c')<\epsilon_1+\epsilon_2$, and then
  \[
  \Psi(v,[L,c])=f'\,\exp_{c'}^{-1}\,\exp_c\,f^{-1}(v)\;.
  \]

	To prove that $\Psi$ is $C^\infty$ in the sense of
Section~\ref{ss: fol sps}, fix any
$(v,[L,c])\in\Phi_1(\NN'_1\cap\NN'_2)$, and take $c'$, $f$ and $f'$ as
above. Let $V$ and $\OO$ be open neighborhoods of $v$ and $[L,c]$ in
$\R^n$ and $\ZZ_1$, respectively, such that
$\ol{V}\times\OO\subset\Phi_1(\NN'_1\cap\NN'_2)$. Take any convergent
sequence $[L_i,c_i]\to[L,c]$ in $\OO$, and define $c'_i$, $f_i$ and
$f'_i$ as before for each $i$.  Notice that $\Psi(v,[L,c])$ and
$\Psi(v,[L_i,c_i])$ are defined for all $v\in V$, and let
$\psi,\psi_i:V\to\R^n$ be the $C^\infty$ maps given by
$\psi(v)=\Psi(v,[L,c])$ and $\psi_i(v)=\Psi(v,[L_i,c_i])$. We have to
prove that $\lim_i\psi_i=\psi$ with respect to the weak $C^\infty$
topology.

	Let $\Omega$ be any compact domain in $L$ such that $\ol
B_L(c,\epsilon_1+2\epsilon_2)\subset\Int(\Omega)$, and thus $\ol
B_L(c',\epsilon_2)\subset\Int(\Omega)$ too. Since the sections
$\sigma_1$ and $\sigma_2$ are continuous, there are $C^\infty$
embeddings $\phi_i:\Omega\to L_i$ for $i$ large enough so that
$\phi_{i*}(f)=f_i$ and $\lim_i\phi_i^*g_i=g|_\Omega$; in particular,
$\phi_i(c)=c_i$. Hence $c'_i\in\phi_i(\Int(\Omega))$ for $i$ large
enough, and moreover $\lim_i\phi_{i*}^{-1}(f'_i)=f'$ by~\eqref{i:
id_L} and Lemma~\ref{l: double convergence}.  Observe that
$\hat\psi:=\exp_{c'}^{-1}\,\exp_c$ is defined on $W=f^{-1}(V)\subset
B_{T_cL}(0_c,\epsilon_1)$. It follows that
$\hat\psi_i:=\phi_{i*}^{-1}\,\exp_{c'_i}^{-1}\,\exp_{c_i}\,\phi_{i*}$
is also defined on $W$ for $i$ large enough, and moreover
$\lim_i\hat\psi_i=\hat\psi$ in the space of $C^\infty$ maps $W\to
T_{c'}L$ with the weak $C^\infty$ topology. So
		\[
			\lim_i\phi_{i*}^{-1}(f'_i)\,\hat\psi_if^{-1}=f'\hat\psi f^{-1}=\psi
		\] 
	in the space of $C^\infty$ maps $V\to\R^n$ with the
weak $C^\infty$ topology. Then the result follows because
		\[ 
			\phi_{i*}^{-1}(f'_i)\,\hat\psi_if^{-1}
			=\phi_{i*}^{-1}(f'_i)\,\hat\psi_i\,(\phi_{i*}^{-1}(f_i))^{-1}
			=f'_i\phi_{i*}\hat\psi_i\phi_{i*}^{-1}f_i^{-1}
			=f'_i\,\exp_{c_i'}^{-1}\,\exp_{c_i}\,f_i^{-1}=\psi_i\;.\qed
		\] 
\renewcommand{\qed}{}
\end{proof}

According to Lemma~\ref{l: change of coordinates}, $\FF_{*,\text{\rm
    lnp}}(n)$ becomes $C^\infty$ with the above kind of charts. Thus
we can consider the tangent bundle $T\FF_{*,\text{\rm lnp}}(n)$. For
each leaf $\iota(M)$ of $\FF_{*,\text{\rm lnp}}(n)$, the canonical
homeomorphism $\bar\iota:\Iso(M)\backslash M\to\iota(M)$ is a
$C^\infty$ diffeomorphism, and $\iota_{*x}:T_xM\to
T_{[M,x]}\FF_{*,\text{\rm lnp}}(n)$ is an isomorphism for each $x\in
M$. According to Proposition~\ref{p: TT}, we get a canonical bijection
$T\FF_{*,\text{\rm lnp}}(n)\to\TT_{*,\text{\rm lnp}}^\infty(n)$
defined by $\iota_{*x}(\xi)\mapsto[M,\xi]$ for
$[M,\xi]\in\MM_{*,\text{\rm lnp}}^\infty(n)$ and $\xi\in T_xM$. It is
an easy exercise to prove that this bijection is an isomorphism of
vector bundles. So the Riemannian structure on $\TT_{*,\text{\rm
    lnp}}^\infty(n)$ defined in Proposition~\ref{p: TT} corresponds to
a Riemannian structure on $T\FF_{*,\text{\rm lnp}}(n)$, which can be
easily proved to be $C^\infty$ by using the above kind of flow boxes
of $\FF_{*,\text{\rm lnp}}(n)$. It is elementary that each isomorphism
$\iota_{*x}:T_xM\to T_{[M,x]}\FF_{*,\text{\rm lnp}}(n)$ is an
isometry. This completes the proof of 
Theorem~\ref{t: FF_*,lnp(n)}-\eqref{i: FF_*,lnp(n) is C^infty and Riemannian}.

Theorem~\ref{t: FF_*,lnp(n)}-\eqref{i: MM_*,np(n)} follows from the following.

\begin{lem}\label{l: holonomy} 
	The following properties hold for any point $[M,x]\in\MM_{*,\text{\rm lnp}}^\infty(n)$, any path $\alpha:I:=[0,1]\to M$ with $\alpha(0)=x$, and any neighborhood $\UU$ of $\iota\alpha$ in $C(I,\FF_{*,\text{\rm lnp}}(n))$:
	\begin{enumerate}[{\rm(}i{\rm)}]

		\item\label{i: alpha(1)=x} If $\alpha(1)=x$ then, for each $[N,y]\in\MM_{*,\text{\rm lnp}}^\infty(n)$ close enough to $[M,x]$, there is a path $\beta\in\UU$ with $\beta(0)=\beta(1)=[N,y]$.

		\item\label{i: alpha(1) ne x} If $\alpha(1)\ne x$ then there is some path $\beta\in\UU$ with $\beta(0)\ne\beta(1)$.

	\end{enumerate}
\end{lem}

\begin{proof}
	Let $\Omega$ be a compact domain in $M$ whose interior contains $\alpha(I)$, let $[N,y]\in\MM_{*,\text{\rm lnp}}^\infty(n)$, and let $\phi:(\Omega,x)\to(N,y)$ be a pointed $C^m$ embedding with $\|g_M-\phi^*g_N\|_{\Omega,C^m,g_M}<\epsilon$ for some $m\in\Z^+$ and $\epsilon>0$. Let $\beta=\iota\phi\alpha\in C(I,\FF_{*,\text{\rm lnp}}(n))$; that is, $\beta(t)=[N,\phi\alpha(t)]$ for each $t\in I$. Observe that $\beta\in\UU$ if $m$ and $\Omega$ are large enough, and $\epsilon$ is small enough (i.e., if $[N,y]$ is close enough to $[M,x]$). When $\alpha(1)=x$, we get
		\[
			\beta(0)=[N,\phi(x)]=[N,y]=[N,\phi\alpha(1)]=\beta(1)\;.
		\]

	Suppose now that $\alpha(1)\ne x$. Since $\MM_{*,\text{\rm np}}^\infty(n)$ is dense in $\MM_{*,\text{\rm lnp}}^\infty(n)$, with the above notation, we can choose $[N,y]\in\MM_{*,\text{\rm np}}^\infty(n)$ as close as desired to $[M,x]$. Hence $\iota:N\to\MM_{*,\text{\rm lnp}}^\infty(n)$ is injective, giving
		\[
			\beta(0)=\iota\phi(x)\ne\iota\phi\alpha(1)=\beta(1)\;.\qed
		\]
\renewcommand{\qed}{}
\end{proof}

\section{Saturated subspaces of $\MM_{*,\text{\rm lnp}}^\infty(n)$}\label{s: saturated subspaces}

Let $X$ be a sequential Riemannian foliated space with complete leaves. 

\begin{defn}\label{d: covering-determined} 
  It is said that $X$ is \emph{covering-determined} when there is a
  connected pointed covering $(\widetilde L_x,\tilde x)$ of $(L_x,x)$
  for all $x\in X$ such that $x_i\to x$ in $X$ if and only if
  $[\widetilde L_{x_i},\tilde x_i]$ is $C^\infty$ convergent to
  $[\widetilde L_x,\tilde x]$. When this condition is satisfied with
  $\widetilde L_x=\widetilde L_x^{\text{\rm hol}}$ for all $x\in X$, it is said that 
  $X$ is \emph{holonomy-determined}.
\end{defn}

\begin{ex}\label{ex: covering-determined}
  \begin{enumerate}[(i)]
			
  \item The Reeb foliation on $S^3$ is not covering-determined with
    any Riemannian metric. 
			
  \item \cite[Example~2.5]{Lessa:Reeb} is covering-determined but not
    holonomy-determined.
		
  \item\label{i: holonomy-determined} $\MM_{*,\text{\rm
        lnp}}^\infty(n)$ is holonomy-determined.
			
	\end{enumerate}
\end{ex}

\begin{rem}\label{r: covering-determined}
  \begin{enumerate}[(i)]
	
  \item\label{i: hereditary} The condition of being
    covering-determined is hereditary by saturated subspaces.
		
  \item The example $X$ of \cite[Example~2.5]{Lessa:Reeb} can be easily realized as a
    saturated subspace of a Riemannian foliated space $Y$ where the holonomy coverings of the leaves are isometric to $\R$. Multiplying the leaves by $S^1$, all holonomy covers of $Y\times S^1$ become isometric to $\R\times S^1$. The metric on $Y\times S^1$ can be modified so that no pair of these holonomy covers are isometric, obtaining a holonomy-determined foliated space, however $X\times S^1$ is not holonomy-determined with any metric. So holonomy-determination is not hereditary by saturated subspaces.
		
  \item\label{i: x=y} If $X$ satisfies the covering-determination with
    the pointed coverings $(\widetilde L_x,\tilde x)$ of $(L_x,x)$ for
    $x\in X$, then $x=y$ in $X$ if and only if $[\widetilde L_x,\tilde
    x]=[\widetilde L_y,\tilde y]$; in particular, the leaves of $X$
    are non-periodic.
		
  \item If $X$ is compact and the mapping $x\mapsto[\widetilde L_x,\tilde x]$ is injective, then the ``if'' part of Definition~\ref{d: covering-determined} can be deleted.

	\end{enumerate}
\end{rem}

\begin{proof}[Proof of Theorem~\ref{t: covering-determined}]
  Any saturated subspace of $\MM_{*,\text{\rm lnp}}^\infty(n)$ is
  covering-determined by Example~\ref{ex:
    covering-determined}-\eqref{i: holonomy-determined} and
  Remark~\ref{r: covering-determined}-\eqref{i: hereditary}.
	 
  Suppose that $X$ satisfies the covering-determination with the pointed
  covers $(\widetilde L_x,\tilde x)$ of $(L_x,x)$ for $x\in X$. Then
  the map $\iota:X\to\MM_{*,\text{\rm lnp}}^\infty(n)$, defined by
  $\iota(x)=[\widetilde L_x,\tilde x]$, is a $C^\infty$ foliated
  embedding whose restrictions to the leaves are isometries.
\end{proof}

\begin{rem}
  Like in the above proof, a map $\iota^{\text{\rm
      hol}}:X\to\MM_*^\infty(n)$ is defined by $\iota^{\text{\rm
      hol}}(x)=[\widetilde L_x^{\text{\rm hol}},\tilde x]$, where
  $\tilde x\in\widetilde L_x^{\text{\rm hol}}$ is over $x$. This map
  may not be continuous \cite[Example~2.5]{Lessa:Reeb}, but its
  restriction to $X_0$ is continuous by the local Reeb stability
  theorem, and therefore $\iota^{\text{\rm hol}}$ is Baire measurable
  if $X$ is second countable.
\end{rem}

Any family $\CC$ of complete connected Riemannian $n$-manifolds
defines a closed $\FF_*(n)$-saturated subspace
$X:=\Cl_\infty(\bigcup_{M\in\CC}\iota(M))\subset\MM_*^\infty(n)$. The
obvious $C^\infty$ version of arguments of \cite{Cheeger1970} (see
also \cite[Chapter~10, Sections~3 and~4]{Petersen1998}) gives the
following.

\begin{thm}\label{t: bounded geometry} 
	A family $\CC$ of complete connected Riemannian $n$-manifolds is of equi-bounded geometry if and only if the closed subspace of $\MM_*^\infty(n)$ defined by $\CC$ is compact.
\end{thm}

\begin{rem}
	A version of Theorem~\ref{t: bounded geometry} using the Ricci curvature instead of $\RR$ can be also proved with the arguments of \cite{Anderson1990}.
\end{rem}

For instance, let $\MM_*(n,r,C_m)\subset\MM_*(n)$ denote the subspace defined by the manifolds of bounded geometry with geometric bound $(r,C_m)$. Each $\MM_*(n,r,C_m)$ is compact by Theorem~\ref{t: bounded geometry}, and the notion of $C^\infty$ convergence in $\MM_*(n,r,C_m)$ is equivalent to the convergence in the topology of the Gromov space $\MM_*$ \cite{Lessa:Reeb}, \cite[Chapter~10]{Petersen1998}. Nonetheless, this is not the case on the whole of $\MM_*(n)$ \cite[Section~7.1.4]{BuragoBuragoIvanov2001}. 

Let us study the case of closed subspaces of $\MM_*^\infty(n)$ defined by a single manifold.

\begin{defn}\label{d: aperiodic} 
	A complete connected Riemannian manifold $M$ is called:
		\begin{enumerate}[(i)]
	
			\item\label{i: aperiodic} \emph{aperiodic} if, for all $m_i\uparrow\infty$ in $\N$, compact domains $\Omega'_i\subset\Omega_i\subset M$, points $x_i\in\Omega'_i$ and $y_i\in\Omega_i$, and $C^{m_i}$ pointed embeddings $\phi_{ij}:(\Omega_i,x_i)\to(\Omega_j,x_j)$ ($i\le j$) and $\psi_i:(\Omega'_i,x_i)\to(\Omega_i,y_i)$ such that
				\[
					\lim_id(x_i,\partial\Omega'_i))=\infty\;,\quad
					\lim_{i,j}\|g-\phi_{ij}^*g\|_{C^{m_i},\Omega_i,g}
					=\lim_i\|g-\psi_i^*g\|_{C^{m_i},\Omega'_i,g}=0\;,
				\]
			we have
				\begin{equation}\label{aperiodic}
					\lim_i\max\{\,d(x,\psi_i(x))\mid x\in\Omega'_i\cap\ol B(x_i,r)\,\}=0
				\end{equation}
			for some $r>0$; and
	
			\item\label{i: weakly aperiodic} \emph{weakly aperiodic} if, to get~\eqref{aperiodic}, besides the conditions of~\eqref{i: aperiodic}, it is also required that there is some $s>0$ and there are points $z_i\in\Omega'_i$ such that $\phi_{ij}(z_i)=z_j$ and $d(z_i,\psi_i(z_i))<s$.
			
		\end{enumerate}
\end{defn}

\begin{lem}\label{l: aperiodic} 
	The following properties hold for any complete connected Riemannian $n$-manifold $M$:
		\begin{enumerate}[{\rm(}i{\rm)}]

			\item\label{i: aperiodic} $M$ is aperiodic if and only if $\Cl_\infty(\iota(M))\subset\MM_{*,\text{\rm np}}^\infty(n)$.

			\item\label{i: weakly aperiodic} $M$ is weakly aperiodic if and only if $\Cl_\infty(\iota(M))\subset\MM_{*,\text{\rm lnp}}^\infty(n)$.

		\end{enumerate}
\end{lem}

\begin{proof}
	This is a consequence of Propositions~\ref{p: C^infty uniformity in MM_*(n)},~\ref{p: D^m_R,epsilon(M,x) subset U^m_R,r(M,x)} and~\ref{p: U^m_R,epsilon(M,x) subset D^m_R,r(M,x)}, and using also arguments from the proof of Proposition~\ref{p: Hausdorff MM_*(n)} for the ``if'' parts.
\end{proof}

\begin{defn}\label{d: repetitive} 
	A complete connected Riemannian manifold $M$ is called \emph{repetitive} if, for every compact domain $\Omega$ in $M$, and all $\epsilon>0$ and $m\in\N$, there is a family of $C^m$ embeddings $\phi_i:\Omega\to M$ such that $\bigcup_i\phi_i(\Omega)$ is a net in $M$ and $\|g-\phi_i^*g\|_{C^m,\Omega,g}<\epsilon$ for all $i$. 
\end{defn}

Here, the term \emph{net} in $M$ is used for a subset $A\subset M$ satisfying $\Pen(A,S)=M$ for some $S>0$.

\begin{lem}\label{l: repetitive} 
	Let $M$ be a complete connected Riemannian $n$-manifold of bounded geometry. Then $M$ is repetitive if and only if $\Cl_\infty(\iota(M))$ is $\FF_*(n)$-minimal.
\end{lem}

\begin{proof}
	The ``only if'' part follows easily from Propositions~\ref{p: C^infty uniformity in MM_*(n)},~\ref{p: D^m_R,epsilon(M,x) subset U^m_R,r(M,x)} and~\ref{p: U^m_R,epsilon(M,x) subset D^m_R,r(M,x)}. 
	
	To prove the ``if'' part, assume that $\Cl_\infty(\iota(M))$ is $\FF_*(n)$-minimal. Let $\Omega$ be a compact domain in $M$, and take some $m\in\N$ and $\epsilon>0$. Take some $x\in M$ and $R>0$ such that $\Omega\subset B(x,R)$. Let $U=\Int_\infty(D^m_{R,\epsilon}(M,x))$. Since $\Cl_\infty(\iota(M))$ is compact because $M$ is of bounded geometry (Theorem~\ref{t: bounded geometry}), there is some $S>0$ such that $\bd_U\le S$ on $\Cl_\infty(\iota(M))$ by Lemma~\ref{l: bd_U}. Hence $[M,x_i]\in U$ for a net of points $x_i$ in $M$. Thus  there are $C^{m+1}$ pointed local diffeomorphisms $\phi_i\colon (M,x) \rightarrowtail (M,x_i)$ so that $\|g-\phi_i^*g\|_{C^m,\Omega_i,g}<\epsilon$ for some compact domain $\Omega_i\subset\dom\phi_i$ with $B(x,R)\subset\Omega_i$; in particular, $\Omega\subset\dom\phi_i$ and $\|g-\phi_i^*g\|_{C^m,\Omega,g}<\epsilon$ for all $i$, and $\bigcup_i\phi_i(\Omega)$ is a net in $M$, showing that $M$ is repetitive.
\end{proof}

\begin{proof}[Proof of Theorem~\ref{t: M is isometric to a dense leaf}]
	Suppose that $M$ is non-periodic and has a weakly aperiodic connected covering $\widetilde M$. Then $Y=\Cl_\infty(\iota(\widetilde M))$ is a compact saturated subspace of $\MM_{*,\text{\rm lnp}}^\infty(n)$ by Theorem~\ref{t: bounded geometry} and Lemma~\ref{l: aperiodic}-\eqref{i: weakly aperiodic}, and $M\equiv\Iso(\widetilde M)\backslash\widetilde M\overset{\bar\iota}{\longrightarrow}\iota(\widetilde M)$ is an isometry. Moreover any sequential covering-determined transitive compact Riemannian foliated space can be obtained in this way by Theorem~\ref{t: covering-determined}. If $\widetilde M$ is also repetitive, then $X$ is minimal by Lemma~\ref{l: repetitive}, completing the proof of~\eqref{i: M is isometric to a dense leaf}.
	
	Asume now that $M$ is aperiodic. Then $X=\Cl_\infty(\iota(M))$ is a compact $\FF_{*,\text{\rm np}}(n)$-saturated subspace of $\MM_{*,\text{\rm np}}^\infty(n)$ by Theorem~\ref{t: bounded geometry} and Lemma~\ref{l: aperiodic}-\eqref{i: aperiodic}, and moreover $\iota:M\to\iota(M)$ is an isometry. Furthermore the leaves of $X$ have trivial holonomy groups by Theorem~\ref{t: FF_*,lnp(n)}-\eqref{i: MM_*,np(n)}. As before, $X$ is minimal if $M$ is also repetitive, showing~\eqref{i: if M is aperiodic}.	
\end{proof}

\section*{Acknowledgements}

The first and third authors are partially supported by MICINN (Spain), grant MTM2011-25656.

\bibliographystyle{amsplain}


\providecommand{\bysame}{\leavevmode\hbox to3em{\hrulefill}\thinspace}
\providecommand{\MR}{\relax\ifhmode\unskip\space\fi MR }
\providecommand{\MRhref}[2]{%
  \href{http://www.ams.org/mathscinet-getitem?mr=#1}{#2}
}
\providecommand{\href}[2]{#2}

\end{document}